\documentclass[11pt]{article}
\usepackage{amsfonts}
\usepackage{latexsym}
\usepackage{amsmath}
\usepackage{amssymb}
\usepackage{amsthm}
\usepackage{amsmath}
\usepackage{verbatim}
\usepackage{hyperref}
\usepackage{enumerate}
\usepackage{color}
\usepackage{bibspacing}
\usepackage{float}
\usepackage{pgfplots}
\pgfplotsset{compat=1.15}
\usepgfplotslibrary{fillbetween}
\usetikzlibrary{patterns}

\newtheorem{thm}{Theorem}[section]
\newtheorem{theorem}[thm]{Theorem}
\newtheorem{problem}[thm]{Problem}
\newtheorem*{thm*}{Theorem}
\newtheorem*{question*}{Question}
\newtheorem*{problem*}{Problem}
\newtheorem{cor}[thm]{Corollary}
\newtheorem{corollary}[thm]{Corollary}

\newtheorem{lem}[thm]{Lemma}
\newtheorem{lemma}[thm]{Lemma}

\newtheorem*{prop*}{Proposition}
\newtheorem{proposition}[thm]{Proposition}

\newtheorem*{conj*}{Conjecture}
\newtheorem{definition}[thm]{Definition}
\newtheorem*{dfn*}{Definition}
\theoremstyle{definition}
\newtheorem{rem}[thm]{\textbf{Remark}}
\newtheorem{remark}[thm]{\textbf{Remark}}
\newtheorem*{rmk*}{Remark}
\newtheorem*{fact*}{Fact}

\theoremstyle{proof}

\newcommand{\bary}{\mathrm{bar}}
\newcommand{\Sym}{\mathrm{Sym}}
\newcommand{\Cov}{\mathrm{Cov}}

\newcommand{\maxo}{\max\phantom{}_0}

\newcommand{\FR}{\mathrm{FR}}
\newcommand{\G}{\mathcal{G}}
\newcommand{\F}{\mathcal{F}}

\newcommand{\A}{\mathbf{A}}
\newcommand{\B}{\mathbf{B}}
\newcommand{\C}{\mathcal{C}}
\newcommand{\Q}{\mathcal{Q}}
\renewcommand{\c}{\mathbf{c}}
\renewcommand{\d}{\mathbf{d}}
\renewcommand{\L}{\mathbf{L}}

\newcommand{\1}{\textbf{1}}
\newcommand{\sgn}{\text{sgn}}

\newcommand{\norm}[1]{\left\Vert#1\right\Vert}
\newcommand{\snorm}[1]{\Vert#1\Vert}
\newcommand{\abs}[1]{\left\vert#1\right\vert}
\newcommand{\set}[1]{\left\{#1\right\}}
\newcommand{\brac}[1]{\left(#1\right)}
\newcommand{\scalar}[1]{\left \langle #1 \right \rangle}
\newcommand{\sscalar}[1]{\langle #1 \rangle}

\newcommand{\R}{\mathbb{R}}

\newcommand{\E}{\mathbb{E}}

\newcommand{\eps}{\varepsilon}

\newcommand{\Id}{\mathrm{Id}}
\newcommand{\tr}{\mathrm{tr}}

\DeclareMathOperator{\conv}{\textnormal{conv}}

\numberwithin{equation}{section}
\numberwithin{figure}{section}

\begin{document}

\renewcommand*{\thefootnote}{\fnsymbol{footnote}}

\author{Emanuel Milman\textsuperscript{$*$}, Shohei Nakamura\textsuperscript{$\dagger$} and Hiroshi Tsuji\textsuperscript{$\ddagger$}}
\footnotetext{$^*$Department of Mathematics, Technion-Israel Institute of Technology, Haifa 32000, Israel. Email: emilman@tx.technion.ac.il.}
\footnotetext{$^\dagger$School of Mathematics, The Watson Building, University of Birmingham, Edgbaston, Birmingham, B15 2TT, England. Email: s.nakamura@bham.ac.uk.}
\footnotetext{$^\ddagger$Department of Mathematics, Institute of Science Tokyo, 2-12-1 Ookayama, Meguro-ku, Tokyo 152-8551, Japan. Email: tsujihiroshi@math.sci.isct.ac.jp.}

\begingroup    \renewcommand{\thefootnote}{}    
    \footnotetext{The research leading to these results is part of a project that has received funding from the European Research Council (ERC) under the European Union's Horizon 2020 research and innovation programme (grant agreement No 101001677), and is also supported by JSPS Kakenhi grant number 24KJ0030 (Tsuji).}
\endgroup

\title{The Gaussian Conjugate Rogers--Shephard Inequality}

\date{}

\maketitle

\begin{abstract}
We fuse between the Rogers--Shephard inequality for the Lebesgue measure and Royen's Gaussian Correlation Inequality, simultaneously extending both into a single sharp inequality for the Gaussian measure $\gamma$ on $\R^n$, stating that
\[
\gamma(K) \gamma(L) \leq \gamma(K\cap L) \gamma(K+L) 
\]
whenever $K$ and $L$ are origin-symmetric convex sets in $\R^n$. This confirms a conjecture of M.~Tehranchi \cite{Tehranchi-RefinedGaussianCorrelation}.
In fact, we show that the inequality remains valid whenever the Gaussian barycenters of $K$ and $L$ are at the origin, and characterize the equality cases.
After rescaling, this also yields the following new inequality for convex sets with (Lebesgue) barycenters at the origin:
\[
|K| |L| \leq |K \cap L| |K + L | ;
\]
this can be seen as a conjugate counterpart to Spingarn's extension of the Rogers--Shephard inequality (where $K+L$ is replaced by $K-L$ above). We also derive an additional conjugate version of a Gaussian inequality due to V.~Milman and Pajor, as well as several extensions. 
Our main tool is a new Gaussian Forward-Reverse Brascamp--Lieb inequality for centered log-concave functions, of independent interest, which is crucially applicable to degenerate Gaussian covariances.
\end{abstract}

\section{Introduction}

Let $K,L$ denote two convex sets on $\R^n$ with non-empty interior. Their Minkowski sum is denoted by $K+L = \{ x+y : x \in K , y \in L \}$. Given a measure $\mu$ on $\R^n$ with $\int_K \abs{x} d\mu(x) \in (0,\infty)$, the $\mu$-barycenter of $K$ is defined as $\frac{1}{\mu(K)} \int_K \vec x d\mu(x)$; when $\mu$ is omitted we simply mean the Lebesgue barycenter. The Lebesgue measure on $\R^n$ is denoted by $\abs{\cdot} = \abs{\cdot}_n$, and the standard Gaussian probability measure on $\R^n$ is denoted by  $\gamma = \gamma^n$.

\subsection{Rogers--Shephard--Spingarn inequality}

 It was shown by V.~Milman and Pajor \cite[Corollary 3]{MilmanPajor-NonSymmetric} that whenever the Gaussian barycenters of $K,L$ are at the origin, then for all $a^2 + b^2 = 1$:
\begin{equation} \label{eq:intro-MilmanPajor}
\gamma (K) \gamma (L) \leq \gamma\brac{\frac{1}{b} K \cap \frac{1}{a} L} \gamma \brac{a K - b L} ;
\end{equation}
when $K=-K$ and $L=-L$ are origin-symmetric and $a=b=\frac{1}{\sqrt{2}}$, this was previously observed by Schecthman, Schlumprecht and Zinn in \cite[Proposition 3]{SSZ-GaussianCorrelationConjecture}. 
Applying this to $\lambda K$ and $\lambda L$ and taking the limit as $\lambda \rightarrow 0$, Milman--Pajor observed the following:

\begin{thm*}[Rogers--Shephard--Spingarn Inequality (RSSI)] 
Let $K,L$ denote bounded convex sets with non-empty interior in $\R^n$ and barycenters at the origin. Then:
\begin{equation} \label{eq:intro-RSSI}
|K| |L| \leq |K \cap L| |K-L | .
\end{equation}
\end{thm*}
To explain our nomenclature, some historical remarks are in order. 
The product $|K \cap L| |K - L|$ was studied by Rogers and Shephard in the 1950's \cite{RogersShephard-ConvexBodiesAssociated}, who obtained the complementing sharp upper bound 
\begin{equation} \label{eq:intro-RS-upper}
|K \cap L| |K-L | \leq {2n \choose n} |K| |L| . 
\end{equation}
However, their method (based on the Brunn--Minkowski inequality) easily gives (\ref{eq:intro-RSSI}) when $K = -K$ and $L = -L$ are origin-symmetric, and so the origin-symmetric case of (\ref{eq:intro-RSSI}) is usually attributed to them (see e.g. \cite[Remark (1) following Corollary 3]{MilmanPajor-NonSymmetric}). 
More generally, given a convex set $C \subset \R^N$, Rogers and Shephard showed that for any $k$-dimensional linear subspace $F$, 
\begin{equation} \label{eq:intro-sec-proj-upper}
|C \cap F|_k |P_{F^{\perp}} C|_{N-k} \leq {N \choose k} |C|_N ,
\end{equation}
where $P_{F^{\perp}}$ denotes orthogonal projection onto the orthogonal subspace to $F$. When $C = -C$ is origin-symmetric, their method easily yields:
\begin{equation} \label{eq:intro-sec-proj-lower}
|C|_N \leq |C \cap F|_k |P_{F^{\perp}} C|_{N-k} .
\end{equation}
The symmetry assumption in (\ref{eq:intro-sec-proj-lower}) was relaxed by Spingarn \cite{Spingarn-RogersShephard} to the requirement that the barycenter of $C$ is at the origin. By using $C = K \times L \subset \R^{2n}$ when $K,L$ are origin-symmetric and $F = \{ (y,y) : y \in \R^n \}$  in (\ref{eq:intro-sec-proj-lower}) and (\ref{eq:intro-sec-proj-upper}), one easily obtains (\ref{eq:intro-RSSI}) and (\ref{eq:intro-RS-upper}) respectively (and a simple adaptation also yields (\ref{eq:intro-MilmanPajor})). Similarly, Spingarn's relaxation immediately yields (\ref{eq:intro-RSSI}) for convex $K,L$ having both barycenters at the origin. Consequently, we refer to (\ref{eq:intro-RSSI}) as the Rogers--Shephard--Spingarn inequality.

\subsection{Gaussian Correlation Inequality}

On the other hand, the Gaussian Correlation Inequality was only established in 2014 by Royen \cite{Royen-GaussianCorrelation} (cf. \cite{LatalaMatlak-GaussianCorrelation}), after being open for many decades, since the work of Pitt on the two-dimensional case \cite{Pitt-GaussianCorrelationInPlane}. It was originally conjectured and finally established by Royen for origin-symmetric convex sets $K=-K$ and $L=-L$  in $\R^n$ (see also \cite{EMilman-GCI} for an alternative argument), but was recently extended by Nakamura and Tsuji \cite{NakamuraTsuji-GCIForCentered} to the case when $K$ and $L$ have Gaussian barycenters at the origin, along with a characterization of the equality case (see \cite{NakamuraTsuji-GCIForCentered} for further extensions): 

\begin{thm*}[Gaussian Correlation Inequality (GCI) \cite{Royen-GaussianCorrelation,EMilman-GCI,NakamuraTsuji-GCIForCentered}] Let $K,L \subset \R^n$ denote convex sets with non-empty interior and Gaussian barycenters at the origin. Then:
\begin{equation} \label{eq:intro-GCI}
\gamma(K) \gamma(L) \leq \gamma(K \cap L) ,
\end{equation}
with equality if and only if 
\begin{equation} \label{eq:intro-equality}
K = K_0 \times E^{\perp},\quad L = E \times L_0 \quad\text{up to null-sets}, 
\end{equation}
for some linear subspace $E\subset \R^n$ and convex $K_0 \subset E$ and $L_0 \subset E^\perp$. 
\end{thm*}

\subsection{Gaussian conjugate Rogers--Shephard inequality}

Our main result in this work is the following theorem, which fuses between (\ref{eq:intro-RSSI}) and (\ref{eq:intro-GCI}). 

\begin{thm}[Gaussian Conjugate Rogers--Shephard Inequality (GCRSI)] \label{thm:intro-GCRSI}
Let $K,L \subset \R^n$ denote convex sets with non-empty interior and Gaussian barycenters at the origin. Then:
\begin{equation} \label{eq:intro-GCRSI}
\gamma(K) \gamma(L) \leq \gamma(K \cap L) \gamma(K + L) ,
\end{equation}
with equality if and only if (\ref{eq:intro-equality}) holds. 
\end{thm}
When $K=-K$ and $L=-L$ are origin-symmetric, the inequality (\ref{eq:intro-GCRSI}) was conjectured by M.~Tehranchi \cite{Tehranchi-RefinedGaussianCorrelation} as a possible strengthening of the GCI (\ref{eq:intro-GCI}) (since $\gamma(K+L) \leq 1$). In particular, (\ref{eq:intro-GCRSI}) yields the following stability estimate:
\[
\gamma(K \cap L) \leq (1+\eps) \gamma(K) \gamma(L) \;\; \Rightarrow \;\; \gamma(K+L) \geq \frac{1}{1+\eps} ,
\]
providing a certain measure of how far $K+L$ is from the entire $\R^n$, as required by the equality case when $\eps = 0$. 

In addition, by scaling $K,L$ by a factor of $\lambda > 0$ and letting $\lambda \rightarrow 0$, (\ref{eq:intro-GCRSI}) immediately recovers the Rogers--Shephard--Spingarn inequality (\ref{eq:intro-RSSI}) when the (Lebesgue) barycenter of $K$ is at the origin and $L = -L$ is origin-symmetric. However, 
when $L$ is not origin-symmetric, we obtain the following seemingly new inequality:

\begin{corollary}[Conjugate Rogers--Shephard--Spingarn Inequality (CRSSI)] \label{cor:intro-CRSSI}
Let $K,L \subset \R^n$ denote bounded convex sets with non-empty interior and barycenters at the origin. Then:
\begin{equation} \label{eq:intro-CRSSI}
|K| |L| \leq |K \cap L| |K+L| .
\end{equation}
\end{corollary}
\noindent In view of the sign difference between the $K+L$ term above and the $K-L$ term which appears in (\ref{eq:intro-RSSI}), we call (\ref{eq:intro-CRSSI}) the conjugate Rogers--Shephard--Spingarn inequality (CRSSI). 
By taking $L = \lambda K$ and letting $\lambda \rightarrow 0$ or $\lambda \rightarrow \infty$, it is clear that the constant $1$ in both (\ref{eq:intro-RSSI}) and (\ref{eq:intro-CRSSI}) cannot be improved, but equality does not seem to be attained in either inequalities. 

\subsection{Log-supermodularity is natural}

In a sense, 
the GCRSI resolves a troubling feature of the GCI --
  why does the measure appear twice on one side of (\ref{eq:intro-GCI}) and only once on the other side? Clearly this makes the GCI lose its utility as the sets $K,L$ become smaller, and in particular nothing survives in the scaling limit because of the different orders of magnitudes on either side of the inequality. 
 
In contrast, the GCRSI does not have any of these caveats, as the Gaussian measure appears twice on \emph{both} sides of the inequality  (\ref{eq:intro-GCRSI}), 
and so is perhaps more natural or fundamental. It may be seen as a type of log-supermodularity of the Gaussian measure with respect to the intersection and Minkowski summation operations on the semi-ring of origin-symmetric convex sets. 

\subsection{Prior Results}

In \cite{Tehranchi-RefinedGaussianCorrelation}, Tehranchi established the following inequality:
\[
\gamma(K) \gamma(L) \leq (1-s)^{-\frac{n}{2}} \gamma\brac{\sqrt{\frac{2(1-s)}{1+t}} (K \cap L)} \gamma\brac{\sqrt{\frac{1-s}{2 (1-t)}} (K + L) } ,
\]
for all $0 \leq s \leq t^2 < 1$ and origin-symmetric convex $K,L \subset \R^n$. In particular, setting $s=0$, he obtained:
\[
\gamma(K) \gamma(L) \leq \gamma\brac{\sqrt{\frac{2}{1+t}} (K \cap L)} \gamma\brac{\frac{1}{\sqrt{2 (1-t)}} (K +L) } ,
\]
which interpolates between (\ref{eq:intro-MilmanPajor}) when $t=0$ and the GCI (\ref{eq:intro-GCI}) when $t \nearrow 1$. In addition, using $s=\frac{1}{5}$ and $t = \frac{3}{5}$ yields:
\[
\gamma(K) \gamma(L) \leq \brac{\frac{5}{4}}^{\frac{n}{2}} \gamma(K \cap L) \gamma(K + L) ,
\]
confirming Tehranchi's conjecture  that the GCRSI (\ref{eq:intro-GCRSI}) holds for all origin-symmetric convex $K,L$ up to a constant exponential in the dimension $n$.

Tehranchi's conjecture was also recently investigated by Assouline, Chor and Sadovsky in \cite{ACS-RefinedKhatriSidak}. These authors verified the conjecture when $K,L$ are both unconditional convex sets, i.e.~invariant under reflection with respect to the coordinate hyperplanes $\{x_i = 0\}_{i=1,\ldots,n}$. They also considered a stronger variant of the conjecture, in which $K+L$ in (\ref{eq:intro-GCRSI}) is replaced by the smaller convex-hull $\conv(K \cup L)$, and showed \cite[Example 14]{ACS-RefinedKhatriSidak} that this version fails, even for (unconditional) rectangular sets in $\R^2$. However, this stronger version does hold when one of the sets is a symmetric slab $\{|x_1| \leq a\}$ and the other is a general origin-symmetric convex set \cite{ACS-RefinedKhatriSidak}; the usual GCI in this case is a classical result of  Khatri \cite{Khatri} and \v{S}id\'ak \cite{Sidak}.

\subsection{Extensions \`a la Milman--Pajor and open questions}

Our analysis can be generalized to give the following.
\begin{thm} \label{thm:intro-ab}
Let $K,L \subset \R^n$ denote convex sets with non-empty interior and Gaussian barycenters at the origin. Then for all $\abs{a},\abs{b} \geq 1$ so that $\abs{a+b} \geq 1$, 
\begin{equation} \label{eq:intro-ab}
\gamma(K) \gamma(L) \leq \gamma(K \cap L) \gamma(a K + bL) ,
\end{equation}
with equality when $\abs{a+b} > 1$ if and only if (\ref{eq:intro-equality}) holds. 
\end{thm}

Since $K,L$ are inclusion-wise monotone under scaling, and since the Gaussian measure is invariant under reflection about the origin, the only two interesting cases above (from which everything else follows) are:
\begin{enumerate}
\item $a=b=1$, yielding the GCRSI (\ref{eq:intro-GCRSI}). 
\item $a=2$ and $b=-1$, which in general (for non-origin-symmetric $L$) is incomparable to (\ref{eq:intro-GCRSI}); if $a > 2$, we also get the equality case. 
\end{enumerate}

It would be very interesting to obtain the case when $a=1$ and $b=-1$, since this would recover in the scaling limit the RSSI (\ref{eq:intro-RSSI}) instead of its conjugate version (\ref{eq:intro-CRSSI}). Unfortunately, our proof does not work in this case, because we reduce the problem to a certain Gaussian saturation sufficient condition, which turns out to simply be false. Consequently, we leave this case as an interesting open problem: under the assumptions of Theorem \ref{thm:intro-ab}, does it hold that
\begin{equation} \label{eq:intro-GRSI}
\gamma(K) \gamma(L) \leq \gamma(K \cap L) \gamma(K - L) \; ?
\end{equation}

\medskip

We can also obtain the following conjugate version of (\ref{eq:intro-MilmanPajor}). 
\begin{thm}[Conjugate Milman--Pajor Inequality (CMPI)] \label{thm:intro-CMPI}
Let $K,L \subset \R^n$ denote convex sets with non-empty interior and Gaussian barycenters at the origin. Then for all $a,b \in \R$ such that $a^2 + b^2 = 1$,
\begin{equation} \label{eq:intro-CMPI}
\gamma(K) \gamma(L) \leq \gamma\brac{\frac{1}{b} K \cap \frac{1}{a} L} \gamma(a K + bL) ,
\end{equation}
with equality when $a,b \neq 0$ if and only if $K = L = \R^n$. 
\end{thm}

More generally, we pose the following problem:
\begin{problem*}
Characterize those $(a,b)$ for which (\ref{eq:intro-CMPI}) holds for all $K,L$ as above. 
\end{problem*}
\noindent 

It is not hard to show (see Section \ref{sec:unify}) that any $(a,b)$ satisfying (\ref{eq:intro-CMPI}) must satisfy:
\[
\abs{a} \leq 1 ~,~ \abs{b} \leq 1 ~,~ \abs{a} + \abs{b} \geq 1 . 
\]
We are able to show the following extension of Theorem \ref{thm:intro-CMPI} (which when specialized to the case $a=b=1$, contains the GCRSI (\ref{eq:intro-GCRSI})).
\begin{thm}[Generalized Conjugate Milman--Pajor Inequality (GCMPI)] \label{thm:intro-GCMPI}
The inequality (\ref{eq:intro-CMPI}) holds for all $K,L$ as above for any $(a,b)$ in the following range:
\[
\abs{a} \leq 1 ~,~ \abs{b} \leq 1 ~,~ 3 \min(a^2,b^2) + \max(a^2,b^2) \geq 1 .
\]
In addition, given $(a,b)$ in this range:
\begin{enumerate}
\item
If $0<\abs{a}<\abs{b}=1$, then equality in (\ref{eq:intro-CMPI}) occurs if and only if $L = \R^n$. 
\item
If $0<\abs{b}<\abs{a}=1$, then equality in (\ref{eq:intro-CMPI}) occurs if and only if $K = \R^n$.
\item
If $\max(\abs{a},\abs{b}) < 1$ then equality in (\ref{eq:intro-CMPI}) occurs if and only if $K=L = \R^n$.
\end{enumerate} 
\end{thm}

In particular, we give a complete answer to the characterization problem in the case that $\abs{a}=\abs{b}$:
\begin{cor} \label{cor:intro-lambda}
The inequality (\ref{eq:intro-CMPI}) holds for all $K,L$ as above for $\abs{a}=\abs{b}=\lambda$ if and only if $\lambda \in [1/2,1]$. Note that the case $\lambda =1$ corresponds to the GCRSI (\ref{eq:intro-GCRSI}) and the case $\lambda = 1/\sqrt{2}$ to the CMPI (\ref{eq:intro-CMPI}). In particular, we have when $\lambda = 1/2$,
\[
\gamma(K) \gamma(L) \leq \gamma\brac{2 (K \cap L)} \gamma\brac{\frac{1}{2}(K + L)} ,
\]
with equality if and only if $K = L = \R^n$. 
\end{cor}

In Section \ref{sec:unify}, we combine Theorems \ref{thm:intro-ab} and \ref{thm:intro-GCMPI} into a single unified formulation (see Figure \ref{fig:diagram}), and explain why the Gaussian saturation sufficient condition is bound to fail when $a K + bL$ in (\ref{eq:intro-CMPI}) is replaced by $a K - b L$ as in (\ref{eq:intro-GRSI}), with the only exception being precisely when $a^2 + b^2=1$, as in the Milman--Pajor inequality (\ref{eq:intro-MilmanPajor}). 

\begin{remark} \label{rem:intro-scaling}
By scaling $K,L$ by a factor of $\lambda > 0$ and letting $\lambda \rightarrow 0$, analogues of (\ref{eq:intro-ab}) and (\ref{eq:intro-CMPI}) are obtained for the Lebesgue measure instead of $\gamma$, for all bounded convex sets $K,L \subset \R^n$ with non-empty interior having Lebesgue barycenters at the origin. However, by homogeneity of the Lebesgue measure and monotonicity, the only two interesting cases remain (\ref{eq:intro-RSSI}) and (\ref{eq:intro-CRSSI}). Note that the equality cases are lost in this limiting procedure. 
\end{remark}

\medskip

Another direction is to try and obtain a Gaussian analogue of the Rogers--Shephard upper bound (\ref{eq:intro-RS-upper}), which would serve as a reverse counterpart to the GCRSI (\ref{eq:intro-GCRSI}). 
However, any inequality of the form:
\[
\gamma(K \cap L) \gamma(C_1 (K + L)) \leq C_2 \gamma(C_3 K) \gamma(C_3 L) ,
\]
is necessarily false already in $\R^2$, for all $C_1> 0$ and $C_2,C_3 \geq 1$. To see this, simply take $K = \{\abs{x_1} \leq a\}$ and $L = U_\theta(K)$, where $U_\theta$ is a rotation in $\R^2$ by $\theta$ degrees. For all $\theta > 0$, $K + U_\theta(K) = \R^2$, whereas $\gamma(K \cap U_\theta(K)) \rightarrow \gamma(K)$ as $\theta \rightarrow 0$. Therefore, by selecting $a > 0$ small enough so that $\gamma(C_3 K) < \frac{1}{C_2 C_3}$, since $\gamma(C_3 K) / C_3 \leq \gamma(K)$, the reverse form above is bound to fail for small enough $\theta > 0$. 

We thus leave open the question of whether there is a sensible reverse form of the GCRSI, which yields a meaningful inequality in the scaling limit.

\subsection{Functional formulation}

While Royen's original proof of the GCI \cite{Royen-GaussianCorrelation} operated on the level of sets, the recent proofs in \cite{EMilman-GCI,NakamuraTsuji-GCIForCentered} are based on a functional formulation of the GCI (to be described later on). Recall that a function $f : \R^n \rightarrow \R_+$ is called quasi-concave if all of its super level sets $\{ f \geq t \}$ are convex, and log-concave if $\log f: \R^n \rightarrow \R \cup \{-\infty\}$ is concave.  
 It is easy to see (just by integration on level sets and using the tail formula) that an equivalent functional formulation of the GCI is that:
\begin{equation} \label{eq:intro-functional-GCI}
 \int_{\R^n} f_1 d\gamma \int_{\R^n} f_2 d\gamma \leq \int_{\R^n} f_1 f_2 d\gamma 
\end{equation}
for all quasi-concave functions $f_i : \R^n \rightarrow \R_+$, all of whose upper level sets have Gaussian barycenter at the origin. 

The first challenge in extending the GCI to the GCRSI is to realize what is the right functional formulation of the latter, as there are many potential operations on functions which would turn $f_1 = \1_K$ and $f_2 = \1_L$ into $\1_{K\cap L}$ and $\1_{K+L}$; at the same time, the homogeneity in $f_1$ and $f_2$ should be the same on both sides of the inequality. For example, one might guess that perhaps we should expect to have:
\[
 \int_{\R^n} f_1 d\gamma \int_{\R^n} f_2 d\gamma \leq \int_{\R^n} \sqrt{f_1 f_2} d\gamma \int_{\R^n} \sup_{z = x+y} \sqrt{f_1(x) f_2(y)} d\gamma(z) \;  ? 
\]
It turns out that the correct functional formulation is the following somewhat surprising one:

\begin{thm}[Functional Formulation of GCRSI] \label{thm:intro-functional}
For all quasi-concave Borel functions $f_1,f_2 : \R^n \rightarrow \R_+$ whose super level sets have Gaussian barycenters at the origin, we have:
\[
\int_{\R^n} f_1 d\gamma \int_{\R^n} f_2 d\gamma \leq \int_{\R^n} \maxo(f_1,f_2) d\gamma \int_{\R^n} f_1 \square f_2  d\gamma ,
\]
where
\[
\maxo(p,q)= \begin{cases} \max(p,q) & p,q > 0 \\ 0 & pq = 0 \end{cases} ,
\]
and $f_1 \square f_2$ denotes the following variant of sup-convolution:
\[
f_1 \square f_2 (z) = \sup_{z = x+y} \min(f_1(x),f_2(y)) . 
\]
\end{thm}

Note that while $\max(\1_K,\1_L) = \1_{K \cup L}$, we have $\maxo(\1_K , \1_L) = \1_{K \cap L}$, and clearly $\1_K \square \1_L = \1_{K+L}$, so the functional formulation recovers its geometric counterpart. Using the four functions theorem, we will show in Section \ref{sec:functional} that the geometric formulation implies back the functional one. We find this interesting, since $\maxo$ and $\square$ are not the usual functional analogues of the geometric $\cap$ and $+$ operations when extending from convex sets to log-concave functions, and the intuitive product $f_1 f_2$ in the functional formulation (\ref{eq:intro-functional-GCI}) of the GCI is no longer present. This difference in the functional formulations of the GCI and GCRSI suggests that the latter is not a simple modification of the former. Analogous functional versions hold for the RSSI (\ref{eq:intro-RSSI}) and CRSSI (\ref{eq:intro-CRSSI}) -- see Corollary \ref{cor:functional-Lebesgue}.

\subsection{Gaussian Forward-Reverse Brascamp--Lieb inequality}

The functional formulation of the GCRSI provides us with a clue for the type of functional inequality we need to prove, after noting that:
\begin{align*}
f_1(x_1) f_2(x_2) &= \maxo(f_1(x_1),f_2(x_2)) \cdot \min(f_1(x_1),f_2(x_2)) \\
& \leq \maxo(f_1(x_1),f_2(x_2)) \cdot f_1 \square f_2 (x_1 + x_2) . 
\end{align*}
We show the following:
\begin{thm} \label{thm:intro-our-FRBL}
Let $f_1,f_2,h_2 : \R^n \rightarrow \R_+$ and $h_1 : \R^{2n} \rightarrow \R_+$ denote four log-concave functions so that $\int \vec x f_i(x) d\gamma(x) = \vec 0$, $i=1,2$ (and no centering assumption on $h_j$). If
\begin{equation} \label{eq:intro-our-FRBL-assumption}
f_1(x_1) f_2(x_2) \leq h_1(x_1,x_2) h_2(x_1 + x_2) \;\;\; \forall x_1,x_2 \in \R^n ,
\end{equation}
then
\begin{equation} \label{eq:intro-our-FRBL-conclusion}
\int_{\R^n} f_1 d\gamma^n \int_{\R^n} f_2 d\gamma^n \leq \int_{\R^n} h_1(y,y) d\gamma^n(y) \int_{\R^n} h_2 d\gamma^n . 
\end{equation}
If equality occurs in (\ref{eq:intro-our-FRBL-conclusion}) then necessarily $h_2 \equiv c$ is a constant function, $c \cdot h_1(y,y) = f_1(y) f_2(y)$ for almost-every $y \in \R^n$, and
$f_1(x) = \bar f_1(P_E x)$ and $f_2(x) = \bar f_2(P_{E^{\perp}} x)$ for some linear subspace $E \subset \R^n$ and almost every $x \in \R^n$. 
\end{thm}
Applying this to $f_1 = \1_K$, $f_2 = \1_L$, $h_1 = \1_{K \times L}$ and $h_2 = \1_{K+L}$ immediately yields Theorem \ref{thm:intro-GCRSI}. More general versions, which yield Theorems \ref{thm:intro-ab} and \ref{thm:intro-GCMPI}, are formulated in Theorems \ref{thm:GCRSI-FRBL} and \ref{thm:GCMPI-FRBL} (with analysis of equality deferred to Theorems \ref{thm:GCRSI-equality} and \ref{thm:GCMPI-equality}, respectively).

\smallskip

The inequality statement of Theorem \ref{thm:intro-our-FRBL} is a particular case of the following Gaussian Forward-Reverse Brascamp--Lieb (FRBL) inequality for centered log-concave functions, which is the main new tool we develop in this work, and of independent interest. Given positive semi-definite (possibly degenerate) $n \times n$ matrices $\Gamma,A \geq 0$, we denote by $\gamma_\Gamma$ the Gaussian probability measure on $\R^n$ with covariance $\Gamma$, and by $g_A(x) = \exp(-\frac{1}{2} \scalar{Ax,x})$ a centered Gaussian function on $\R^n$. 

\begin{thm}[Gaussian Forward-Reverse Brascamp--Lieb inequality] \label{thm:intro-FRBL}
    Let $E = \oplus_{i=1}^I \R^{n_i}$, let $L_j : E \rightarrow \R^{m_j}$ denote linear surjective maps, $j = 1,\ldots,J$, let $\{c_i\}_{i=1,\ldots,I}$ and $\{d_j\}_{j =1,\ldots,J}$ denote positive scalars, and let $\Q$ denote a symmetric operator (of arbitrary signature) on $E$. Let $\Sigma_i > 0$ denote positive-definite covariances on $\R^{n_i}$ ($i=1,\ldots,I$), and $\Gamma_j \geq 0$ denote positive semi-definite covariances (\textbf{possibly degenerate}) on $\R^{m_j}$ ($j=1,\ldots,J$).  
    Then for all log-concave $f_i : \R^{n_i} \rightarrow \R_+$ so that $\int \vec x f_i(x) d\gamma_{\Sigma_i}(x) = \vec 0$ ($i=1,\ldots,I$),  and for all log-concave $h_j : \R^{m_j} \rightarrow \R_+$ (\textbf{with no centering assumption}, $j = 1,\ldots,J$), we have
        \begin{align*}
        & \prod_{i=1}^I f_i(x_i)^{c_i} \le e^{-\frac12 \scalar{\Q x, x}}\prod_{j=1}^J h_j(L_jx)^{d_j} \;\;\; \forall x \in E \\
        & \Rightarrow
        \prod_{i=1}^I \big( \int_{\R^{n_i}} f_i\, d\gamma_{\Sigma_i} \big)^{c_i}
        \le {\rm FR}^{(\mathcal{G})}_{LC}
        \prod_{j=1}^J \big( \int_{\R^{m_j}} h_j\,d\gamma_{\Gamma_j} \big)^{d_j},
     \end{align*}
        where $\FR^{(\G)}_{LC}$ denotes the best constant when testing the above implication with centered Gaussian functions $f_i = g_{A_i}$ and $h_j = g_{B_j}$ ($A_i, B_j \geq 0$), namely:
    \[
    \FR^{(\G)}_{LC} := 
    \sup\bigg\{ \frac{ \prod_{j=1}^J {\rm det}\, \big( {\rm Id}_{m_j} + \Gamma_j B_j \big)^{d_j/2} }{ \prod_{i=1}^I {\rm det}\, \big( {\rm Id}_{n_i} + \Sigma_i A_i \big)^{c_i/2} }:\; {\rm diag}\, (c_1A_1,\ldots, c_I A_I) \ge \Q+ \sum_{j=1}^J d_j L_j^* B_j L_j \bigg\}. 
    \]
\end{thm}

When $\Q=0$ and all integration is performed with respect to the Lebesgue measure, there is no need to assume that the functions are log-concave or have barycenters at the origin -- that case is the original Forward-Reverse Brascamp--Lieb inequality, introduced and established by Liu--Courtade--Cuff--Verd\'u in \cite{LCCV-BrascampLieb} (see also \cite{CourtadeLiu-BrascampLieb}) using an information-theoretic approach involving entropies. The ``Forward-Reverse" nomenclature captures the fact that when $I=1$ the inequality reduces to the usual (forward) Brascamp--Lieb inequality \cite{BrascampLieb-YoungInq, Lieb-MultiDimBL,CarlenLiebLoss-EntropyOnSn,BCCT-BrascampLieb}, whereas when $J=1$ the inequality becomes the reverse Brascamp--Lieb inequality \cite{Barthe-ReverseBL-CRAS,Barthe-ReverseBL,BartheCordero-InverseBLviaSemiGroup,Valdimarsson-GenCaffarelli,BartheHuet} (see \cite{LCCV-BrascampLieb} for more details). 

 In contrast, when integration above is with respect to general Gaussian measures, or equivalently (after modifying $f_i$ and $h_j$ by Gaussian factors), when integration is with respect to Lebesgue measure but $\Q$ has a general non-trivial signature, Theorem \ref{thm:intro-FRBL} would be false without some conditions on $f_i$ and $h_j$ (beyond integrability). Indeed, if Theorem \ref{thm:intro-FRBL} would remain true in that generality, 
the GCI (\ref{eq:intro-GCI}) or GCRSI (\ref{eq:intro-GCRSI}) would apply to arbitrary (non-convex) sets $K,L$, which is easily seen to be false (even if the sets are origin-symmetric). This phenomenon was first observed in the context of the Inverse Brascamp--Lieb inequality, where the effect of the signature of $\Q$ on the validity of this inequality was studied in detail by Barthe and Wolff \cite{BartheWolff-InverseBrascampLieb}. The importance of the signature of $\Q$ also appears in the work of Courtade--Liu in the context of the Forward-Reverse Brascamp--Lieb inequality for certain Gaussian measures \cite[Theorem 4.8]{CourtadeLiu-BrascampLieb}. However, it was subsequently realized by Nakamura and Tsuji \cite{NakamuraTsuji-InverseBrascampLieb,NakamuraTsuji-GCIForCentered} that for (even, and later centered) log-concave data, the signature of $\Q$ is actually irrelevant, allowing for multiple new applications \cite{NakamuraTsuji-InverseBrascampLieb,EMilman-GCI, NakamuraTsuji-GCIForCentered}. A similar phenomenon occurs with the Forward-Reverse Brascamp--Lieb inequality. This is not entirely surprising, since an observation of Wolff \cite[Remark 4.5]{CourtadeLiu-BrascampLieb} shows that a version of the Inverse Brascamp--Lieb inequality (crucially allowing negative exponents) is in fact equivalent to the Forward-Reverse Brascamp--Lieb inequality (at least, when all the functions are even and all Gaussian measures are nondegenerate). We will not use this equivalence here, but rather derive Theorem \ref{thm:intro-FRBL} from scratch using a similar approach to the one in \cite{NakamuraTsuji-InverseBrascampLieb} involving self-convolutions and the Central Limit Theorem. To treat our desired level of generality (essential for establishing all of our applications in this work), several new challenges arise: 
\begin{itemize}
\item It is crucial for us to allow for degenerate covariances $\Gamma_j$ (which may not be positive-definite), because for our application in Theorem \ref{thm:intro-our-FRBL} we need to integrate $h_1$ on the diagonal subspace $F = \{ (y,y) : y \in \R^n \} \subset \R^{2n}$. This requires a delicate approximation argument (and taking limits in the correct order). 
\item Another non-standard approximation step is due to the fact that we have a product of functions $\Pi_{j=1}^J h_j(L_j x)$ on the right-hand side of the assumption, which are mutually interacting on $E$, and we need to regularize all of them simultaneously while preserving the inequality. 
\item We cannot afford to assume that the $h_j$'s have Gaussian barycenter at the origin (since in our application, this may not hold for $h_2 = \1_{K + L}$ nor for $h_1(y,y) = \1_{K \cap L}(y)$). 
\item The analysis of the equality case in Theorem \ref{thm:intro-our-FRBL} requires a series of new arguments, since the approximation argument from the non-degenerate case destroys all hope of tracking the equality along the approximation. In particular, we need to show that the Gaussian barycenters of optimal $h_1$ and $h_2$ must be at the origin, and analyze the equality case for a \emph{partial} Gaussian saturation (when $f_1,f_2,h_2$ are centered Gaussians but $h_1$ is a general log-concave function). 
\end{itemize}

Thus, Theorem \ref{thm:intro-FRBL} reduces the task of finding the optimal constant on the right-hand-side of Theorem \ref{thm:intro-our-FRBL} (and the various other versions we consider in this work) to a semi-definite algebraic inequality for symmetric matrices. This turns out to be a surprisingly non-trivial task, which is carried out in Sections \ref{sec:GCRSI} and \ref{sec:GCMPI}, where we also show that a naive approach can either fail or lead to the ``wrong" inequality. 
\medskip

The rest of this work is organized as follows. In Section \ref{sec:notation} we introduce some convenient definitions and notation, and state Theorem \ref{thm:FRBL}, a generalized version of Theorem \ref{thm:intro-FRBL}. In Section \ref{sec:preparatory} we record and derive some useful preparatory lemmas. The proof of Theorem \ref{thm:FRBL} is carried out in Section \ref{sec:FRBL-proof}. In Sections \ref{sec:GCRSI} and \ref{sec:GCMPI} we formulate appropriate functional versions of Theorems \ref{thm:intro-ab} and \ref{thm:intro-GCMPI} for log-concave functions, and find the sharp constants $\FR_{LC}^{(\G)}$ in the corresponding Gaussian saturation reductions. In Section \ref{sec:equality} we analyze the equality conditions in our inequalities. In Section \ref{sec:unify} we put forward a general problem of characterizing for which $\alpha,\beta,a ,b \in \R \setminus \{0\}$ the inequality $\gamma(K) \gamma(L) \leq \gamma(\alpha K \cap \beta L) \gamma(a L + b L)$ holds; we explain why the Gaussian saturation is bound to fail for some cases, and present a partial characterization. In Section \ref{sec:functional} we derive an equivalent functional formulation of our geometric results for quasi-concave functions.

\section{Notation} \label{sec:notation}

Given a Euclidean space $E$ endowed with a scalar product $\scalar{\cdot,\cdot}$ we denote the Euclidean norm by $\abs{x} := \sqrt{\scalar{x,x}}$. 
The family of all symmetric operators on $E$ is denoted by $\Sym(E)$. When $E = \R^n$, we simply write $\Sym(n)$, and naturally identify it with the set of all $n \times n$ symmetric matrices over $\R$. The subsets of positive semi-definite and positive-definite operators are denoted by $\Sym_{\geq 0}(E)$ and $\Sym_{>0}(E)$, respectively. For $A,B \in \Sym(E)$, we write $A \leq B$ if $B - A \in \Sym_{\geq 0}(E)$. 
                                  
    The Lebesgue measure on $E$ is denoted by $dx$. Given $A \in \Sym_{\geq 0}(E)$, we denote
    \[
    g_A(x) := e^{-\frac12 \langle Ax, x\rangle} ,
    \]
    referring to $g_A$ as a Gaussian (even though $A$ may be degenerate). 
    For $\Sigma \in \Sym_{\geq 0}(E)$,  we denote the Gaussian probability measure on $E$ with covariance $\Sigma$ by $\gamma_{\Sigma}$. When $\Sigma \in \Sym_{>0}(E)$ is non-degenerate, this means
    \[
    \gamma_\Sigma:= \frac{1}{\sqrt{\det(2\pi \Sigma)}} e^{-\frac12 \scalar{\Sigma^{-1}x, x}}\, dx ,
    \]
    whereas in general it is useful to interpret $\gamma_\Sigma$ as the push-forward of $\gamma_{\Id_E}$ via $\Sigma^{1/2}$. When $\Sigma = \Id_E$, we simply write $\gamma := \gamma_{\Id_E}$; when $E = \R^n$, we will sometimes use $\gamma^n$. 
                        
    We will also formally consider the case when $\Sigma = \infty$, in which case we interpret $\gamma_\Sigma$ as the Lebesgue measure $dx$ on $E$ (the only exceptional case when $\gamma_\Sigma$ is not a probability measure); in that case we interpret $\Sigma^{-1}$ as $0$. 
    By abuse of notation, it will be convenient to treat $\gamma_\Sigma$ both as a measure and as its corresponding density $\frac{d\gamma_\Sigma(x)}{dx}$, depending on the context.   
      
    \begin{definition}[More log-concave / log-convex]
    A function $f : \R^n \rightarrow \R_+$ is called log-concave (respectively, log-convex) if $\log f : \R^n \rightarrow \R \cup \{-\infty\}$ is concave (respectively, convex). Given $A \in \Sym_{\geq 0}(n)$, we say that $f$ is more log-concave (respectively, more log-convex) than $g_{A}$ if $f = g_{A} h$ for some log-concave (respectively, log-convex) function $h$.
    \end{definition}
    
    When $A = 0$, note that $f$ is more log-concave than $g_A$ if and only if $f$ is log-concave. We also formally consider the case that $A = \infty$, in which case every function is more log-convex than $g_A$, and only the zero function is more log-concave than $g_A$. When $f = \exp(-V)$ with $C^2$-smooth $V$ and given $0\le A^\flat\le A^\sharp < \infty$, it is easy to see that $f$ is more log-concave than $g_{A^\flat}$ and more log-convex than $g_{A^\sharp}$ if and only if
    \[
    A^\flat \leq \nabla^2 V(x) \leq A^\sharp  \;\;\; \forall x \in \R^n . 
    \]
    
    \begin{definition}[Classes of log-concave functions $\F$]
    For matrices $A^\flat \in \Sym_{\geq 0}(n)$ and $A^\sharp \in \Sym_{\geq 0}(n) \cup \{\infty\}$ with $0\le A^\flat\le A^\sharp \leq \infty$ (referred to as ``regularization parameters"), and a measure $\mu$ on $\R^n$, we denote by $\F_{ A^\flat, A^\sharp} (\mu) $ the set of non-negative functions on $\R^n$ which are $\mu$-integrable, more log-concave than $g_{A^\flat}$ and more log-convex than $g_{A^{\sharp}}$. 
            The subclass of $\mu$-centered functions, namely $f\in \mathcal{F}_{A^\flat,A^\sharp}(\mu)$ such that $\int_{\R^n} \vec x f(x) \, d\mu(x) = \vec 0$, is denoted by $\F^{(o)}_{A^\flat,A^\sharp}(\mu)$. The subclass of even functions is denoted by $\F^{(e)}_{A^\flat,A^\sharp}(\mu)$.     
    We abbreviate $ \mathcal{F}_{0,\infty}(\mu)$ by $\mathcal{F}_{LC}(\mu)$, the class of $\mu$-integrable log-concave functions.
    \end{definition}
    
\begin{definition}[Brascamp--Lieb datum]
Fix $I,J \in \mathbb{N}$. 
    Let $\mathbf{n} = (n_1,\ldots, n_I) \in \mathbb{N}^I$, $\mathbf{m} = (m_1,\ldots, m_J) \in \mathbb{N}^J$, $\mathbf{c}= (c_1,\ldots,c_I) \in (0,\infty)^I$, and $\mathbf{d} = (d_1,\ldots, d_J)\in (0,\infty)^J$.
    We denote $E := \bigoplus_{i=1}^I \R^{n_i}$. Let $\L = (L_j)_{j=1,\ldots,J}$ where $L_j : E \to \R^{m_j}$ are linear surjective maps, and let $\Q \in \Sym(E)$.  
    For $i=1,\ldots,I$ and $j=1,\ldots,J$, let $\Sigma_i \in \Sym_{>0}(n_i) \cup \{\infty\}$ and $\Gamma_j\in \Sym_{\geq 0}(m_j) \cup \{\infty\}$, and let $A_i^\flat \in \Sym_{\geq 0}(n_i)$, $A_i^\sharp \in \Sym_{\geq 0}(n_i) \cup \{\infty\}$, $B_j^\flat \in \Sym_{\geq 0}(m_j)$, $B_j^\sharp \in \Sym_{\geq 0}(m_j) \cup \{\infty\}$ with 
    $0\le A_i^\flat\le A_i^\sharp\le \infty$ and $0\le B_j^\flat \le B_j^\sharp \le \infty$, which we collectively denote by $\boldsymbol{\Sigma}$, $\boldsymbol{\Gamma}$, $\A$ and $\B$, respectively.  
    We call $(\c,\d,\L,\Q,\boldsymbol{\Sigma},\boldsymbol{\Gamma},\A,\B)$ a Brascamp--Lieb datum (with regularization parameters $\A,\B$). \\
    When $A_i^\flat = B_j^\flat = 0$ and $A_i^\sharp = B_j^\sharp = \infty$ for all $i=1,\ldots,I$ and $j=1,\ldots,J$, we say there is no regularization and simply write $(\c,\d,\L,\Q,\boldsymbol{\Sigma},\boldsymbol{\Gamma})$.
\end{definition}

\begin{definition}[Forward-Reverse constant]
    Let $(\c,\d,\L,\Q,\boldsymbol{\Sigma},\boldsymbol{\Gamma},\A,\B)$  be a Brascamp--Lieb datum. For $* \in \{o,e,\G\}$, define  $\FR^{(*)}_{\A,\B}= \FR^{(*)}(\mathbf{c},\mathbf{d},\mathbf{L},\mathcal{Q} ,\boldsymbol{\Sigma},\boldsymbol{\Gamma},\A,\B) \in (0,\infty]$ to be the smallest constant in the following implication:
        \begin{align}
        \nonumber         & \forall f_i \in \C_{f,i}^{(*)} \;\;  \forall h_j \in \C_{h,j}^{(*)} \; \text{ such that } \;
         \prod_{i=1}^I f_i(x_i)^{c_i} \le e^{-\frac12 \langle \mathcal{Q} x,x\rangle} \prod_{j=1}^J h_j(L_jx)^{d_j}  \;\;\; \forall x \in E \\
        \label{eq:FR-conclusion}
        & \Rightarrow
        \prod_{i=1}^I \big( \int_{\R^{n_i}} f_i\, d\gamma_{\Sigma_i} \big)^{c_i}
        \le \FR^{(*)}_{\A,\B} 
        \prod_{j=1}^J \big( \int_{\R^{m_j}} h_j\,d\gamma_{\Gamma_j} \big)^{d_j} ,
        \end{align}
        where the families $\C_{f,i}^{(*)}$ and $\C_{h,j}^{(*)}$ are defined as follows:
        \begin{itemize}
	\item  If $*=o$ then $\C_{f,i}^{(o)} = \F_{A_i^\flat, A_i^\sharp}^{(o)}(\gamma_{\Sigma_i})$ and $\C_{h,j}^{(o)} = \F_{B_j^\flat, B_j^\sharp}(\gamma_{\Gamma_j})$. Note that there is \textbf{no requirement for the functions $h_j$ to have barycenter at the origin}. We will also denote $\FR_{\A,\B} = \FR^{(o)}_{\A,\B}$.
	\item If $*=e$ then $\C_{f,i}^{(e)} = \F_{A_i^\flat, A_i^\sharp}^{(e)}(\gamma_{\Sigma_i})$ and $\C_{h,j}^{(e)} =\F^{(e)}_{B_j^\flat, B_j^\sharp}(\gamma_{\Gamma_j})$.
	\item If $* = \G$ then $\C_{f,i}^{(\G)} = \{ g_{A_i} \in  \mathcal{F}_{A_i^\flat, A_i^\sharp}^{(e)}(\gamma_{\Sigma_i})\}$ and $\C_{h,j}^{(\G)} = \{ g_{B_j} \in  \mathcal{F}^{(e)}_{B_j^\flat, B_j^\sharp}(\gamma_{\Gamma_j})\}$. Explicitly, in the two most important cases, 
        \begin{equation} \label{eq:Gaussian-constant}
        \FR_{\A,\B}^{(\G)}
        :=  \begin{cases} 
        \sup \frac{ \prod_{j=1}^J {\rm det}\, \big( {\rm Id}_{m_j} + \Gamma_j B_j \big)^{d_j/2} }{ \prod_{i=1}^I {\rm det}\, \big( {\rm Id}_{n_i} + \Sigma_i A_i \big)^{c_i/2} } &  \Sigma_i, \Gamma_j < \infty \\
        \\
         \sup \frac{ \prod_{j=1}^J (\det (B_j / 2 \pi))^{d_j/2} }{ \prod_{i=1}^I (\det (A_i / 2\pi))^{c_i/2} }  &  \Sigma_i = \Gamma_j = \infty 
         \end{cases}
         ,
        \end{equation} 
        where the supremum is taken over all $ A_i^\flat\le  A_i \le A_i^\sharp$ and $B_j^\flat\le B_j \le B_j^\sharp $ ($A_i, B_j < \infty$, $A_i > 0$ if $\Sigma_i = \infty$ and $B_j > 0$ if $\Gamma_j = \infty$) such that 
        \[
        {\rm diag}\, (c_1A_1,\ldots, c_IA_I) \ge  \mathcal{Q} + \sum_{j=1}^J d_j L_j^* B_j L_j. 
        \]
       \end{itemize}
     When $A_i^\flat = B_j^\flat = 0$ and $A_i^\sharp = B_j^\sharp = \infty$ for all $i=1,\ldots,I$ and $j=1,\ldots,J$, we simply denote the above constants by $\FR_{LC} = \FR^{(o)}_{LC}$, $\FR^{(e)}_{LC}$ and $\FR^{(\G)}_{LC}$, respectively. 
    \end{definition}
    
    \begin{remark}
    Here and below, we freely use that $\det(\Id_n + C D) = \det(\Id_m + D C)$ for any $n \times m$ matrix $C$ and $m \times n$ matrix $D$. Consequently, $\det(\Id_{n_i} + \Sigma_i A_i) = \det(\Id_{n_i} + \Sigma_i^{1/2} A_i \Sigma_i^{1/2}) = \det(\Id_{n_i} + A_i \Sigma_i)$ for example. 
    \end{remark}

With these notations, we formulate a generalized version of Theorem \ref{thm:intro-FRBL}. 

\begin{theorem}\label{thm:FRBL}
    Let $(\mathbf{c},\mathbf{d},\mathbf{L},\mathcal{Q} ,\boldsymbol{\Sigma},\boldsymbol{\Gamma},\A,\B)$ be a Brascamp--Lieb datum. Assume that $\Sigma_i \in \Sym_{> 0}(n_i)$ and $\Gamma_j \in \Sym_{\geq 0}(m_j)$, and that either all $A_i^\sharp = B_j^\sharp = \infty$ or $A_i^\sharp, B_j^\sharp < \infty$ (and as usual, $A^\flat_i, B^\flat_j \geq 0$). Then $\FR^{(o)}_{\A,\B} =  \FR^{(e)}_{\A,\B} =  \FR^{(\G)}_{\A,\B}$. \\
    In particular, $\FR^{(o)}_{LC}= \FR^{(e)}_{LC} =  \FR^{(\G)}_{LC}$, namely, for all $f_i\in \mathcal{F}_{LC}^{(o)}(\gamma_{\Sigma_i})$ and  $h_j\in \mathcal{F}_{LC}(\gamma_{\Gamma_j})$ such that
    \begin{equation} \label{eq:FRBL-assumption}
    \prod_{i=1}^I f_i(x_i)^{c_i} \le e^{-\frac12 \langle \mathcal{Q} x, x\rangle}\prod_{j=1}^J h_j(L_jx)^{d_j} \;\;\; \forall x \in E  ,
    \end{equation}
    it holds that
    \begin{equation} \label{eq:FRBL-conclusion}
    \prod_{i=1}^I \brac{ \int_{\R^{n_i}} f_i\, d\gamma_{\Sigma_i} }^{c_i}
        \le {\rm FR}^{(\mathcal{G})}_{LC}
        \prod_{j=1}^J \brac{ \int_{\R^{m_j}} h_j\,d\gamma_{\Gamma_j} }^{d_j} . 
     \end{equation}
\end{theorem}

\begin{remark}
Note that contrary to $\Gamma_j \geq 0$ which are allowed to be degenerate, we require $\Sigma_i > 0$ to be non-degenerate. The reason is that this does not lose any generality: for degenerate $\Sigma_i \geq 0$, we can always redefine $\R^{n_i}$ to be the subspace perpendicular to the null-space of $\Sigma_i$, which would only weaken the assumption (\ref{eq:FRBL-assumption}), while preserving the conclusion (\ref{eq:FRBL-conclusion}). We also remark that we can handle the case when $\Sigma_i = \Gamma_j = \infty$, or more generally, when the Gaussian measures $\gamma_{\Sigma_i}$ and $\gamma_{\Gamma_j}$ degenerate to be products of the Lebesgue measure on a subspace $F$ and a Gaussian measure on $F^{\perp}$, but this would make the notation and analysis even heavier. Since we can get this case by a scaling limit argument, we restrict to Gaussian probability measures in our formulation. 
\end{remark}

\section{Preparatory Lemmas} \label{sec:preparatory}

\subsection{Equivalent Definitions}

\begin{lemma} \label{lem:equivalence}
Let $(\mathbf{c},\mathbf{d},\mathbf{L},\mathcal{Q} ,\boldsymbol{\Sigma},\boldsymbol{\Gamma},\A,\B) $  be a Brascamp--Lieb datum.
For all $* \in \{o,e,\G\}$, $\FR_{\A,\B}^{(*)}$ is 
the supremum over all constants $C > 0$ in the following statement:
\begin{align*}
& \exists \{f_i\}_{i=1,\ldots,I} \;\; \exists \{h_j\}_{j=1,\ldots,J} \;\;  \int f_i d\gamma_{\Sigma_i} = \int h_j d\gamma_{\Gamma_j} = 1 \\
& \text{such that } \;\; \prod_{i=1}^I f_i(x_i)^{c_i} \leq  \frac{1}{C}  e^{-\frac12 \langle \Q x, x\rangle}\prod_{j=1}^J h_j(L_jx)^{d_j}  \;\;\; \forall x \in E   ,
\end{align*}
where $f_i$ and $h_j$ above are taken from the following families:
\begin{itemize}
\item If $*=o$ then $f_i\in \F_{A_i^\flat, A_i^\sharp}^{(o)}(\gamma_{\Sigma_i}),\;  h_j\in \F_{B_j^\flat, B_j^\sharp}(\gamma_{\Gamma_j})$.
\item If $*=e$ then $f_i\in \F_{A_i^\flat, A_i^\sharp}^{(e)}(\gamma_{\Sigma_i}),\;  h_j\in \F^{(e)}_{B_j^\flat, B_j^\sharp}(\gamma_{\Gamma_j})$. 
\item If $* = \G$ then $f_i = p_i g_{A_i} \in \mathcal{F}_{A_i^\flat, A_i^\sharp}^{(e)}(\gamma_{\Sigma_i}) , \; h_j = q_j g_{B_j} \in  \mathcal{F}^{(e)}_{B_j^\flat, B_j^\sharp}(\gamma_{\Gamma_j})$ (for some constants $p_i,q_j > 0$). 
\end{itemize}
\end{lemma}
\begin{proof}
Recalling the original definition of $\FR_{\A,\B}^{(*)}$, write:
\[
\FR_{\A,\B}^{(*)} = \sup_{f_i,h_j} \set{  \frac{\prod_{i=1}^I \brac{ \int_{\R^{n_i}} f_i\, d\gamma_{\Sigma_i} }^{c_i}}{ \prod_{j=1}^J \brac{ \int_{\R^{m_j}} h_j\,d\gamma_{\Gamma_j} }^{d_j}   }  : \sup_{x \in E} \frac{\prod_{i=1}^I f_i(x_i)^{c_i}}{ e^{-\frac12 \langle \Q x, x\rangle}\prod_{j=1}^J h_j(L_jx)^{d_j}} \leq 1 } ,
\]
where the supremum is taken over $f_i \in \C_{f,i}^{(*)}$ and $h_j \in \C_{h,j}^{(*)}$ such that $\int_{\R^{n_i}} f_i\, d\gamma_{\Sigma_i} > 0$ and $\int_{\R^{m_j}} h_j\,d\gamma_{\Gamma_j} > 0$, and we use the convention that $0 / 0 = 0$. Whenever the families $\C_{f,i}^{(*)}$ and $\C_{h,j}^{(*)}$ are closed under multiplication by a positive constant, we can proceed by normalizing $f_i,h_j$ and using homogeneity as follows:
\begin{align*}
& = \sup_{f_i,h_j} \set{  \frac{\prod_{i=1}^I \brac{ \int_{\R^{n_i}} f_i\, d\gamma_{\Sigma_i} }^{c_i}}{ \prod_{j=1}^J \brac{ \int_{\R^{m_j}} h_j\,d\gamma_{\Gamma_j} }^{d_j}   } / \sup_{x \in E} \frac{\prod_{i=1}^I f_i(x_i)^{c_i}}{ e^{-\frac12 \langle \Q x, x\rangle}\prod_{j=1}^J h_j(L_jx)^{d_j}} } \\
& = \sup_{f_i,h_j} \set{ 1 /  \sup_{x \in E} \frac{\prod_{i=1}^I f_i(x_i)^{c_i}}{ e^{-\frac12 \langle \Q x, x\rangle}\prod_{j=1}^J h_j(L_jx)^{d_j}}  : \int f_i \, d\gamma_{\Sigma_i} = \int h_j \, d\gamma_{\Gamma_j} = 1 } \\
& = 1 / \inf_{f_i,h_j}  \set{ \sup_{x \in E} \frac{\prod_{i=1}^I f_i(x_i)^{c_i}}{ e^{-\frac12 \langle \Q x, x\rangle}\prod_{j=1}^J h_j(L_jx)^{d_j}}  : \int f_i \, d\gamma_{\Sigma_i} = \int h_j \, d\gamma_{\Gamma_j} = 1 } .
\end{align*}
This establishes the claim for $* \in \{ o , e\}$. 

For $* = \G$, we cannot simply normalize $g_{A_i} \in \C^{(\G)}_{f,i}$ or $g_{B_j} \in \C^{(\G)}_{h,j}$ as we did above, because we always have $g_{A_i}(0) = g_{B_j}(0) = 1$. 
 However, we observe that the value of $\FR^{(\G)}_{\A,\B}$ does not change if instead of taking supremum over $f_i = g_{A_i} \in \C^{(e)}_{f,i}$ and $h_j = g_{B_j} \in \C^{(e)}_{h,j}$, we test $f_i = p_i g_{A_i} \in \C^{(e)}_{f,i}$ and $h_j = q_j g_{B_j} \in \C^{(e)}_{h,j}$ for all scalars $p_i, q_j > 0$. Indeed, this is a simple consequence of the homogeneity of the integrals and the fact that for any quadratic form $Q \in \Sym(E)$ and $D \geq 1$,
\[
\scalar{Q x, x} + \log D \geq 0 \;\;\; \forall x \in E \;\;\; \Leftrightarrow \;\;\; \scalar{Q x, x} \geq 0 \;\;\; \forall x \in E .
\]
Hence, by closing $\C^{(\G)}_{f,i}$ and $\C^{(\G)}_{h,j}$ under positive scalar multiples, the previous argument applies, establishing the claim also for $* = \G$. 
\end{proof}

\subsection{Log-concave bounds}

\begin{lemma} \label{lem:log-concave}
For any log-concave function $h : \R^n \rightarrow \R_+$, there exists $a \in \R^n$ and $b \in \R$ such that:
\[
h(x) \leq e^{\scalar{a,x} + b} \;\;\; \forall x \in \R^n . 
\]
\end{lemma}
\begin{proof}
If $h$ is identically zero, the claim is obvious. Otherwise, by translating $h$ if necessary, we may assume that $h(0) > 0$. 
Now simply use $b = \log h(0)$ and any subgradient $-a \in \partial (-\log h)(0)$ in the subdifferential of the convex function $-\log h$ at the origin. 
\end{proof}

\begin{corollary} \label{cor:log-concave-integrable}
For any log-concave function $h : \R^n \rightarrow \R_+$ and $\Sigma \in \Sym_{\geq 0}(n)$ (also degenerate), $h$ is $\gamma_\Sigma$-integrable. \end{corollary}

\begin{lemma} \label{lem:log-concave-cov}
For any log-concave density $h : \R^n \rightarrow \R_+$ with $\int h(x) dx = 1$, 
\[
\norm{h}_{L^\infty} \leq \frac{C_n}{\det^{\frac{1}{2}} \Cov(h)} ,
\]
where $C_n > 1$ is some constant depending solely on $n$. 
\end{lemma}
\begin{remark}
Thanks to the recent remarkable resolution of the Slicing Problem by Klartag and Lehec \cite{KlartagLehec-SlicingSolved}, it is actually known that $C_n \leq C^n$ for some dimension-independent constant $C > 1$, but we do not require such a powerful result here -- any constant $C_n$ independent of $h$ is good enough for our application. 
\end{remark}
\begin{proof}
The isotropic constant $\mathcal{L}_h$ associated to $h : \R^n \rightarrow \R_+$ is defined as $\mathcal{L}_h := \norm{h}^{\frac{1}{n}}_{L^\infty} \det^{\frac{1}{2n}} \Cov(h)$, and a well-known trivial estimate is $\mathcal{L}_h \leq C \sqrt{n}$ for all log-concave $h$ (see e.g.~\cite{GreekBook}). \end{proof}

\subsection{Gaussian approximation of identity}

\begin{definition}[Gaussian approximation of identity]
A family $\{\gamma_{\Sigma_t}\}_{t > 0}$ is called a Gaussian approximation of identity in $\R^n$ if $\Sym_{>0}(n) \ni \Sigma_t \rightarrow 0$ as $t \rightarrow 0$. \end{definition}
\begin{remark}
In the finite dimensional case all norms on $\Sym(n) \subset \R^{n^2}$ are equivalent, so throughout this work we do not specify a particular one. However, we will sometimes invoke the operator norm $\norm{A}_{op}$, utilizing that $\scalar{A x,x} \leq \norm{A}_{op} \abs{x}^2$ and  $-\norm{A}_{op} \Id_n \leq A \leq \norm{A}_{op} \Id_n$. 
\end{remark}

\begin{lemma} \label{lem:approximation-of-identity}
Let $\{ \gamma_{\Sigma_t} \}_{t > 0}$ be a Gaussian approximation of identity in $\R^n$. Then for any log-concave function $h : \R^n \rightarrow \R_+$:
\begin{enumerate}
\item $\lim_{t \to 0} (h \ast \gamma_{\Sigma_t})(x) = h(x)$ for any continuity point $x$ of $h$ (and hence, for almost-every $x \in \R^n$). 
\item There exist $C > 0$ and $a \in \R^n$ such that $(h \ast \gamma_{\Sigma_t})(x) \leq  C e^{\scalar{a,x}}$ for all $x \in \R^n$ and small enough $t > 0$. 
\end{enumerate}
\end{lemma}
\begin{proof}
To show the first assertion, let $x$ be a point of continuity of $h$. By a standard argument, it is enough to show that for any $\eps > 0$, 
\[
\lim_{t \rightarrow 0} \int_{\R^n \setminus B(0,\eps)} h(x+y) d\gamma_{\Sigma_t}(y) = 0 . 
\]
By Lemma \ref{lem:log-concave} applied to $h(x+\cdot)$, there exist $a \in \R^n$ and $b \in \R$ such that $h(x+y) \leq e^{\scalar{a,y}+b}$ for all $y \in \R^n$, so it is enough to show that
\begin{equation} \label{eq:approximation-goal}
\lim_{t \rightarrow 0} \int_{\R^n \setminus B(0,\eps)} e^{\scalar{a,y}} d\gamma_{\Sigma_t}(y) = 0 .
\end{equation}
A computation reveals that
\[ \int_{\R^n \setminus B(0,\eps)} e^{\scalar{a,y}} d\gamma_{\Sigma_t}(y) = e^{\frac{1}{2} \scalar{\Sigma_t a,a}} \int_{\R^n \setminus B(-\Sigma_t a,\eps)}  \gamma_{\Sigma_t}(z) dz = e^{\frac{1}{2} \scalar{\Sigma_t a,a}} \gamma(\R^n \setminus \Sigma_t^{-1/2} B(-\Sigma_t a,\eps)) .
\] Denoting $R_t^2 := \norm{\Sigma_t}_{op}$, it follows that for small-enough $R_t > 0$:
\[
\int_{\R^n \setminus B(0,\eps)} e^{\scalar{a,y}} d\gamma_{\Sigma_t}(y)  \leq e^{\frac{R_t^2}{2} |a|^2} \gamma\brac{\R^n \setminus B\brac{0, \frac{\eps}{R_t} - R_t |a|}} ,
\] 
and since $R_t \rightarrow 0$ by assumption, we confirm (\ref{eq:approximation-goal}). 

For the second assertion, if $a \in \R^n$ and $b \in \R$ are such that $h(y) \leq e^{\scalar{a,y}+b}$, then
\[
\int h(x+y) d\gamma_{\Sigma_t}(y) dy \leq \int e^{\scalar{a,x+y}+b} d\gamma_{\Sigma_t}(y) = e^{\scalar{a,x}+b} \int e^{\scalar{a,\Sigma_t^{1/2} y}} d\gamma(y) = e^{\scalar{a,x}+b + \frac{1}{2} \scalar{\Sigma_t a,a}} . 
\]
Denoting $C = e^{b + \frac{1}{2} |a|^2}$, the second assertion is established for small-enough $t > 0$. 
\end{proof}

\subsection{Determinants}

\begin{lemma} \label{lem:det-stability}
For all $A,B \in \Sym_{\geq 0}(n)$ we have:
\[
\frac{1}{(1+\norm{A-B}_{op})^n} \leq \frac{\det(\Id_n + A)}{\det(\Id_n + B)} \leq (1+\norm{A-B}_{op})^n . 
\]
\end{lemma}
\begin{proof}
By exchanging the roles of $A,B$, it is enough to establish the second inequality. Indeed, since $A \leq B + \norm{A-B}_{op} \Id_n$ we have
\[
\det(\Id_n + A) \leq \det(\Id_n + B + \norm{A-B}_{op} \Id_n) \leq \det((1+\norm{A-B}_{op}) (\Id_n + B)) . 
\]
\end{proof}

\begin{lemma} \label{lem:preparatory-det}
If $A,D \in \Sym_{\geq 0}(n)$ then
\[
D \geq \Id_n \;\; \Rightarrow \;\; \det(\Id_n + D A D) \geq \det(\Id_n + A) . 
\]
\end{lemma}
\begin{proof}
By choosing an appropriate orthonormal basis we may assume that $D$ is diagonal. 
It is well-known that  $\det(\Id_n + A) = \sum_{S \subset \{1,\ldots,n\}} \det(A_S)$, where $A_S$ denotes the principal minor of $A$ indexed by $S$. 
Using the fact that $\det( (D A D)_S) = \det(D_S A_S D_S) = \Pi_{i \in S} d_i^2 \det(A_S) \geq \det(A_S)$ since $D = \text{diag}(d_1,\ldots,d_n)$ is diagonal with $d_i \geq 1$ and $A_S \geq 0$, the assertion readily follows. 
\end{proof}

\subsection{Quadratic forms}

We will constantly use the following obvious parallelogram identity:
\begin{lemma}
Let $E$ be a Euclidean space, and let $Q \in \Sym(E)$. Then for all $x,y \in E$:
\begin{equation} \label{eq:parallel} 
\scalar{Q x , x} + \scalar{Q y, y} = \scalar{Q \,\frac{x+y}{\sqrt{2}} , \frac{x+y}{\sqrt{2}}} + \scalar{Q \,\frac{x-y}{\sqrt{2}} , \frac{x-y}{\sqrt{2}}} . 
\end{equation}
\end{lemma}

\begin{definition}[Barycenter and covariance]
Given $f \in L^1(\R^n, \R_+)$ with finite second moments which is not identically zero, 
define its barycenter and covariance by
\[
\bary(f) := \E(X) \in \R^n ~,~ \Cov(f) := \E(X \otimes X) - \E(X) \otimes \E(X) \in \Sym_{\geq 0}(n) ,
\]
where $X$ is a random vector in $\R^n$ having density $\frac{1}{\int f(y) dy} f$. 
\end{definition}

\begin{proposition} \label{prop:cov}
Let $f : \R^n \rightarrow \R_+$ be log-concave, and let $\Sigma \in \Sym_{>0}(n)$. 
\begin{itemize}
\item $\Cov (f \gamma_{\Sigma}) \leq \Sigma$. 
\item $\Cov(f \gamma_{\Sigma}) = \Sigma$ if and only if $f$ is (non-zero) constant. 
\end{itemize}
\end{proposition}
\begin{proof} 
Note that $f \gamma_{\Sigma}$ is integrable by Corollary \ref{cor:log-concave-integrable}. 
The first assertion is a well-known consequence of the classical variance Brascamp--Lieb inequality \cite[Theorem 4.1]{BrascampLiebPLandLambda1} (at least for smooth $f$, the general case easily follows by an approximation argument, cf. \cite[Theorem 4]{HJS-ExtremalPropertyOfGaussians}). The second assertion is already more delicate, and was noted in \cite[Theorem 4]{HJS-ExtremalPropertyOfGaussians} based on \cite[Corollary 2.1]{ChenLou-CharacterizationOfGaussianViaPoincare}. 
\end{proof}

\subsection{Convolution}

\begin{lemma} \label{lem:silly}
Let $f_i \in L^1(\R^n,\R_+)$, $i=1,2$. If $f_i$ is more log-concave (respectively, more log-convex) than $g_{A_i}$,  $A_i \in \Sym_{\geq 0}(n) \cup \{\infty\}$ and $i=1,2$, then:
\begin{enumerate}[(i)]
\item \label{it:part1} For all $x \in \R^n$, $\R^n \ni y \mapsto f_1(\frac{x+y}{\sqrt{2}}) f_2(\frac{x-y}{\sqrt{2}})$ is more log-concave (respectively, more log-convex) than $g_{\frac{A_1+A_2}{2}}$.
\item \label{it:part1.5} For all invertible linear maps $L : \R^n \rightarrow \R^n$, $L_{\#} f_i := f_i \circ L^{-1}$ is more log-concave (respectively, more log-convex) than $g_{(L^{-1})^* A_i L^{-1}}$. 
\item \label{it:part2} If $A_1 + A_2 > 0$ then $f_1 \ast f_2$ is more log-concave (respectively, more log-convex) than $g_{C}$ for $C^{-1} = A_1^{-1} + A_2^{-1}$ (interpreted as $C = A_1 - A_1 (A_1 + A_2)^{-1} A_1 \geq 0$ when either $A_i$ is degenerate). 
In particular (when $A_2 = \infty$), if  $f_1$ is more log-convex than $g_{A_1}$, then without any further assumptions on $f_2$, $f_1 \ast f_2$ is also more log-convex than $g_{A_1}$. 
\end{enumerate}
\end{lemma}
\begin{proof}
When $A_i \in \Sym_{>0}(n) \cup \{\infty\}$ this is well-known, see e.g.~\cite[Lemma 2.1]{EMilman-GCI}. The exact same proof of parts (\ref{it:part1}) and (\ref{it:part1.5}) applies when $A_i \in \Sym_{\geq 0}(n)$ is degenerate. For part (\ref{it:part2}) with the interpretation $C = A_1 - A_1 (A_1 + A_2)^{-1} A_1$, this was verified in the first arXiv version of \cite[Lemma 2.1]{EMilman-GCI}. 
Note that $A_1 =  A_1 (A_1+A_2)^{-1}(A_1 + A_2)$ and hence $C = A_1 (A_1+A_2)^{-1} A_2$, so when both $A_i$ are nondegenerate we have $C^{-1} = A_2^{-1} (A_1 + A_2) A_1^{-1} = A_2^{-1} + A_1^{-1}$, and therefore $C \geq 0$ also for general $A_1 + A_2 > 0$. 
\end{proof}

\begin{proposition}\label{p:BLTrick}
    Let $(\mathbf{c},\mathbf{d},\mathbf{L},\mathcal{Q} ,\boldsymbol{\Sigma},\boldsymbol{\Gamma},\A,\B) $  be a Brascamp--Lieb datum.
         Given $f_i \in \F_{A_i^\flat,A_i^\sharp}(\gamma_{\Sigma_i})$ and $h_j \in \F_{B_j^\flat,B_j^\sharp}(\gamma_{\Gamma_j} )$, define 
    \begin{align*}
    &f_i^{(1)}(\overline{x}_i):= \int_{\R^{n_i}} f_i\brac{ \frac{\overline{x}_i+x_i}{\sqrt2} } f_i \brac{ \frac{\overline{x}_i-x_i}{\sqrt2} }\, d\gamma_{\Sigma_i}(x_i) ~,~ \overline{x}_i \in \R^{n_i} ,
    \\
    &h_j^{(1)}(\overline{y}_j) := 
    \int_{\R^{m_j}} h_j \brac{\frac{\overline{y}_j+y_j}{\sqrt{2}}} h_j \brac{\frac{\overline{y}_j-y_j}{\sqrt{2}}}\, d\gamma_{\Gamma_j}(y_j) ~,~ \overline{y}_j \in \R^{m_j} . 
    \end{align*}
    Then:
    \begin{enumerate}
    \item \label{it:BLTrick1}
    $f_i^{(1)} \in \F_{ A_i^\flat,A_i^\sharp }(\gamma_{\Sigma_i})$ and $h_j^{(1)} \in \F_{B_j^\flat,B_j^\sharp }(\gamma_{\Gamma_j})$. 
    
    \item \label{it:BLTrick2} 
    $\int_{\R^{n_i}} f_i^{(1)}\, d\gamma_{\Sigma_i} = \big( \int_{\R^{n_i}} f_i\, d\gamma_{\Sigma_i} \big)^2$ and $\int_{\R^{m_j}} h_j^{(1)}\, d\gamma_{\Gamma_j} = \big( \int_{\R^{m_j}} h_j\, d\gamma_{\Gamma_j} \big)^{2}$. 
    
    \item \label{it:BLTrick3} 
    We have
    \[
    \bary( f_i^{(1)} \gamma_{\Sigma_i} ) = \sqrt{2} \, \bary( f_i \gamma_{\Sigma_i} ) ~,~ \Cov(f_i^{(1)} \gamma_{\Sigma_i}) = \Cov(f_i \gamma_{\Sigma_i})  ,
    \] 
    and similarly for $h_j$. 
   
    \item \label{it:BLTrick4} 
    For any $D > 0$, if for all $x \in E$
    \begin{equation} \label{eq:ConvClosureAssumption}
    \prod_{i=1}^I f_i(x_i)^{c_i} \le D \cdot e^{-\frac12 \langle \Q x, x\rangle}\prod_{j=1}^J h_j(L_jx)^{d_j} 
    \end{equation}
    then for all $\bar x \in E$
    \begin{equation}\label{eq:ConvClosure}
    \prod_{i=1}^I f_i^{(1)}(\overline{x}_i)^{c_i}\le D^2 \cdot \FR_{\A,\B}^{(e)} e^{-\frac12\langle \Q \overline{x},\overline{x}\rangle } \prod_{j=1}^J h_j^{(1)}(L_j\overline{x})^{d_j} .
    \end{equation}
    \end{enumerate}
\end{proposition}

\begin{proof}
It is enough to prove the first assertion for $h_j$. 
Assume first that $\Gamma_j \in \Sym_{>0}(m_j) \cup \{\infty\}$. Applying (\ref{eq:parallel}), write:
\begin{align*}
h_j^{(1)}(\overline{y}_j) \gamma_{\Gamma_j}(\overline{y}_j) &  = 
        \int_{\R^{m_j}} h_j \brac{\frac{\overline{y}_j+y_j}{\sqrt{2}}} h_j \brac{\frac{\overline{y}_j-y_j}{\sqrt{2}}}\,  \gamma_{\Gamma_j}(\overline{y}_j) \gamma_{\Gamma_j}(y_j) d y_j \\     & = 
            \int_{\R^{m_j}} (h_j \gamma_{\Gamma_j})\brac{\frac{\overline{y}_j+y_j}{\sqrt{2}}} (h_j \gamma_{\Gamma_j}) \brac{\frac{\overline{y}_j-y_j}{\sqrt{2}}}\,  d y_j \\
        & = 2^{m_j/2}  (h_j \gamma_{\Gamma_j} \ast h_j \gamma_{\Gamma_j}) (\sqrt{2} \overline{y}_j ) .
\end{align*}
Since $h_j \gamma_{\Gamma_j} \in \F_{B_j^{\flat} + \Gamma_j^{-1}, B_j^{\sharp} + \Gamma_j^{-1}}(dy)$, Lemma \ref{lem:silly} implies that $h_j^{(1)} \gamma_{\Gamma_j} \in \F_{B_j^{\flat} + \Gamma_j^{-1}, B_j^{\sharp} + \Gamma_j^{-1}}(dy)$, and hence $h_j^{(1)} \in \F_{B_j^{\flat} , B_j^{\sharp}}(\gamma_{\Gamma_j})$. Note that this argument applies equally well to the case when $\Gamma_j = \infty$, in which case, recall, we interpret $\gamma_{\Gamma_j} = 1$ and $\Gamma_j^{-1} = 0$. 

The general case for degenerate $\Gamma_j \geq 0$ follows by approximating $\Gamma_j$ by $\Gamma_j^{(t)} \in \Sym_{>0}(m_j)$ (say, as $t \rightarrow 1^-$). Denoting:
\begin{equation} \label{eq:BLTrick}
h_j^{(t)}(\overline{y}_j) := 
    \int_{\R^{m_j}} h_j \brac{\frac{\overline{y}_j+y_j}{\sqrt{2}}} h_j \brac{\frac{\overline{y}_j-y_j}{\sqrt{2}}}\, d\gamma_{\Gamma^{(t)}_j}(y_j) ,
\end{equation}
we have $h_j^{(t)} \in \F_{B_j^{\flat} , B_j^{\sharp}}(\gamma_{\Gamma^{(t)}_j})$ for all $t < 1$. It is well-known that $\gamma_{\Gamma^{(t)}_j}$ weakly converges to $\gamma_{\Gamma_j}$ (e.g. by pointwise converge of the corresponding characteristic functions), and so by a version of the Portmanteau theorem \cite[Theorem 2.7]{BillingsleyConvergenceBook-2ndEd}, 
we have $\int \Psi_j d\gamma_{\Gamma^{(t)}_j} \rightarrow \int \Psi_j d\gamma_{\Gamma_j}$ for any bounded function $\Psi_j : \R^{m_j} \rightarrow \R$ so that the set of its discontinuities has zero $\gamma_{\Gamma_j}$ measure. The latter measure is supported on some linear subspace $F_j$, and since the integrand of (\ref{eq:BLTrick}) is even and log-concave in $y_j$, it is bounded by its value at the origin and continuous on the (non-empty) interior of its convex origin-symmetric support $K_{\overline{y}_j} \ \subset \R^{m_j}$. Therefore, its set of discontinuities is precisely $\partial K_{\overline{y}_j}$, and since $\partial K_{\overline{y}_j} \cap F_j = \partial_{F_j} (K_{\overline{y}_j} \cap F_j)$ by origin-symmetry (where $\partial_{F_j}$ denotes the boundary operator in $F_j$), and the latter set is a null-set in $F_j$ since $K_{\overline{y}_j} \cap F_j$ is convex, the preceding convergence applies. It follows that $h_j^{(t)} \rightarrow h_j^{(1)}$ pointwise as $t \rightarrow 1^-$, and since being more log-concave / log-convex is a property about three points and hence closed under pointwise limits, it follows that $h_j^{(1)} \in \F_{B_j^{\flat} , B_j^{\sharp}}(\gamma_{\Gamma_j})$.

To see the second assertion, we apply a change of variables $x_i^\pm = \frac{\overline{x}_i \pm x_i}{\sqrt2}$ and write
\begin{align*}
    \int_{\R^{n_i}} f_i^{(1)}(\overline{x}_i)\, d\gamma_{\Sigma_i}
    &= 
    \int_{\R^{n_i}\times \R^{n_i}} 
    f_i\brac{ \frac{\overline{x}_i +x_i}{\sqrt2} } f_i \brac{\frac{\overline{x}_i - x_i}{\sqrt2} }\, d\gamma_{\Sigma_i}(x_i) d\gamma_{\Sigma_i}(\overline{x}_i)\\
    &= 
    \int_{\R^{n_i}\times \R^{n_i}} 
    f_i( x_i^+ ) f_i( x_i^-  )\, d\gamma_{\Sigma_i}\brac{\frac{x_i^++x_i^-}{\sqrt{2}}} d\gamma_{\Sigma_i}\brac{\frac{x_i^+-x_i^-}{\sqrt{2}}} \\
    &= 
    \int_{\R^{n_i} \times \R^{n_i}} f_i( x_i^+ ) f_i( x_i^-  )\, d\gamma_{\Sigma_i}(x_i^+)d\gamma_{\Sigma_i}(x_i^-) = \brac{\int_{\R^{n_i}} f_i d\gamma_{\Sigma_i}}^2 ,
\end{align*}
using \eqref{eq:parallel} in the penultimate transition. The argument for $h_j^{(1)}$ is identical (even when $\Gamma_j$ is degenerate). 

For the third assertion, compute:
    \begin{align*}
        \int_{\R^{n_i}} \overline{x}_i f_i^{(1)}(\overline{x}_i)\, d\gamma_{\Sigma_i}(\overline{x}_i)
        &= 
        \int_{\R^{n_i}\times \R^{n_i}} \overline{x}_i 
        f_i\brac{\frac{\overline{x}_i+x_i}{\sqrt2} }f_i\brac{\frac{\overline{x}_i-x_i}{\sqrt2} }\, d\gamma_{\Sigma_i}(x_i)d\gamma_{\Sigma_i}(\overline{x}_i) \\
        &= 
        \int_{\R^{n_i} \times \R^{n_i}} \frac{x_i^++x_i^-}{\sqrt2} f_i(x_i^+)f_i(x_i^-)\, d\gamma_{\Sigma_i}(x_i^+)d\gamma_{\Sigma_i}(x_i^-)  \\
        & = \sqrt{2} \int_{\R^{n_i}} x_i f_i(x_i) d\gamma_{\Sigma_i}(x_i) \int_{\R^{n_i}} f_i d\gamma_{\Sigma_i} . 
    \end{align*}
An identical argument verifies that
\begin{align*}
    \int_{\R^{n_i}} \overline{x}_i \otimes \overline{x}_i  f_i^{(1)}(\overline{x}_i)\, d\gamma_{\Sigma_i}(\overline{x}_i) & = \int_{\R^{n_i}} x_i \otimes x_i  f_i(x_i) d\gamma_{\Sigma_i}(x_i) \int f_i(x_i) d\gamma_{\Sigma_i}(x_i)  \\
    & + \int_{\R^{n_i}} x_i f_i(x_i) d\gamma_{\Sigma_i}(x_i) \otimes \int_{\R^{n_i}} x_i f_i(x_i) d\gamma_{\Sigma_i}(x_i) . 
\end{align*}
Together with the second assertion and the first part of the third one, this concludes the proof of part (\ref{it:BLTrick3}). 

For the last assertion, denote 
\[
f_i^{(\overline{x}_i)}(x_i):= f_i\brac{\frac{\overline{x}_i+x_i}{\sqrt2}} f_i \brac{\frac{\overline{x}_i-x_i}{\sqrt2}},
\quad 
h_j^{(\overline{y}_j)}(y_j)
:= 
h_j\brac{ \frac{\overline{y}_j+y_j}{\sqrt{2}}} h_j \brac{ \frac{\overline{y}_j - y_j}{\sqrt{2}} }
\]
for each fixed $\overline{x}\in \bigoplus_{i=1}^I \R^{n_i}$ and $\overline{y} \in \bigoplus_{j=1}^J \R^{m_j}$.  By (\ref{eq:ConvClosureAssumption}), we have 
\begin{align*}
    \prod_{i=1}^I f_i^{(\overline{x}_i)}(x_i)^{c_i}
    &= 
    \prod_{i=1}^I f_i\brac{ \frac{\overline{x}_i +x_i}{\sqrt2} }^{c_i} f_i\brac{\frac{\overline{x}_i - x_i}{\sqrt2}}^{c_i}\\
    &\le 
    D^2 \cdot e^{-\frac12 \langle \Q \frac{\overline{x} + x }{\sqrt2}, \frac{ \overline{x} + x }{\sqrt2}\rangle }
    e^{-\frac12 \langle \Q  \frac{ \overline{x} - x }{\sqrt2}, \frac{ \overline{x} - x }{\sqrt2}\rangle }
    \prod_{j=1}^J 
    h_j\brac{ L_j \brac{ \frac{\overline{x}+x}{\sqrt2} } }^{d_j} h_j \brac{L_j \brac{\frac{\overline{x}-x}{\sqrt2} } }^{d_j} \\
    &= 
    D^2 \cdot e^{-\frac12 \langle \Q \overline{x},\overline{x}\rangle}
    e^{-\frac12 \langle \Q x, x \rangle}
    \prod_{j=1}^J h_j^{(L_j\overline{x})}(L_jx)^{d_j} .
\end{align*}
Since we fixed $\overline{x}$, we regard $e^{-\frac12 \langle \Q \overline{x},\overline{x}\rangle}$ as a constant and write 
$$
\prod_{i=1}^I f_i^{(\overline{x}_i)}(x_i)^{c_i}
\le 
    e^{-\frac12 \langle \Q  x,x \rangle}
    \prod_{j=1}^J \widetilde{h}_j^{(L_j\overline{x})}(L_jx)^{d_j},
$$
where $\widetilde{h}_j^{(L_j\overline{x})}(y_j):= \big(D^2  e^{-\frac12\langle \Q \overline{x},\overline{x}\rangle}\big)^{1/(d_jJ)}h_j^{(L_j\overline{x})}(y_j)$. 
By Lemma \ref{lem:silly}, $f_i^{(\overline{x}_i)} \in \F^{(e)}_{A_i^\flat,A_i^\sharp}(\gamma_{\Sigma_i})$ and $\widetilde{h}_j^{(L_j\overline{x})} \in \F^{(e)}_{B_j^\flat,B_j^\sharp}(\gamma_{\Gamma_j})$. Therefore, 
\begin{align*}
\prod_{i=1}^I \big(\int_{\R^{n_i}} f_i^{(\overline{x}_i)}\, d\gamma_{\Sigma_i}(x_i)\big)^{c_i}
&\le 
{\rm FR}_{\A,\B}^{(e)} 
\prod_{j=1}^J \big(
\int_{\R^{m_j}} \widetilde{h}_j^{(L_j\overline{x})}\, d\gamma_{\Gamma_j}(y_j)
\big)^{d_j}\\
&= 
D^2 \cdot {\rm FR}_{\A,\B} ^{(e)} e^{-\frac12 \langle \Q \overline{x},\overline{x}\rangle}
\prod_{j=1}^J \big(
\int_{\R^{m_j}} {h}_j^{(L_j\overline{x})}\, d\gamma_{\Gamma_j}(y_j)
\big)^{d_j}. 
\end{align*}
Recalling the definitions of $f_i^{(1)}$ and $h_j^{(1)}$, this verifies (\ref{eq:ConvClosure}). 
\end{proof}

\subsection{Finiteness of $\FR^{(o)}_{\A,\B}$}

\begin{lemma}\label{l:Finiteness}
Let $(\mathbf{c},\mathbf{d},\mathbf{L},\mathcal{Q} ,\boldsymbol{\Sigma},\boldsymbol{\Gamma},\A,\B) $  be a Brascamp--Lieb datum.
Assume that $\Sigma_i \in \Sym_{>0}(n_i) \cup \{\infty\}$, $\Gamma_j \in \Sym_{>0}(m_j) \cup \{\infty\}$, $A^\flat_i + \Sigma_i^{-1} > 0$, $B^\flat_j + \Gamma_j^{-1} > 0$ and that $A_i^\sharp, B_j^\sharp < \infty$, for all $i=1,\ldots,I$ and $j=1,\ldots,J$. 
Then $\FR^{(o)}_{\A,\B} < \infty$.  
\end{lemma}

\begin{proof}
Let $f_i\in \F_{A_i^\flat,A_i^\sharp}^{(o)}(\gamma_{\Sigma_i})$ and $h_j\in \F_{B_j^\flat,B_j^\sharp}(\gamma_{\Gamma_j})$ satisfy 
\begin{equation} \label{eq:finite1}
\prod_{i=1}^I f_i(x_i)^{c_i} \le e^{-\frac12\langle \Q x, x\rangle} \prod_{j=1}^J h_j(L_jx)^{d_j} . 
\end{equation}
Note that our assumptions ensure that $f_i\gamma_{\Sigma_i} \in \F^{(o)}_{2 \varepsilon \Id_{n_i},\frac{R}{2} \Id_{n_i}}(dx_i)$ and $h_j \gamma_{\Gamma_j} \in \mathcal{F}_{2 \varepsilon \Id_{m_j},\frac{R}{2} \Id_{m_j}}(dy_j)$ for some $0< \varepsilon\le R <\infty$.     This yields the following pointwise control (see e.g.~\cite[Lemma 3.2]{NakamuraTsuji-GCIForCentered}), 
    \begin{align*}
        e^{-n_i} f_i(0)\gamma_{\Sigma_i}(0) e^{-\frac{R}2|x_i|^2}
        & \le f_i(x_i)\gamma_{\Sigma_i}(x_i) \le e^{2n_i} f_i(0)\gamma_{\Sigma_i}(0) e^{-\frac{\varepsilon}2|x_i|^2},\\
        e^{-m_j} h_j(b_j)\gamma_{\Gamma_j}(b_j) e^{-\frac{R}2|y_j-b_j|^2}
        & \le h_j(y_j) \gamma_{\Gamma_j}(y_j) \le e^{2 m_j} h_j(b_j) \gamma_{\Gamma_j}(b_j) e^{-\frac{\varepsilon}2|y_j-b_j|^2},
    \end{align*}
    where $\vec b_j:= \int_{\R^{m_j}} \vec y_j h_j\, d\gamma_{\Gamma_j}/\int_{\R^{m_j}} h_j\, d\gamma_{\Gamma_j}$ denotes the barycenter of $h_j \gamma_{\Gamma_j}$ (recall that the barycenter of $f_i \gamma_{\Sigma_i}$ is at the origin). 
    From this, 
    \begin{align*}
    \frac{ \prod_{i=1}^I \big( \int f_i\, d\gamma_{\Sigma_i} \big)^{c_i} }{ \prod_{j=1}^J \big( \int h_j\, d\gamma_j \big)^{d_j} }
    &\le 
    \frac{ \prod_{i=1}^I \big( e^{2n_i} f_i(0) \gamma_{\Sigma_i}(0) \int e^{-\frac{\varepsilon}2|x_i|^2}\, dx_i \big)^{c_i} }{ \prod_{j=1}^J \big( e^{-m_j} h_j(b_j)\gamma_{\Gamma_j}(b_j) \int e^{-\frac{R}2|y_j-b_j|^2}\, dy_j \big)^{d_j} } \\
    &= 
    C(\mathbf{c},\mathbf{d},\mathbf{n},\mathbf{m},\varepsilon,R)\frac{\prod_{i=1}^I f_i(0)^{c_i} \gamma_{\Sigma_i}(0)^{c_i} }{ \prod_{j=1}^J h_j(b_j)^{d_j} \gamma_{\Gamma_j}(b_j)^{d_j} }. 
    \end{align*}
    On the other hand, combining the pointwise estimates with (\ref{eq:finite1}), 
    \begin{align*}
        &\prod_{i=1}^I\bigg( e^{-n_i} f_i(0)\gamma_{\Sigma_i}(0) e^{-\frac{R}2|x_i|^2} \bigg)^{c_i} \le \prod_{i=1}^I f_i(x_i)^{c_i} \gamma_{\Sigma_i}(x_i)^{c_i}\\
        &\le 
        e^{-\frac12 \langle \Q x,x\rangle}
        \prod_{j=1}^J h_j(L_jx)^{d_j} \gamma_{\Gamma_j}(L_jx)^{d_j} \frac{\prod_{i=1}^I \gamma_{\Sigma_i}(x_i)^{c_i}}{\prod_{j=1}^J\gamma_{\Gamma_j}(L_jx)^{d_j}} \\
        &\le 
        e^{-\frac12 \langle \Q x,x\rangle}
        \prod_{j=1}^J \bigg( e^{2 m_j} h_j(b_j) \gamma_{\Gamma_j}(b_j) e^{- \frac{\varepsilon}2|L_jx-b_j|^2} \bigg)^{d_j}\frac{\prod_{i=1}^I \gamma_{\Sigma_i}(x_i)^{c_i}}{\prod_{j=1}^J\gamma_{\Gamma_j}(L_jx)^{d_j}}. 
    \end{align*}
    In particular, at $x =0$ this yields  
    \begin{align*}
        &\prod_{i=1}^I\bigg( e^{-n_i} f_i(0)\gamma_{\Sigma_i}(0)  \bigg)^{c_i}
        \le 
        \prod_{j=1}^J \bigg( e^{2 m_j} h_j(b_j) \gamma_{\Gamma_j}(b_j) e^{- \frac{\varepsilon}2|b_j|^2} \bigg)^{d_j}\frac{\prod_{i=1}^I \gamma_{\Sigma_i}(0)^{c_i}}{\prod_{j=1}^J\gamma_{\Gamma_j}(0)^{d_j}}. 
    \end{align*}
    Rearranging, we deduce 
    $$
    \frac{\prod_{i=1}^I f_i(0)^{c_i} \gamma_{\Sigma_i}(0)^{c_i} }{\prod_{j=1}^J h_j(b_j)^{d_j} \gamma_{\Gamma_j}(b_j)^{d_j}}
    \le 
    C'(\mathbf{c},\mathbf{d},\mathbf{n},\mathbf{m},\boldsymbol{\Sigma}, \boldsymbol{\Gamma}) e^{-\frac{\varepsilon}2\sum_j d_j|b_j|^2}.
    $$
    Combining everything, we conclude that 
    \begin{align*}
        \frac{ \prod_{i=1}^I \big( \int f_i\, d\gamma_{\Sigma_i} \big)^{c_i} }{ \prod_{j=1}^J \big( \int h_j\, d\gamma_j \big)^{d_j} }
        &\le 
        C C' e^{-\frac{\varepsilon}2\sum_j d_j|b_j|^2}\le C C' < \infty, 
    \end{align*}
   concluding the proof. 
\end{proof}

\subsection{Reduction to the even case}

By using the last lemma and an argument of Courtade--Wang \cite{CourtadeWang-BlaschkeSantalo}, we may reduce to the case that $f_i$ and $h_j$ are even. 
\begin{proposition}\label{p:reduce-to-even}
Let $(\mathbf{c},\mathbf{d},\mathbf{L},\mathcal{Q} ,\boldsymbol{\Sigma},\boldsymbol{\Gamma},\A,\B) $  be a Brascamp--Lieb datum.
Assume that both $\Sigma_i > 0$ and $\Gamma_j > 0$ are nondegenerate and that $A_i^\sharp, B_j^\sharp < \infty$, for all $i=1,\ldots,I$ and $j=1,\ldots,J$. 
Then $\FR^{(o)}_{\A,\B}  = \FR_{\A,\B}^{(e)}$.  
\end{proposition}

\begin{proof}
     Assume that 
    \begin{equation}
    \label{eq:reduce-to-even-assumption}
        f_i\in \F_{A_i^\flat,A_i^\sharp}^{(o)}(\gamma_{\Sigma_i}),\; h_j\in \F_{B_j^\flat,B_j^\sharp}(\gamma_{\Gamma_j}) ,\;  \prod_{i=1}^I f_i(x_i)^{c_i} \le e^{-\frac12 \langle \Q x, x\rangle} \prod_{j=1}^J h_j(L_jx)^{d_j} \;\; \forall x \in E .
      \end{equation}
    Recall the definitions of $f_i^{(1)}$ and $h_j^{(1)}$ in Proposition \ref{p:BLTrick}, which we apply below throughout. By part (\ref{it:BLTrick4}), we have the pointwise bound 
    $$
    \prod_{i=1}^I f_i^{(1)}(\overline{x}_i) 
    \le 
    e^{-\frac12 \langle \Q \overline{x},\overline{x}\rangle} 
    \prod_{j=1}^J \widetilde{h}_j^{(1)}(L_j\overline{x})^{d_j},\quad {\rm where} \quad 
    \widetilde{h}_j^{(1)}:= (\FR_{\A,\B}^{(e)})^{ 1/(d_j J)}h_j^{(1)}.
    $$
    By part (\ref{it:BLTrick3}), since $\int x_i f_i(x_i) \, d\gamma_{\Sigma_i}(x_i) =0$ we have $\int_{\R^{n_i}} \overline{x}_i f_i^{(1)}(\overline{x}_i)\, d\gamma_{\Sigma_i}(\overline{x}_i) = 0$. By part (\ref{it:BLTrick1}), we deduce that $f_i^{(1)} \in \F^{(o)}_{A_i^\flat, A_i^\sharp}(\gamma_{\Sigma_i})$ and $\widetilde{h}_j^{(1)} \in \F_{B_j^\flat,B_j^\sharp}(\gamma_{\Gamma_j})$. 
    
    It follows by the definition of $\FR^{(o)}_{\A,\B}$ that
    \begin{align*}
    \prod_{i=1}^I \big( \int_{\R^{n_i}} f_i^{(1)}\, d\gamma_{\Sigma_i} \big)^{c_i}
    &\le 
    \FR^{(o)}_{\A,\B} \prod_{j=1}^J \big(\int_{\R^{m_j}} \widetilde{h}_j^{(1)}\, d\gamma_{\Gamma_j} \big)^{d_j}\\
    &=
    \FR^{(o)}_{\A,\B} \FR^{(e)}_{\A,\B} \prod_{j=1}^J \big(\int_{\R^{m_j}} {h}_j^{(1)}\, d\gamma_{\Gamma_j} \big)^{d_j}. 
    \end{align*}
    It follows by part (\ref{it:BLTrick2}) that:
    $$
    \prod_{i=1}^I \big( \int_{\R^{n_i}} f_i\, d\gamma_{\Sigma_i} \big)^{c_i}
    \le 
    \sqrt{\FR^{(o)}_{\A,\B} \FR^{(e)}_{\A,\B}}
    \prod_{j=1}^J \big(\int_{\R^{m_j}} {h}_j\, d\gamma_{\Gamma_j} \big)^{d_j} .
    $$
    Since this holds for all $f_i$ and $h_j$ satisfying (\ref{eq:reduce-to-even-assumption}), the definition of $\FR^{(o)}_{\A,\B}$ implies that $\FR^{(o)}_{\A,\B} \le \sqrt{\FR^{(o)}_{\A,\B} \FR^{(e)}_{\A,\B}}$.
    Since $\FR^{(o)}_{\A,\B}<\infty$ thanks to Lemma \ref{l:Finiteness}, we deduce that $\FR^{(o)}_{\A,\B} \le\FR^{(e)}_{\A,\B}$. The other direction is trivial, thereby concluding the proof. 
\end{proof}

\section{Proof of the Gaussian Forward-Reverse Brascamp--Lieb Theorem} \label{sec:FRBL-proof}

Let $(\mathbf{c},\mathbf{d},\mathbf{L},\mathcal{Q} ,\boldsymbol{\Sigma},\boldsymbol{\Gamma},\A,\B) $  be a Brascamp--Lieb datum.
The proof of Theorem \ref{thm:intro-FRBL}, and more generally, Theorem \ref{thm:FRBL}, is divided into 5 steps: 
\begin{enumerate}[Step 1.]
	\item \label{it:step0} (Gaussian saturation for even regularized inputs and Lebesgue measure) 
	For the case of $\gamma_{\Sigma_i} = dx_i$ and $\gamma_{\Gamma_j} = dy_j$,  $0< A_i^\flat \le A_i^\sharp < \infty$ and $0< B_j^\flat \le B_j^\sharp <\infty$, show that $\FR^{(e)}_{\A,\B} = \FR^{(\G)}_{\A,\B}$.
	\item \label{it:step1} (Replace Lebesgue measure by nondegenerate Gaussian measures) 
	For the case of $\Sigma_i\in \Sym_{>0}(n_i)$, $\Gamma_j \in \Sym_{>0}(m_j)$, $0\le A_i^\flat \le A_i^\sharp < \infty$ and $0\le B_j^\flat \le B_j^\sharp <\infty$, show that $\FR^{(e)}_{\A,\B} = \FR^{(\G)}_{\A,\B}$.
	\item \label{it:step2} (Remove evenness assumption)
	For the case of $\Sigma_i\in \Sym_{>0}(n_i)$, $\Gamma_j \in \Sym_{>0}(m_j)$, $0\le A_i^\flat \le A_i^\sharp < \infty$ and $0\le B_j^\flat \le B_j^\sharp <\infty$, show that $\FR^{(o)}_{\A,\B} = \FR^{(\G)}_{\A,\B}$.
	\item \label{it:step3} (Approximate degenerate $\gamma_{\Gamma_j}$'s by nondegenerate ones) 
	For the case of $\Sigma_i\in \Sym_{>0}(n_i)$, $\Gamma_j \in \Sym_{\geq 0}(m_j)$, $0\le A_i^\flat \le A_i^\sharp < \infty$ and $0\le B_j^\flat \le B_j^\sharp <\infty$, show that $\FR^{(o)}_{\A,\B} = \FR^{(\G)}_{\A,\B}$. 
	\item \label{it:step4} (Remove the regularization)
	For the case of $\Sigma_i\in \Sym_{>0}(n_i)$, $\Gamma_j \in \Sym_{\geq 0}(m_j)$ with $0\le A_i^\flat < A_i^\sharp = \infty$ and $0\le B_j^\flat < B_j^\sharp = \infty$ (and in particular, no regularization), show that $\FR^{(o)}_{\A,\B} = \FR^{(\G)}_{\A,\B}$ (and in particular, $\FR^{(o)}_{LC} = \FR^{(\G)}_{LC}$).  
\end{enumerate}

\subsection{Step \ref{it:step0}}

\begin{proof}[Proof of Step \ref{it:step0}]
	As a first step, we consider the case of $d\gamma_{\Sigma_i} = dx_i$ and $d\gamma_{\Gamma_j} = dy_j$ (so $\Sigma_i = \Gamma_j = \infty$) as well as $0< A_i^\flat\le A_i^\sharp< \infty$ and $0< B_j^\flat\le B_j^\sharp <\infty$. 
	
	By Lemma \ref{l:Finiteness} $\FR^{(e)}_{\A,\B} \leq \FR^{(o)}_{\A,\B} < \infty$ is finite, and by a compactness argument utilizing the Arzel\`a--Ascoli theorem as in the proof of \cite[Theorem 2.1]{NakamuraTsuji-InverseBrascampLieb}, there exist maximizing functions $\mathfrak{f}_i \in \F^{(e)}_{A_i^\flat, A_i^\sharp }(dx_i)$ and $\mathfrak{h}_j \in \mathcal{F}^{(e)}_{B_j^\flat, B_j^\sharp}(dy_j)$ for which equality occurs in (\ref{eq:FR-conclusion}). 
	Normalizing these so that $\int \mathfrak{f}_i(x_i) dx_i = \int \mathfrak{h}_j(y_j) dy_j = 1$, we have by Lemma \ref{lem:equivalence}:	
	\[
	\prod_{i=1}^I \mathfrak{f}_i(x_i)^{c_i} \le \frac{1}{\FR^{(e)}_{\A,\B}} e^{-\frac12 \langle \Q x, x\rangle} \prod_{j=1}^J \mathfrak{h}_j(L_jx)^{d_j} \;\;\; \forall x \in E . 
	\]
	
	Now define inductively:
		\begin{align*}
	&\mathfrak{f}_i^{(N+1)}(\overline{x}_i):= \int_{\R^{n_i}} \mathfrak{f}_i^{(N)}\brac{\frac{\overline{x}_i + x_i}{\sqrt2}} \mathfrak{f}_i^{(N)}\brac{\frac{\overline{x}_i - x_i}{\sqrt2}} \, dx_i = 2^{n_i/2}  (\mathfrak{f}^{(N)}_i \ast \mathfrak{f}^{(N)}_i) (\sqrt{2} \overline{x}_i ) , \\
	&\mathfrak{h}_j^{(N+1)}(\overline{y}_j) := \int_{\R^{m_j}} \mathfrak{h}_j^{(N)}\brac{\frac{\overline{y}_j + y_j}{\sqrt2}}  \mathfrak{h}_j^{(N)}\brac{\frac{\overline{y}_j - y_j}{\sqrt2}}\, dy_j = 2^{m_j/2}  (\mathfrak{h}^{(N)}_j \ast \mathfrak{h}^{(N)}_j) (\sqrt{2} \overline{y}_j ), 
	\end{align*}
	where $\mathfrak{f}_i^{(0)}:= \mathfrak{f}_i$ and $\mathfrak{h}_j^{(0)}:= \mathfrak{h}_j$, and apply Proposition \ref{p:BLTrick} iteratively. By part (\ref{it:BLTrick1}) we have $\mathfrak{f}_i^{(N)} \in \F^{(e)}_{A_i^\flat, A_i^\sharp}(dx_i)$ and $\mathfrak{h}_j\in \F^{(e)}_{B_j^\flat,B_j^\sharp}(dy_j)$ for all $N \geq 1$. By part (\ref{it:BLTrick2}) of course 
	$\int \mathfrak{f}_i^{(N)}(x_i) dx_i = \int \mathfrak{h}_j^{(N)}(y_j) dy_j = 1$ for all $N \geq 1$. 
	By part (\ref{it:BLTrick4}) applied with $D = \frac{1}{\FR^{(e)}_{\A,\B}}$, we see that for all $N \geq 1$, 
							\begin{equation} \label{eq:step0-N}
		\prod_{i=1}^I \mathfrak{f}_i^{(N)}(\overline{x}_i)^{c_i}
		\le \frac{1}{\FR^{(e)}_{\A,\B}} e^{-\frac12 \langle \Q \overline{x},\overline{x}\rangle} \prod_{j=1}^J \mathfrak{h}_j^{(N)}(L_j\overline{x})^{d_j} \;\;\; \forall \overline{x} \in E . 
	\end{equation}

	Since $f_i$ and $h_j$ are even and log-concave, they attain their maximal value at the origin, and hence are bounded. 
	By a Local version of the Central Limit Theorem (see e.g. \cite[Theorem 19.1]{BhattacharyaRao-Book2010} or \cite{Bobkov-LocalLimitTheoremsInOrliczSpace}), it follows that we have the following uniform (and hence pointwise) convergence:
	\[
	\lim_{N\to \infty}\mathfrak{f}_i^{(N)} = \gamma_{\Cov(\mathfrak{f}_i)} ,\quad 
	\lim_{N\to\infty} \mathfrak{h}_j^{(N)} = \gamma_{\Cov(\mathfrak{h}_j)} .
	\]	
	Since  being more log-concave / log-convex is a property about three points and hence closed under pointwise limits, it follows that $\gamma_{\Cov(\mathfrak{f}_i)} \in \F^{(e)}_{A_i^\flat, A_i^\sharp}(dx_i)$ and $\gamma_{\Cov(\mathfrak{h}_j)} \in \F^{(e)}_{B_j^\flat,B_j^\sharp}(dy_j)$. In addition, the inequality (\ref{eq:step0-N}) is preserved in the limit:
	$$
	\prod_{i=1}^I \gamma_{\Cov(\mathfrak{f}_i)}(\overline{x}_i)^{c_i}
		\le  \frac{1}{\FR^{(e)}_{\A,\B}} e^{-\frac12 \langle \Q \overline{x},\overline{x}\rangle} \prod_{j=1}^J \gamma_{\Cov(\mathfrak{h}_j)}(L_j\overline{x})^{d_j} \;\;\; \forall \bar x \in E. 
	$$
	And of course $\int \gamma_{\Cov(\mathfrak{f}_i)} dx_i = \int  \gamma_{\Cov(\mathfrak{h}_j)} dy_j = 1$.

	Applying Lemma \ref{lem:equivalence} again, it follows that $\FR^{(\G)}_{\A,\B} \geq \FR^{(e)}_{\A,\B}$. As the converse inequality is trivial, we deduce that $\FR^{(e)}_{\A,\B} = \FR^{(\G)}_{\A,\B}$.
\end{proof}

\subsection{Step \ref{it:step1}}

\begin{proof}[Proof of Step \ref{it:step1}] 
Assume now that 
\begin{equation}\label{e:SettingStep1}
	\Sigma_i \in \Sym_{>0}(n_i),\; \Gamma_j \in \Sym_{>0}(m_j), \; 0\le A_i^\flat\le A_i^\sharp <\infty, \; 0\le B_j^\flat\le B_j^\sharp < \infty. 
\end{equation}
Our goal is to prove that $\FR^{(e)}_{\A,\B} = \FR^{(\G)}_{\A,\B}$. 

Let $f_i \in \mathcal{F}^{(e)}_{A_i^\flat,A_i^\sharp}(\gamma_{\Sigma_i})$ and $h_j \in \mathcal{F}^{(e)}_{B_j^\flat, B_j^\sharp} (\gamma_{\Gamma_j})$ satisfy
\begin{equation} \label{eq:step1-inq}
\prod_{i=1}^I f_i(x_i)^{c_i} \le e^{-\frac12 \langle \Q x,x\rangle} \prod_{j=1}^J h_j(L_jx)^{d_j}.
\end{equation}
Since $\Sigma_i$ and $\Gamma_j$ are nondegenerate, we may define:
$$
\mathfrak{f}_i:= f_i g_{\Sigma_i^{-1}},\; \mathfrak{h}_j := h_j g_{\Gamma_j^{-1}},\; e^{-\frac12 \langle \mathfrak{Q} x, x\rangle} := \frac{\prod_{i=1}^I g_{\Sigma_i^{-1}}(x_i)^{c_i}}{\prod_{j=1}^J g_{\Gamma_j^{-1}}(L_jx)^{d_j} }e^{-\frac12\langle \mathcal{Q} x, x\rangle}. 
$$
We then have 
$$
\prod_{i=1}^I \mathfrak{f}_i(x_i)^{c_i} \le e^{-\frac12 \langle \mathfrak{Q} x, x\rangle} \prod_{j=1}^J \mathfrak{h}_j(L_jx)^{d_j},
$$
and 
$$
\int_{\R^{n_i}} \mathfrak{f}_i\, dx_i = {\rm det}\, (2\pi \Sigma_i)^{1/2}\int _{\R^{n_i}} f_i\, d\gamma_{\Sigma_i},\quad 
\int_{\R^{m_j}} \mathfrak{h}_j\, dy_j= {\rm det}\, (2\pi \Gamma_j)^{1/2} \int_{\R^{m_j}} h_j\, d\gamma_{\Gamma_j}. 
$$
Moreover,  $\mathfrak{f}_i \in \mathcal{F}^{(e)}_{ \mathfrak{A}_i^\flat, \mathfrak{A}_i^\sharp }(dx_i)$ and $\mathfrak{h}_j \in \mathcal{F}^{(e)}_{\mathfrak{B}_j^\flat, \mathfrak{B}_j^\sharp}(dy_j)$, where 
$$
\mathfrak{A}_i^{*}:= A_i^{*} + \Sigma_i^{-1}\in {\rm Sym}_{>0}(n_i),\quad 
\mathfrak{B}_j^{*}:= B_j^{*} + \Gamma_j^{-1} \in {\rm Sym}_{>0}(m_j), \;\; * \in \{ \flat,\sharp \} . 
$$
Denoting the Gaussian constant corresponding to the new Brascamp--Lieb datum by 
\[
\mathfrak{FR}^{(\G)}_{\mathfrak{A},\mathfrak{B}} := \FR^{(\G)}( \mathbf{c},\mathbf{d},\mathbf{L}, \infty , \infty,\mathfrak{Q},\mathfrak{A},\mathfrak{B}), 
\] 
we see from Step \ref{it:step0} that 
\begin{align*}
\prod_{i=1}^I \big(\int_{\R^{n_i}} f_i\, d\gamma_{\Sigma_i}\big)^{c_i}
&= 
\prod_{i=1}^I {\rm det}\, (2\pi \Sigma_i)^{-c_i/2}
\prod_{i=1}^I \big(\int_{\R^{n_i}} \mathfrak{f}_i\, dx_i \big)^{c_i} \\
&\le 
\prod_{i=1}^I {\rm det}\, (2\pi \Sigma_i)^{-c_i/2} 
\mathfrak{FR}^{(\mathcal{G})}_{\mathfrak{A},\mathfrak{B}} \prod_{j=1}^J \big( \int_{\R^{m_j}} \mathfrak{h}_j\, dy_j \big)^{d_j} \\
&= 
\frac{\prod_{i=1}^I {\rm det}\, (2\pi \Sigma_i)^{-c_i/2}}{\prod_{j=1}^J {\rm det}\, (2\pi \Gamma_j)^{-d_j/2} }
\mathfrak{FR}^{(\mathcal{G})}_{\mathfrak{A},\mathfrak{B}} \prod_{j=1}^J \big( \int_{\R^{m_j}} {h}_j\, d\gamma_{\Gamma_j} \big)^{d_j}.  
\end{align*}
Recalling (\ref{eq:Gaussian-constant}), it is straightforward to check that 
$$
\frac{\prod_{i=1}^I {\rm det}\, (2\pi \Sigma_i)^{-c_i/2}}{\prod_{j=1}^J {\rm det}\, (2\pi \Gamma_j)^{-d_j/2} }
\mathfrak{FR}^{(\G)}_{\mathfrak{A},\mathfrak{B}} = {\rm FR}^{(\G)}(\mathbf{c},\mathbf{d},\mathbf{L},\boldsymbol{\Sigma},\boldsymbol{\Gamma},\mathcal{Q},\A,\B) = \FR^{(\G)}_{\A,\B},
$$
and hence $\FR^{(e)}_{\A,\B} \leq \FR^{(\G)}_{\A,\B}$. Since the converse inequality is trivial, the proof is complete. 
\end{proof}
 
 \subsection{Step \ref{it:step2}}
 
\begin{proof}[Proof of Step \ref{it:step2}]
	In this step, we weaken the evenness assumption and confirm that $\FR^{(o)}_{\A,\B} =\FR^{(\G)}_{\A,\B}$ in the case of \eqref{e:SettingStep1}.  
	However, we have already seen in Proposition \ref{p:reduce-to-even} that ${\rm FR}^{(o)}_{\A,\B} = {\rm FR}^{(e)}_{\A,\B}$ assuming \eqref{e:SettingStep1}. 
	Since we proved ${\rm FR}^{(e)}_{\A,\B}  = {\rm FR}^{(\mathcal{G})}_{\A,\B} $ in Step \ref{it:step1}, the proof is complete.   
\end{proof}

\subsection{Step \ref{it:step3}}

\begin{proof}[Proof of Step \ref{it:step3}]
	We now weaken the assumption that $\Gamma_j \in \Sym_{> 0}(m_j)$  in \eqref{e:SettingStep1} to the assumption that $\Gamma_j \in \Sym_{\geq 0}(m_j)$. Our goal is to prove that $\FR_{\A,\B}^{(o)} = \FR_{\A,\B}^{(\G)}$ under these assumptions. 
	
	Let $f_i \in \F^{(o)}_{A_i^\flat, A_i^\sharp}(\gamma_{\Sigma_i})$ and $h_j \in \F_{B_j^\flat, B_j^\sharp}(\gamma_{\Gamma_j})$ satisfy (\ref{eq:step1-inq}).  	Given $\varepsilon>0$, set
	$$ 
	\Gamma_{j,\varepsilon}:= \Gamma_j + \varepsilon \Id_{m_j}  \in \Sym_{>0}(m_j). 
	$$
	Since $h_j$ is log-concave, we have by Lemma \ref{lem:log-concave} that $h_j(y_j) \le e^{\langle a_j, y_j\rangle + b_j}$ for some $a_j \in \R^{m_j}$ and $b_j\ge 0$. Therefore $h_j \in L^1(\gamma_{\Gamma_{j,\varepsilon}})$ and so $h_j \in \F_{B_j^\flat, B_j^\sharp}(\gamma_{\Gamma_{j,\varepsilon}})$. 
	 	 								Denote $$\FR^{(*)}_{\A,\B}(\varepsilon):= \FR^{(*)}_{\mathbf{A},\mathbf{B}}(\mathbf{c},\mathbf{d},\mathbf{L},\boldsymbol{\Sigma},\boldsymbol{\Gamma}_{\varepsilon}, \mathcal{Q}).$$
	 Since $\Gamma_{j,\varepsilon}$ are nondegenerate, we have by Step \ref{it:step2} that $\FR^{(o)}_{\A,\B}(\eps) = \FR^{(\G)}_{\A,\B}(\eps)$, and therefore
	$$
	\prod_{i=1}^I \big( \int_{\R^{n_i}} f_i\, d\gamma_{\Sigma_i} \big)^{c_i}
	\le 
	\FR_{\A,\B}^{(\mathcal{G})}(\varepsilon) \prod_{j=1}^J \big( \int_{\R^{m_j}} h_j\, d\gamma_{\Gamma_{j,\varepsilon}}\big)^{d_j}.
	$$
	By taking the limit as $\varepsilon\to 0$, we formally obtain the desired conclusion that $\FR^{(o)}_{\A,\B} \leq \FR^{(\G)}_{\A,\B}$. 
		For a rigorous justification, we will verify that for some subsequence $\varepsilon_k \rightarrow 0$,
	\begin{align}
	& \lim_{\varepsilon_k \to0} 	\FR_{\A,\B}^{(\G)}(\varepsilon_k)= \FR_{\A,\B}^{(\G)}, \label{e:Justification1} \\
	& \lim_{\varepsilon \to 0} \int_{\R^{m_j}}
 h_j\, d\gamma_{\Gamma_{j,\varepsilon}}= \int_{\R^{m_j}} h_j\, d\gamma_{\Gamma_j}. \label{e:Justification2}
 \end{align}
 
We commence with \eqref{e:Justification2}. 
Note that 
$$
\int_{\R^{m_j}} h_j\, d\gamma_{\Gamma_{j,\varepsilon}}
= 
\int_{\R^{m_j}} h_j(\Gamma_{j,\varepsilon}^{1/2}y_j)\, d\gamma(y_j),\quad 
\lim_{\varepsilon\to0} h_j(\Gamma_{j,\varepsilon}^{1/2}y_j)= h_j(\Gamma_{j}^{1/2}y_j). 
$$
Hence, it suffices to find a dominating function. Recalling that $h_j(y_j) \le e^{\langle a_j, y_j\rangle + b_j}$, we use
\begin{align*}
h_j(\Gamma_{j,\varepsilon}^{1/2}y_j)
&\le 
e^{|\Gamma_{j,\varepsilon}^{1/2}a_j||y_j| + b_j} 
\le 
e^{|\Gamma_{j,1}^{1/2}a_j||y_j| + b_j} \in L^1(d\gamma).
\end{align*}
Therefore, Lebesgue's dominated convergence theorem justifies \eqref{e:Justification2}. 

We next confirm \eqref{e:Justification1}. 
Recall that 
\begin{align}
\nonumber \FR^{(\G)}_{\A,\B}(\varepsilon)
&= 
\sup \frac{ \prod_{j=1}^J \det \, (\Id_{m_j} + \Gamma_{j,\varepsilon} B_j)^{d_j/2} }{ \prod_{i=1}^I {\det\, ( \Id_{n_i} +  \Sigma_i A_i)^{c_i/2} }}\\
\label{eq:supremum-attained} &=
\sup \frac{ \prod_{j=1}^J \det \, ( \Id_{m_j} + \Gamma_{j} B_j + \varepsilon B_j)^{d_j/2} }{ \prod_{i=1}^I {\det\, ( \Id_{n_i} +  \Sigma_i A_i)^{c_i/2} }},
\end{align}
where the supremum is taken over all $A_i,B_j$ satisfying
\begin{equation}\label{e:Assump13Dec}
A_i^\flat \le A_i \le A_i^\sharp~,~B_j^\flat \le B_j \le B_j^\sharp ~,~ 
{\rm diag}\, (c_1A_1,\ldots, c_I A_I) \ge \mathcal{Q} + \sum_{j=1}^J d_j L_j^* B_j L_j. 
\end{equation}
In particular, it is clear from (\ref{eq:supremum-attained}) that $\FR^{(\G)}_{\A,\B}\le \FR^{(\G)}_{\A,\B} (\varepsilon) $ for all $\varepsilon>0$. To show the reverse inequality in some subsequential limit  $\varepsilon_k\to0$, recall that $A_i^\sharp, B_j^\sharp <\infty$. Hence, by compactness, we may find extremizers $A_{i,\varepsilon}$ and $B_{j,\varepsilon}$ for which the supremum in (\ref{eq:supremum-attained}) is attained.  
Since $A_i^\flat \le A_{i,\varepsilon}\le A_i^\sharp$ and $B_j^\flat \le B_{j,\varepsilon}\le B_j^\sharp$ 
 for all $\varepsilon>0$, by passing to a subsequence if necessary a finite number of times, 
 \[
\exists A_{i,0}:= \lim_{k\rightarrow \infty} A_{i,\varepsilon_k} ~,~ \exists B_{j,0} = \lim_{k \rightarrow \infty} B_{j,\varepsilon_k} ~,~ \lim_{k \rightarrow \infty} \varepsilon_k B_{j,\varepsilon_k} = 0 ,
 \]
for all $i,j$. Since $A_{i,\varepsilon}, B_{j,\varepsilon}$ satisfy \eqref{e:Assump13Dec} for all $\varepsilon>0$, so do $A_{i,0}, B_{j,0}$. 
Thus, 
\begin{align*}
	\lim_{k\rightarrow \infty}\FR^{(\G)}_{\A,\B}(\varepsilon_k)
	&= \lim_{k \rightarrow \infty} 
	\frac{ \prod_{j=1}^J {\rm det}\, ( {\rm Id}_{m_j} + \Gamma_{j} B_{j,\varepsilon_k} + \varepsilon_k B_{j,\varepsilon_k})^{d_j/2} }{ \prod_{i=1}^I {\rm det}\, ( {\rm Id}_{n_i} +  \Sigma_i A_{i,\varepsilon_k})^{c_i/2} }\\
	&= 
	\frac{ \prod_{j=1}^J {\rm det}\, ( {\rm Id}_{m_j} + \Gamma_{j} B_{j,0})^{d_j/2} }{ \prod_{i=1}^I {\rm det}\, ( {\rm Id}_{n_i} +  \Sigma_i A_{i,0})^{c_i/2} }
	\le \FR^{(\G)}_{\A,\B}. 
\end{align*}
This concludes the proof. 
\end{proof}

\subsection{Step \ref{it:step4}}

\begin{proof}[Proof of Step \ref{it:step4}]
	In this final step, we allow using  $A_i^\sharp,B_j^\sharp = \infty$, establishing that for any Brascamp--Lieb datum $(\mathbf{c},\mathbf{d},\mathbf{L},\boldsymbol{\Sigma},\boldsymbol{\Gamma},\mathcal{Q},\A,\B)$ with $\Sigma_i \in \Sym_{>0}(n_i)$, $\Gamma_j \in \Sym_{\geq 0}(m_j)$ and $A_i^\sharp,B_j^\sharp = \infty$ (no restriction on $A_i^\flat, B_j^\flat \geq 0$), we have $\FR^{(o)}_{\A,\B} = \FR^{(\G)}_{\A,\B}$. In particular, $\FR^{(o)}_{LC} = \FR^{(\G)}_{LC}$.

	Let $f_i \in \F^{(o)}_{\A,\B}(\gamma_{\Sigma_i})$ and $h_j \in \F_{\A,\B}(\gamma_{\Gamma_j})$ satisfy (\ref{eq:step1-inq}). 
	The idea is to regularize $f_i$ by applying the Ornstein-Uhlenbeck flow $P_t^{\Sigma_i}$ with invariant measure $\gamma_{\Sigma_i}$ for time $t \geq 0$, defined for any $\varphi_i \in L^2(\gamma_{\Sigma_i})$ (by density, see e.g. \cite{BGL-Book}) via
	\[
	\partial_t P_t^{\Sigma_i} \varphi_i = \mathcal{L}^{\Sigma_i} P_t^{\Sigma_i} \varphi_i ~,~ \mathcal{L}^{\Sigma_i} \psi_i := \Delta_{\R^{n_i}} \psi_i - \langle \Sigma_i^{-1}x_i, \nabla_{\R^{n_i}} \psi_i \rangle, \quad P_0^{\Sigma_i} = \Id . 
	\]
					We denote the corresponding integral kernel (with respect to Lebesgue measure) by $p_t^{\Sigma_i}(x_i,\hat{x}_i)$, so that 
	\begin{equation}\label{e:OUHeat}
	P^{\Sigma_i}_t \varphi_i (x_i) = \int_{\R^{n_i}} \varphi_i(\hat{x}_i) p_t^{\Sigma_i} (x_i,\hat{x}_i)\, d\hat{x}_i.
	\end{equation}
	Explicitly, it is given by the Mehler formula:
	\begin{equation}\label{e:OUKernel}
	p_t^{\Sigma_i}(x_i,\hat{x}_i)
	= 
	\gamma_{\Sigma_{i,t}}(\hat{x}_i  -  e^{-t \Sigma_i^{-1}}x_i ),\quad \Sigma_{i,t}:= \Sigma_i( {\rm Id}_{n_i} - e^{- 2t \Sigma_i^{-1}} ).
	\quad 	\end{equation}
			By construction, the semi-group $P_t^{\Sigma_i}$ is self-adjoint on $L^2(\gamma_{\Sigma_i})$:
	\begin{equation}\label{e:IntParts}
	\int_{\R^{n_i}} \varphi_i (P_t^{\Sigma_i} \psi_i) \, d\gamma_{\Sigma_i}
	= 
	\int_{\R^{n_i}} (P_t^{\Sigma_i}\varphi_i) \psi_i\, d\gamma_{\Sigma_i} \;\;\; \forall \varphi_i, \psi_i \in L^2(\gamma_{\Sigma_i}) . 
	\end{equation}
	Denoting
	\begin{align*}
	&\Sigma:= {\rm diag}\, (\Sigma_1,\ldots, \Sigma_I) \in \Sym_{>0}(E) \;\;,\;\;
	p^{\Sigma}_t(x,\hat{x}):= \prod_{i=1}^I p_t^{\Sigma_i}(x_i,\hat{x}_i), 
			\end{align*}
	we have
	\begin{equation}\label{e:KernelBigspace}
	p^\Sigma_t(x,\hat{x}) = \gamma_{\Sigma_t}(\hat{x} - e^{-t\Sigma^{-1}} x ),\; {\rm where}\; \Sigma_t:= {\rm diag}\, (\Sigma_{1,t},\ldots, \Sigma_{I,t}) 
	= 
	\Sigma ( {\rm Id}_E - e^{-2t \Sigma^{-1}} ) . 
	\end{equation}

	Let $c_{\min}:= \min(c_1,\ldots,c_I)$, and define for $t \geq 0$
	\begin{equation} \label{eq:u-and-v}
		u_{i,t} := P_t^{\Sigma_i} f_i ~,~ v_{j,t}(y_j)
		:= 
		\bigg( \mathfrak{c}_t^{\frac{1}{J}}
		\int_{\R^{m_j}}h_j (\hat{y}_j)^{ \frac{d_j(J+1)}{c_{\min}} } \gamma_{ L_j \Sigma_t L_j^* }(y_j - \hat{y}_j )\, d\hat{y}_j
		\bigg)^{ \frac{c_{\min}}{d_j(J+1)}}
	\end{equation}
	for some constant $\mathfrak{c}_t > 0$ to be determined below. 
	We now claim that for all $t > 0$ small enough:
	\begin{enumerate}
		\item \label{it:step4-1} $u_{i,t} \in \mathcal{F}^{(o)}_{A^\flat_{i,t},A^\sharp_{i,t}}(\gamma_{\Sigma_i})$ for some $A^*_{i,t} \in \Sym_{\geq 0}(n_i)$, $0 \leq A^\flat_{i,t} < A^\sharp_{i,t} < \infty$, such that $\lim_{t \to 0} A^\flat_{i,t} = A^\flat_i$, and $\int u_{i,t} d\gamma_{\Sigma_i} = \int f_i d\gamma_{\Sigma_i}$. 
		\item \label{it:step4-2}
		$v_{j,t}\in \mathcal{F}_{B^\flat_{j,t}, B^\sharp_{j,t}}(\gamma_{\Gamma_j})$ for some $B^*_{j,t} \in \Sym_{\geq 0}(m_j)$, $0 \leq B^\flat_{j,t} < B^\sharp_{j,t} < \infty$, such that $\lim_{t \to 0} B^\flat_{j,t} = B^\flat_j$. 
		\item \label{it:step4-3} It holds that
		\begin{equation}\label{e:PWAssump15Dec}
			\prod_{i=1}^I u_{i,t}(x_i)^{c_i}\le e^{-\frac12 \langle\Q_t x,  x\rangle} \prod_{j=1}^J v_{j,t}(L_{j,t}x)^{d_j} \;\;\; \forall x \in E ,
		\end{equation}
		for some  $\Q_t \in \Sym(E)$ such that $\lim_{t\to 0} \Q_t = \Q$ and with  
		\begin{equation}\label{e:DefL(j,t)}
		L_{j,t}:= L_j e^{-t\Sigma^{-1}}:E \to \R^{m_j} ,
		\end{equation} 
		which are obviously linear surjective maps satisfying $\lim_{t\to0} L_{j,t} = L_j$. 		Moreover, $\lim_{t \to 0} \mathfrak{c}_t = 1$. 
	\end{enumerate}
	
	To see claim (\ref{it:step4-1}), note that (coordinate-wise) $x_i \in L^2(\gamma_{\Sigma_i})$ and $P_t^{\Sigma_i} x_i = e^{-t \Sigma^{-1}} x_i$  since $\mathcal{L}^{\Sigma_i} x_i = - \Sigma_i^{-1} x_i$, and so it follows from \eqref{e:IntParts} that the $\gamma_{\Sigma_i}$-barycenter of $u_{i,t}$ remains at the origin for all $t \geq 0$:
	\[
	 \int_{\R^{n_i}} x_i u_{i,t} d\gamma_{\Sigma_i} = \int_{\R^{n_i}} x_i P_t^{\Sigma_i} f_i d\gamma_{\Sigma_i} = \int_{\R^{n_i}} P_t^{\Sigma_i}(x_i) f_i d\gamma_{\Sigma_i} = e^{-t \Sigma^{-1}} \int x_i f_i d\gamma_{\Sigma_i} = 0 .
	\]
	Similarly,
	\[
	\int_{\R^{n_i}} u_{i,t} d\gamma_{\Sigma_i} = \int_{\R^{n_i}} P_t f_i d\gamma_{\Sigma_i} = \int_{\R^{n_i}} f_i d\gamma_{\Sigma_i}  . 
	\]
	In addition, since $u_{i,t} = (e^{t \Sigma^{-1}})_{\#} (f_i \ast \gamma_{\Sigma_{i,t}})$ by (\ref{e:OUHeat}) and (\ref{e:OUKernel}), it follows by Lemma \ref{lem:silly} that $u_{i,t} \in \F^{(o)}_{A_{i,t}^\flat,A_{i,t}^\sharp}$, where 
	\[
	A_{i,t}^{*} = e^{-t \Sigma_i^{-1}} ((A_i^*)^{-1} + \Sigma_{i,t})^{-1} e^{-t \Sigma_i^{-1}}  ~,~ * \in \{\flat,\sharp \} ,
	\]
	and, recall, we interpret $((A_i^*)^{-1} + \Sigma_{i,t})^{-1}$ as $A_i^* - A_i^* (A_i^* + \Sigma_{i,t}^{-1})^{-1} A_i^* \geq 0$ if $A_i^*$ is degenerate (which is well-defined since $\Sigma_{i,t} > 0$). In particular, $A_{i,t}^\flat = 0$ whenever $A_{i}^\flat = 0$. Clearly $A^\flat_{i,t} \leq A^\sharp_{i,t} < \infty$ and $\lim_{t \to 0} A^\flat_{i,t} = A^\flat_i$ since $\Sigma_{i,t} \rightarrow 0$ as $t \rightarrow 0$. This establishes claim (\ref{it:step4-1}). 
	
	Similarly, Lemma \ref{lem:silly} implies that $v_{j,t} \in \F_{B_{j,t}^\flat,B_{j,t}^\sharp}$, where
	\[
	B_{j,t}^* = \brac{(B_j^*)^{-1} + \frac{d_j (J+1)}{c_{\min}} L_j \Sigma_t L_j^*}^{-1} ~,~ * \in \{\flat,\sharp \} .
	\]
	Since $L_j$ are surjective, $L_j \Sigma_t L_j^* > 0$ is non-degenerate, and so the same comments as above apply, 
	establishing claim (\ref{it:step4-2}).

	To verify claim (\ref{it:step4-3}),  we recall \eqref{e:OUHeat} and apply Jensen's inequality for each of the probability measures $p_t^{\Sigma_i}(x_i,\hat{x}_i)d\hat{x}_i$ (with $x_i$ fixed) as follows:
	\begin{align*}
	\prod_{i=1}^I u_{i,t}(x_i)^{c_i}
	&= 
	\brac{\prod_{i=1}^I u_{i,t}(x_i)^{\frac{c_i}{c_{\min}}}}^{c_{min}} \\
	&= 
	\bigg(
	\prod_{i=1}^I 
	\big( 
	\int_{\R^{n_i}} f_i(\hat{x}_i) p_t^{\Sigma_i}(x_i, \hat{x}_i)\, d\hat{x}_i 
	\big)^{\frac{c_i}{c_{\min}} }
	\bigg)^{c_{\min}} \\
	&\le 
	\bigg(	\prod_{i=1}^I 
	\int_{\R^{n_i}} f_i(\hat{x}_i)^{ \frac{c_i}{c_{\min}} } p_t^{\Sigma_i}(x_i, \hat{x}_i)\, d\hat{x}_i 
	\bigg)^{c_{\min}} \\
	& = \bigg( \int_{E} 
	\big(\prod_{i=1}^I 
	 f_i(\hat{x}_i)^{ c_i}\big)^{\frac{1}{c_{\min}} } p_t^{\Sigma}(x, \hat{x})\, d\hat{x}
	\bigg)^{c_{\min}}. 
	\end{align*}
	We then make use of the pointwise assumption (\ref{eq:step1-inq}) to estimate this quantity by 
	\begin{align*}
	&\le 
	\bigg( \int_{E} 
	 e^{-\frac12 \langle \frac{1}{c_{\min}} \Q \hat{x}, \hat{x}\rangle } \prod_{j=1}^J h_j(L_j\hat{x})^{\frac{d_j}{c_{\min}}}  p_t^{\Sigma}(x, \hat{x})\, d\hat{x}
	\bigg)^{c_{\min}}\\
	&\le 
	\bigg( \int_{E} 
	 e^{-\frac12 \langle \frac{J+1}{c_{\min}} \Q \hat{x}, \hat{x}\rangle } p_t^{\Sigma}(x, \hat{x})\, d\hat{x}
	 \bigg)^{\frac{c_{\min}}{J+1}}
	 \prod_{j=1}^J 
	 \bigg( \int_E h_j(L_j\hat{x})^{\frac{d_j(J+1)}{c_{\min}}}  p_t^{\Sigma}(x, \hat{x})\, d\hat{x}
	\bigg)^{\frac{c_{\min}}{J+1}}, 
	\end{align*}	
	where we used H\"{o}lder's inequality in the last step. 
	By choosing $t > 0$ small enough (so that $\frac{J+1}{c_{\min}} \Q + \Sigma_t^{-1} > 0$),  the first integral is finite for all $x$ and moreover
	$$
	\bigg( \int_{E} 
	 e^{-\frac12 \langle \frac{J+1}{c_{\min}} \Q \hat{x}, \hat{x}\rangle } p_t^{\Sigma}(x, \hat{x})\, d\hat{x}
	 \bigg)^{\frac{c_{\min}}{J+1}}
	 = 
	 \mathfrak{c}_t^{\frac{c_{\min}}{J+1}} e^{-\frac12 \Q_t \langle x, x\rangle} 
	$$
	 for some $\mathfrak{c}_t > 0$ and $\mathcal{Q}_t \in {\rm Sym}(E)$ such that $\lim_{t\to0} \mathfrak{c}_t = 1$ and $\lim_{t\to0} \mathcal{Q}_t = \mathcal{Q}$ (as follows by a direction Gaussian computation). 	 	 	 On the other hand, recalling the expression \eqref{e:KernelBigspace},  
	  \begin{align*}
	  &\int_E h_j(L_j\hat{x})^{\frac{d_j(J+1)}{c_{\min}}}  p_t^{\Sigma}(x, \hat{x})\, d\hat{x} \\
	  &= 
	  \int_E h_j(L_j\hat{x})^{\frac{d_j(J+1)}{c_{\min}}}  \gamma_{\Sigma_t} ( e^{-t\Sigma^{-1}}x - \hat{x} ) \, d\hat{x}\\
	  &= 
	  \int_{\R^{m_j}} h_j(\hat{y}_j)^{\frac{d_j(J+1)}{c_{\min}}} (L_j)_\# \big[ \gamma_{\Sigma_t} ( e^{-t\Sigma^{-1}}x- \hat{x}) d\hat{x} \big](d\hat{y}_j) \\
	  & =
	   \int_{\R^{m_j}} h_j(\hat{y}_j)^{\frac{d_j(J+1)}{c_{\min}}} \gamma_{L_j \Sigma_t L_j^*}( L_j e^{-t\Sigma^{-1}}x - \hat{y}_j) d\hat{y}_j . 	   
	  \end{align*}
	  Combining everything and recalling the definitions of $v_{j,t}$ in (\ref{eq:u-and-v}) and $L_{j,t}$ in (\ref{e:DefL(j,t)}), this confirms (\ref{e:PWAssump15Dec}) and hence claim (\ref{it:step4-3}).

	Having confirmed claims (\ref{it:step4-1})--(\ref{it:step4-3}) for small enough $t>0$,  we may apply the consequence of Step \ref{it:step3} with the regularized Brascamp--Lieb datum  
	$$
	\mathbf{c}, \; \mathbf{d},\; \mathbf{L}_t = (L_{j,t})_{j=1,\ldots, J},\; \mathcal{Q}_t,\; \boldsymbol{\Sigma},\; \boldsymbol{\Gamma}, \; 0 \leq A_{i,t}^\flat < A_{i,t}^\sharp <\infty ,\; 0 \leq B_{j,t}^\flat < B_{j,t}^\sharp < \infty,  
	$$
	to conclude that 
	\begin{equation}\label{e:ApplyStep3}
	\prod_{i=1}^I \big( \int_{\R^{n_i}} u_{i,t}\, d\gamma_{\Sigma_i} \big)^{c_i} 
	\le 
	{\rm FR}^{(\mathcal{G})}(t) 
	\prod_{j=1}^J \big( \int_{\R^{m_j}} v_{j,t} \, d\gamma_{\Gamma_j} \big)^{d_j}. 
	\end{equation}
	Here, 
	\begin{equation}\label{e:FR(t)}
	{\rm FR}^{(\mathcal{G})}(t) 
	:= 
	\sup \frac{\prod_{j=1}^J {\rm det}\, ( {\rm Id}_{m_j} + \Gamma^{1/2}_j B_{j,t} \Gamma^{1/2}_j  )^{d_j/2}}{ \prod_{i=1}^I {\rm det}\, ({\rm Id}_{n_i} + \Sigma^{1/2}_i A_{i,t} \Sigma^{1/2}_i )^{c_i/2} }, 
	\end{equation}
	where the supremum is taken over all 
	\begin{equation}\label{e:AcceptedInputs}
	\begin{split}
	& A^{\flat}_{i,t} \le A_{i,t} \le A^{\sharp}_{i,t},\; B^{\flat}_{j,t} \leq B_{j,t} \le B^{\sharp}_{j,t} \; : \\
	& {\rm diag}\, (c_1A_{1,t},\ldots, c_I A_{I,t}) \ge \mathcal{Q}_t + \sum_{j=1}^J d_j L_{j,t}^* B_{j,t} L_{j,t} .
	\end{split}
	\end{equation}
	Recall that $\int_{\R^{n_i}} u_{i,t}\, d\gamma_{\Sigma_i} = \int_{\R^{n_i}} f_i\, d\gamma_{\Sigma_i}$. Hence, we will conclude the proof of $\FR^{(o)}_{\A,\B} = \FR^{(\G)}_{\A,\B}$ by showing 
	\begin{align}
	 \limsup_{t\to0} {\rm FR}^{(\mathcal{G})}(t) & \le {\rm FR}^{(\mathcal{G})}_{\A,\B},\label{e:LimitGaussconst} \\
	\lim_{t\to 0} \int_{\R^{m_j}} v_{j,t} \, d\gamma_{\Gamma_j}
	& = 
	\int_{\R^{m_j}} h_j\, d\gamma_{\Gamma_j}. \label{e:LimitMass-v(j,t)}
	\end{align}

	To show \eqref{e:LimitMass-v(j,t)}, recall the definition (\ref{eq:u-and-v}) of $v_{j,t}$:
	\[
	v_{j,t}(y_j) = \bigg( \mathfrak{c}_t^{\frac{1}{J}}
		\int_{\R^{m_j}}h_j (\hat{y}_j)^{ \frac{d_j(J+1)}{c_{\min}} } \gamma_{ L_j \Sigma_t L_j^* }(y_j - \hat{y}_j )\, d\hat{y}_j
		\bigg)^{ \frac{c_{\min}}{d_j(J+1)}} .
	\]
	Since $\lim_{t \to 0} \mathfrak{c}_t = 1$, and since $h_j^{ \frac{d_j(J+1)}{c_{\min}} }$ is log-concave and $\gamma_{L_j \Sigma_t L_j^*}$ is a Gaussian approximation of identity, 	it follows by Lemma \ref{lem:approximation-of-identity} that $\lim_{t \to 0} v_{j,t}(y_j) = h_j(y_j)$ for almost-every $y_j \in \R^{m_j}$, and that moreover,
	\[
	v_{j,t}(y_j) \leq C e^{\scalar{a,y_j}} \;\;\; \forall y_j \in \R^{m_j} ,
	\]
	for some $C > 0$, $a \in \R^{m_j}$ and all small-enough $t > 0$. Note that the right-hand-side is in $L^1(\gamma_{\Gamma_j})$, whether $\Gamma_j \geq 0$ is degenerate or not. Therefore (\ref{e:LimitMass-v(j,t)}) follows by Lebesgue's dominated convergent theorem.

	To show \eqref{e:LimitGaussconst}, take arbitrary $A_{i,t},B_{j,t}$ satisfying \eqref{e:AcceptedInputs}. Since $\lim_{t \to 0} B_{j,t}^\flat = B_j^\flat$, there exist $B^t_j \in \Sym_{\geq 0}(m_j)$, $B_j^\flat \leq B^t_j < \infty$, so that $\snorm{B_j^t - B_{j,t}}_{op} \leq R_t$ for some $R_t \rightarrow 0$ as $t \rightarrow 0$ for all $j$. Similarly, since $\lim_{t \to 0} e^{t \Sigma_i^{-1}} A_{i,t}^\flat e^{t \Sigma_i^{-1}} = A_i^\flat$,  there exist $A^t_i \in \Sym_{\geq 0}(n_i)$, $A_i^\flat \leq A^t_i < \infty$, so that $\snorm{A_i^t - e^{t \Sigma_i^{-1}} A_{i,t} e^{t \Sigma_i^{-1}} }_{op} \leq R_t$ for all $i$ (we may assume a common $R_t$ for all $i$, $j$).
	
	Conjugating \eqref{e:AcceptedInputs} with $e^{t \Sigma^{-1}}$, recalling (\ref{e:DefL(j,t)}), and using $B_{j,t} \geq B_j^t - R_t \Id_{m_j}$, we obtain
	\[
	e^{t \Sigma^{-1}} \Q_t e^{t \Sigma^{-1}}  - \sum_{j=1}^J d_j L_j^* L_j R_t +  \sum_{j=1}^J d_j L_{j}^* B^t_j L_{j} \leq  e^{t\Sigma^{-1}}{\rm diag}\, (c_1A_{1,t},\ldots, c_I A_{I,t}) e^{t\Sigma^{-1}} .
	\]
	Denoting $\lambda_t = \norm{\mathcal{Q} - e^{t\Sigma^{-1}}\mathcal{Q}_t e^{t\Sigma^{-1}} + \sum_{j=1}^J d_j L_j^* L_j R_t}_{op}$, we have $\lim_{t \to 0} \lambda_t = 0$ and
	\[
	\Q + \sum_{j=1}^J d_j L_{j}^* B^t_j L_{j} \leq \lambda_t \Id_E + e^{t\Sigma^{-1}}{\rm diag}\, (c_1A_{1,t},\ldots, c_I A_{I,t}) e^{t\Sigma^{-1}} .
	\]
	Recalling that $e^{t \Sigma_i^{-1}} A_{i,t} e^{t \Sigma_i^{-1}} \leq A_i^t + R_t \Id_{n_i}$, if we define
\[
\bar{A}_i^t  :=  A_i^t + \brac{\frac{\lambda_t}{c_i} + R_t}{\rm Id}_{n_i} , 
\]
we see that $A_i^\flat \leq \bar{A}_i^t < A_i^\sharp = \infty$ and $B_j^\flat \leq B^t_j < B_j^\sharp = \infty$ satisfy
\[
\Q + \sum_{j=1}^J d_j L_{j}^* B_j^t L_{j} \le  {\rm diag}\, (c_1\bar A_1^t,\ldots, c_I \bar A_I^t) .
\]
Consequently, by definition of $\FR^{(\G)}_{\A,\B} = \FR^{(\G)}(\mathbf{c},\mathbf{d},\mathbf{L},\boldsymbol{\Sigma},\boldsymbol{\Gamma},\mathcal{Q},\A,\B)$, we have
\[
\prod_{j=1}^J {\rm det}\, ( {\rm Id}_{m_j} + \Gamma_j^{1/2} B_j^t \Gamma_j^{1/2} )^{d_j/2}  \leq \FR^{(\G)}_{\A,\B} \prod_{i=1}^I {\rm det}\, ({\rm Id}_{n_i} + \Sigma_i^{1/2} \bar A_i^t \Sigma_i^{1/2})^{c_i/2} .
\]
On the other hand, applying Lemma \ref{lem:det-stability},
\begin{align*}
\det (\Id_{m_j} + \Gamma_j^{1/2} B_{j,t} \Gamma_j^{1/2}) & \leq \det(\Id_{m_j} + \Gamma_j^{1/2} B_j^t \Gamma_j^{1/2}) (1 + \snorm{\Gamma_j^{1/2} (B_{j,t} - B_j^t) \Gamma_j^{1/2}}_{op})^{m_j} \\
& \leq \det(\Id_{m_j} + \Gamma_j^{1/2} B_j^t \Gamma_j^{1/2}) (1 + \snorm{\Gamma_j}_{op} R_t)^{m_j} ,
\end{align*}
and similarly, 
\begin{align*}
\det (\Id_{n_i} + \Sigma_i^{1/2} A_{i,t} \Sigma_i^{1/2}) & = \det(\Id_{n_i} + \Sigma_i^{1/2} e^{-t \Sigma_i^{-1}} e^{t \Sigma_i^{-1}} A_{i,t} e^{t \Sigma_i^{-1}} e^{-t \Sigma_i^{-1}} \Sigma_i^{1/2}) \\
& \geq  \frac{\det(\Id_{n_i} + \Sigma_i^{1/2} e^{-t \Sigma_i^{-1}} \bar A_i^t e^{-t \Sigma_i^{-1}} \Sigma_i^{1/2}) }{( 1+ \snorm{\Sigma_i e^{-2t \Sigma_i^{-1}}}_{op} \snorm{e^{t \Sigma_i^{-1}} A_{i,t} e^{t \Sigma_i^{-1}} - \bar A_i^t}_{op})^{n_i}} \\
& \geq \frac{\det(e^{-2t \Sigma_i^{-1}} + e^{-t \Sigma_i^{-1}} \Sigma_i^{1/2} \bar A_i^t  \Sigma_i^{1/2} e^{-t \Sigma_i^{-1}}) }{( 1+ \snorm{\Sigma_i e^{-2t \Sigma_i^{-1}}}_{op} (2 R_t + \frac{\lambda_t}{c_i}))^{n_i}} , \\
& = \frac{\det(\Id_{n_i} + \Sigma_i^{1/2} \bar A_i^t  \Sigma_i^{1/2})}{\det(e^{2 t \Sigma_i^{-1}}) ( 1+ \snorm{\Sigma_i e^{-2t \Sigma_i^{-1}}}_{op} (2 R_t + \frac{\lambda_t}{c_i}))^{n_i}} .
\end{align*}
Combining everything, we deduce
\begin{align*}
& \frac{\prod_{j=1}^J {\rm det}\, ( {\rm Id}_{m_j} + \Gamma_j^{1/2} B_{j,t}\Gamma_j^{1/2} )^{d_j/2}}{\prod_{i=1}^I {\rm det}\, ({\rm Id}_{n_i} + \Sigma_i^{1/2}  A_{i,t} \Sigma_i^{1/2})^{c_i/2}} \leq \FR^{(\G)}_{\A,\B} \cdot M_t , \\
& M_t := \Pi_{j=1}^J (1 + \snorm{\Gamma_j}_{op} R_t)^{m_j} \Pi_{i = 1}^I \det(e^{2 t \Sigma_i^{-1}})  ( 1+ \snorm{\Sigma_i e^{-2t \Sigma_i^{-1}}}_{op} (2 R_t + \frac{\lambda_t}{c_i}))^{n_i} .
\end{align*}
Hence ${\rm FR}^{(\mathcal{G})}(t) \le {\rm FR}^{(\mathcal{G})}_{\A,\B} \cdot M_t$, and 
since $\lim_{t\to0} \lambda_t = \lim_{t\to0} R_t = 0$, $\lim_{t \to 0} \det(e^{2 t \Sigma_i^{-1}}) = 1$ and $\snorm{\Sigma_i e^{-2t \Sigma_i^{-1}}}_{op} \leq \snorm{\Sigma_i}_{op}$ remains bounded, we see that $\lim_{t \to 0} M_t =1$ and so \eqref{e:LimitGaussconst} is established. This concludes the proof.
\end{proof}

\section{Gaussian saturation in the GCRSI} \label{sec:GCRSI}

We now return to the main application of our Gaussian Forward-Reverse Brascamp--Lieb Theorem \ref{thm:intro-FRBL} in this work --- the Gaussian conjugate Rogers--Shephard inequality and its variants. 

Throughout all applications in this work, we will use the Brascamp--Lieb datum given by:
\begin{equation} \label{eq:our-BL-datum}
\begin{cases}
& I = J = 2 ~,~ c_1 = c_2 = d_1 = d_2 = 1 ~,~ n_1=n_2 = m_2 = n, \;\; m_1 = 2n , \\
& L_1(x) = (\alpha x_1, \beta x_2) ~,~ L_2(x) = a x_1 + b x_2 ~,~ x = (x_1,x_2) \in \R^{n} \times \R^n, \\
&  \Q = 0 ~,~ \Sigma_1 = \Sigma_2 = \Gamma_2 = \Id_n \text{ and the degenerate } \Gamma_1 = \begin{pmatrix} \Id_n & \Id_n \\ \Id_n & \Id_n \end{pmatrix} ,
\end{cases}
\end{equation}
for some constants $\alpha,\beta, a,b  \in \R \setminus \{0\}$. In this section we analyze the case when $\alpha = \beta = 1$, namely that $L_1(x) = x$ is the identity map. 

\medskip

Let us start by formulating a generalized version of (the inequality statement in) Theorem \ref{thm:intro-our-FRBL}, which immediately also implies (the inequality statement in) Theorem \ref{thm:intro-ab}. The analysis of equality is deferred to Section \ref{sec:equality}. 

\begin{thm} \label{thm:GCRSI-FRBL}
Let $f_1,f_2 \in \F_{LC}^{(o)}(\gamma^n)$, and let $h_1 : \R^{2n} \rightarrow \R_+$ and $h_2 : \R^{n} \rightarrow \R_+$ denote two log-concave functions. 
Let $a,b \in \R$ with $\abs{a},\abs{b}, \abs{a+b} \geq 1$. If
\begin{equation} \label{eq:GCRSI-FRBL-assumption}
f_1(x_1) f_2(x_2) \leq h_1(x_1,x_2) h_2(a x_1 + b x_2) \;\;\; \forall x_1,x_2 \in \R^n ,
\end{equation}
then
\begin{equation} \label{eq:GCRSI-FRBL-conclusion}
\int_{\R^n} f_1 d\gamma \int_{\R^n} f_2 d\gamma \leq \int_{\R^n} h_1(y,y) d\gamma(y) \int_{\R^n} h_2 d\gamma . 
\end{equation}
\end{thm}

Note that all log-concave functions are automatically $\gamma_{\Sigma}$-integrable for any $\Sigma < \infty$ by Corollary \ref{cor:log-concave-integrable}.
Applying Theorem \ref{thm:intro-FRBL} with the above Brascamp--Lieb datum,
 we see that the optimal constant on the right-hand-side of (\ref{eq:GCRSI-FRBL-conclusion}) is saturated by centered Gaussians satisfying (\ref{eq:GCRSI-FRBL-assumption}). Therefore, to establish (\ref{eq:GCRSI-FRBL-conclusion}), the remaining task is to show that for this datum, $\FR^{(\G)}_{LC} \leq 1$ (and hence, by testing constant functions in (\ref{eq:GCRSI-FRBL-assumption}), in fact $\FR^{(\G)}_{LC} = 1$). 

In other words, for all $A_1,A_2,C \in \Sym_{\geq 0}(n)$ and $B = \begin{pmatrix} B_1 & B_2 \\ B^*_2 & B_4 \end{pmatrix} \in \Sym_{\geq 0}(2n)$ such that
\[
g_{A_1}(x_1) g_{A_2}(x_2) \leq g_{B}(x_1,x_2) g_C(a x_1+b x_2) \;\;\; \forall x_1,x_2 \in \R^n ,
\]
we need to show (assuming $\abs{a},\abs{b}, \abs{a+b} \geq 1$) that
\[
\int_{\R^n} g_{A_1} d\gamma \int_{\R^n} g_{A_2} d\gamma \leq \int_{\R^n} g_{\bar B} d\gamma \int_{\R^n} g_{C} d\gamma ,
\]
where:
\[
\bar B := B_1+B_2+B_2^*+B_4 . 
\]
Equivalently, whenever
\begin{equation} \label{eq:matrix-assumption}
\begin{pmatrix} A_1 & 0 \\ 0 & A_2 \end{pmatrix} \geq \begin{pmatrix} B_1 & B_2 \\ B^*_2 & B_4 \end{pmatrix} + \begin{pmatrix} a^2 C & ab C \\ ab C & b^2 C \end{pmatrix} ,
\end{equation}
we would like to show that:
\begin{equation} \label{eq:det-conclusion}
\det(\Id_n + A_1) \det(\Id_n + A_2) \geq \det(\Id_n + \bar B) \det(\Id_n + C) .
\end{equation}

\subsection{Some failed attempts}

Before establishing (\ref{eq:det-conclusion}), we first demonstrate that a naive argument is bound to fail, even in the simplest case when $a=b=1$. Indeed, 
adding $\Id_{2n}$ to both sides of (\ref{eq:matrix-assumption}) and taking determinant, (\ref{eq:det-conclusion}) would follow if we could show that:
\[
\det\brac{ \begin{pmatrix} \Id_n & 0 \\ 0 & \Id_n \end{pmatrix} + \begin{pmatrix} B_1 & B_2 \\ B^*_2 & B_4 \end{pmatrix} + \begin{pmatrix} C & C \\ C & C \end{pmatrix}}\geq \det(\Id_n + \bar B) \det(\Id_n + C) . 
\]
Changing coordinates with respect to the orthonormal basis $\{ \frac{1}{\sqrt{2}} (e_i, -e_i)\}_{i=1,\ldots,n} \cup \{\frac{1}{\sqrt{2}} (e_i, e_i)\}_{i=1,\ldots,n}$, our goal becomes showing, if $\begin{pmatrix} D_1 & D_2 \\ D^*_2 & D_4 \end{pmatrix} \in \Sym_{\geq 0}(2n)$, that
\[
\det \brac{ \begin{pmatrix} \Id_n & 0 \\ 0 & \Id_n \end{pmatrix} + \begin{pmatrix} D_1 & D_2 \\ D^*_2 & D_4 \end{pmatrix} + \begin{pmatrix} 0 & 0 \\ 0 & 2C \end{pmatrix}} \geq \det(\Id_n + 2 D_4) \det(\Id_n + C) .
\]
However, even for $n=1$, simple examples (e.g. $D_1=D_2 = C = 0$ and $D_4 = 1$) show that this is in general false. 

As another alternative, which makes no difference for even functions, we might want to try using $a=-b=1$ and assume that
\[
g_{A_1}(x_1) g_{A_2}(x_2) \leq g_{B}(x_1,x_2) g_C(x_1-x_2) \;\;\; \forall x_1,x_2 \in \R^n ,
\]
so
\[
\begin{pmatrix} A_1 & 0 \\ 0 & A_2 \end{pmatrix} \geq \begin{pmatrix} B_1 & B_2 \\ B^*_2 & B_4 \end{pmatrix} + \begin{pmatrix} C & -C \\ -C & C \end{pmatrix} .
\]
To show (\ref{eq:det-conclusion}), the same argument as above reduces the task to showing that
\[
\det \brac{ \begin{pmatrix} \Id_n & 0 \\ 0 & \Id_n \end{pmatrix} + \begin{pmatrix} D_1 & D_2 \\ D^*_2 & D_4 \end{pmatrix} + \begin{pmatrix} 2C & 0 \\ 0 & 0 \end{pmatrix}} \geq \det(\Id_n + 2 D_4) \det(\Id_n + C) .
\]
This is again false in general, but by using Schur's complement, it is not hard to show that the left-hand-side is lower-bounded by
\[
\geq \det(\Id_n + D_4) \det(\Id_n + 2 C) .
\]
Applying Theorem \ref{thm:intro-FRBL}, it is easy to check that this recovers the Milman--Pajor inequality (\ref{eq:intro-MilmanPajor}) with $a = b = \frac{1}{\sqrt{2}}$ (for convex sets $K,L$ having Gaussian barycenter at the origin):
\[
\gamma(K) \gamma(L) \leq \gamma(\sqrt{2} (K\cap L)) \gamma(\frac{1}{\sqrt{2}} (K-L)) ,
\]
but this is not what we were aiming for. 

Moreover, contrary to the case $a=b=1$, it is simply false that (\ref{eq:matrix-assumption}) implies (\ref{eq:det-conclusion}) when $a=-b=1$, even for $n=1$. 
For example, choosing $A_1 = A_2 = 1$ and $B_1 = B_2 = B_2^* = B_4 = C = 1/2$ satisfies (\ref{eq:matrix-assumption}) but not (\ref{eq:det-conclusion}) since $2 \cdot 2 < 3 \cdot \frac{3}{2}$. Further counterexamples in the same spirit are provided in Subsection \ref{subsec:counterexamples}. 

\smallskip

The above failed attempts illustrate that, surprisingly (to us), there is more information in the assumption (\ref{eq:matrix-assumption}) than what have used thus far.

\subsection{GCI for Gaussians}

To establish (\ref{eq:det-conclusion}), we first need the following. 

\begin{lemma}[GCI for Gaussians] \label{lem:GaussianGCI}
For any $A,B \in \Sym_{\geq 0}(n)$, we have:
\[
\det(\Id_n + A +B) \leq \det(\Id_n +A) \det(\Id_n + B). 
\]
\end{lemma}
\begin{rem}
Note that this is actually the functional version of the GCI (\ref{eq:intro-functional-GCI}) applied to Gaussian inputs: $\int g_{A} g_{B} d\gamma \geq \int g_{A} d\gamma \int g_{B} d\gamma$. For completeness, we present a short independent proof. 
\end{rem}
\begin{proof}
Denote $U = \begin{pmatrix} A^{1/2} \\ B^{1/2} \end{pmatrix}$ a $2n \times n$ rectangular matrix. Note that $U^* U = A + B$ , but
\[
U U^* = \begin{pmatrix} A  & A^{1/2} B^{1/2} \\ B^{1/2} A^{1/2} & B \end{pmatrix} . 
\]
Since $\det(\Id_n + C D) = \det(\Id_m + D C)$ for any $n \times m$ matrix $C$ and $m \times n$ matrix $D$, 
it follows that:
\[
\det(\Id_n + A + B) = \det \begin{pmatrix} \Id_n + A  & A^{1/2} B^{1/2} \\ B^{1/2} A^{1/2} & \Id_n + B \end{pmatrix} \leq \det(\Id_n +A) \det(\Id_n + B) ,
\]
where we used Fischer's inequality for positive-definite matrices in the last inequality \cite[Theorem 7.8.5]{HornJohnson-MatrixAnalysis} (an equivalent form of Hadamard's inequality). 
\end{proof}
\begin{corollary}[GCI for Gaussians] \label{cor:GaussianGCI}
For any $\mathcal{A},\mathcal{B},\mathcal{C} \in \Sym_{\geq 0}(n)$ we have:
\[
\det(\mathcal{C}) \det(\mathcal{C} + \mathcal{A}+\mathcal{B}) \leq \det(\mathcal{C} + \mathcal{A}) \det(\mathcal{C}+\mathcal{B}) . 
\]
\end{corollary}
\begin{proof}
There is nothing to prove if $\det(\mathcal{C}) = 0$, so we may assume $\mathcal{C} > 0$. 
Now reduce to the case $\mathcal{C}=\Id_n$ by writing $A = \mathcal{C}^{-1/2} \mathcal{A} \mathcal{C}^{-1/2}$ and $B = \mathcal{C}^{-1/2} \mathcal{B} \mathcal{C}^{-1/2}$, and apply Lemma \ref{lem:GaussianGCI}. 
\end{proof}

\subsection{GCRSI for Gaussians}

\begin{proposition} \label{prop:GaussianGCRSI}
Let $A_1,A_2,C \in \Sym_{\geq 0}(n)$ and $B = \begin{pmatrix} B_1 & B_2 \\ B^*_2 & B_4 \end{pmatrix} \in \Sym_{\geq 0}(2n)$, and assume that
\begin{equation} \label{eq:GaussianGCRSI-assumption}
\begin{pmatrix} A_1 & 0 \\ 0 & A_2 \end{pmatrix} \geq \begin{pmatrix} B_1 & B_2 \\ B_2^* & B_4 \end{pmatrix}  + \begin{pmatrix} a^2 C & ab C \\ ab C & b^2 C \end{pmatrix} ,
\end{equation}
for some $a,b \in \R$ such that $\abs{a},\abs{b} , \abs{a+b}\geq 1$. Then
\[
\det(\Id_n + A_1) \det(\Id_n + A_2) \geq \det(\Id_n + C) \det(\Id_n + \bar B + ((a+b)^2 -1) C) ,
\]
where recall, $\bar B = B_1 + B_2 + B_2^* + B_4$. In particular,
\begin{equation} \label{eq:GaussianGCRSI-conclusion}
\det(\Id_n + A_1) \det(\Id_n + A_2) \geq \det(\Id_n + C) \det(\Id_n + \bar B) ,
\end{equation}
and if $\abs{a+b} > 1$ and equality occurs in (\ref{eq:GaussianGCRSI-conclusion}) then necessarily $C = 0$. 
\end{proposition}

\begin{rem}
We shall see from the proof that the requirement that $B \geq 0$ is not really needed, and that it is enough to assume that $B_1,B_4 \geq 0$ and that $\Id_n + \bar B + ((a+b)^2 -1) C \geq 0$.  
\end{rem}

\begin{proof}
Note that $\mathcal{A} := A_1 - C \geq 0$ and $\mathcal{B} := A_2 - C \geq 0$ by (\ref{eq:GaussianGCRSI-assumption}) since $B_1,B_4 \geq 0$ and $a^2,b^2 \geq 1$. 
Applying Corollary \ref{cor:GaussianGCI} with $\mathcal{C} = \Id_n + C$, we obtain:
\begin{equation} \label{eq:GaussianGCRSI-temp}
\det(\Id_n + A_1) \det(\Id_n + A_2) \geq \det(\Id_n + C) \det(\Id_n + A_1 + A_2 - C) .
\end{equation}
 Evaluating the inequality (\ref{eq:GaussianGCRSI-assumption}) between quadratic forms on the diagonal $(y,y)$, it remains to note that
\[
A_1 + A_2 \geq \bar B + (a+b)^2 C ,
\]
and hence $A_1 + A_2 - C \geq \bar B + ((a+b)^2 -1) C$. Since $(a+b)^2 \geq 1$, the latter is positive semi-definite, and hence plugging this into (\ref{eq:GaussianGCRSI-temp}), the direction of the inequality is preserved and the assertion follows. 

If equality occurs when $(a+b)^2 > 1$, then the strict monotonicity of the determinant on $\Sym_{>0}(n)$ with respect to the positive-semi-definite order implies that necessarily $C = 0$. 
\end{proof}

Together with Theorem \ref{thm:intro-FRBL}, this concludes the proof of Theorem \ref{thm:GCRSI-FRBL}. 

\subsection{A more general version}

To properly analyze the equality case in Theorem \ref{thm:GCRSI-FRBL} (and hence Theorems \ref{thm:intro-GCRSI}, \ref{thm:intro-ab} and \ref{thm:intro-our-FRBL} from the Introduction), we will need a more general version of Proposition \ref{prop:GaussianGCRSI}. 

\begin{proposition} \label{prop:GenGaussianGCRSI}
Let $A_1,A_2,\bar B , C \in \Sym_{\geq 0}(n)$, and assume that: \begin{equation} \label{eq:GenGaussianGCRSI-assumption}
g_{A_1}(x_1) g_{A_2}(x_2) \leq H(x_1,x_2) g_C(a x_1+ bx_2) \;\;\; \forall x_1,x_2 \in \R^n , 
\end{equation}
for some $\abs{a},\abs{b},\abs{a+b} \geq 1$ and a log-concave function $H : \R^{2n} \rightarrow \R_+$ such that 
\[
H(y,y) = p_{\bar B} g_{\bar B}(y) \;\;\; \forall y \in \R^n ,  
\]
 for some constant $p_{\bar B} > 0$.  
Then
\[
\int g_{A_1} d\gamma \int g_{A_2} d\gamma  \leq \int  g_{\bar B} d\gamma \int g_C d\gamma ,
\]
with equality when $\abs{a+b} > 1$ implying that necessarily $C = 0$. 
\end{proposition}
\begin{proof}
Since $x_1 \mapsto H(x_1,0)$ and $x_2 \mapsto H(0,x_2)$ are log-concave, Lemma \ref{lem:log-concave} and (\ref{eq:GenGaussianGCRSI-assumption}) imply that
\[
\scalar{(A_i - c_i^2 C) x_i,x_i} \geq \scalar{v_i,x_i} + w_i \;\;\; \forall x_i \in \R^n , 
\]
for some vector $v_i \in \R^n$, scalar $w_i \in \R$ and $c_1 = a$, $c_2=b$. Since $|c_i| \geq 1$ by assumption, by letting $x_i$ tend to infinity, we have $A_i - C \geq 0$. Inspecting (\ref{eq:GenGaussianGCRSI-assumption}) on the diagonal $(y,y)$, a similar argument verifies that $A_1 + A_2 \geq \bar B + (a+b)^2 C$. 
The argument from the proof of Proposition \ref{prop:GaussianGCRSI} and the bound $A_1 + A_2 - C \geq \bar B + ((a+b)^2-1) C \geq 0$ then yield
\begin{align*}
\det(\Id_n + A_1) \det(\Id_n + A_2) & \geq \det(\Id_n + C) \det(\Id_n + A_1 + A_2 - C) \\
& \geq \det(\Id_n + C) \det(\Id_n + \bar B + ((a+b)^2-1) C) . 
\end{align*}
It follows that
\[
\int g_{A_1} d\gamma \int g_{A_2} d\gamma  \leq \int  g_{((a+b)^2-1) C} g_{\bar B} d\gamma \int g_C d\gamma ,
\]
concluding the proof, including the equality case. \end{proof}

Since the value of $p_{\bar B}$ above is totally immaterial, we can rewrite this as follows. 
\begin{cor} \label{cor:GenGaussian}
With the same assumptions as in the previous proposition, assume instead of (\ref{eq:GenGaussianGCRSI-assumption}) that
\[
p_{A_1} g_{A_1}(x_1) \cdot p_{A_2} g_{A_2}(x_2) \leq H(x_1,x_2)  \cdot p_C g_C(a x_1+ b x_2) \;\;\; \forall x_1,x_2 \in \R^n , 
\]
for some constants $p_{A_1}, p_{A_2}, p_C > 0$. Then
\[
\int p_{A_1} g_{A_1} d\gamma \int p_{A_2} g_{A_2} d\gamma  \leq  \int  p_{\bar B} g_{\bar B} d\gamma \int p_C g_C d\gamma ,
\]
with equality implying when $\abs{a+b}> 1$ that necessarily $C = 0$ and $p_{A_1} p_{A_2} = p_{\bar B} p_C$. 
\end{cor}
\begin{proof}
By Proposition \ref{prop:GenGaussianGCRSI}, we know that:
\[
\int p_{A_1} g_{A_1} d\gamma \int p_{A_2} g_{A_2} d\gamma  \leq \frac{p_{A_1} p_{A_2}}{p_{\bar B} p_{C}} \int  p_{\bar B} g_{\bar B} d\gamma \int p_C g_C d\gamma ,
\]
with equality implying when $\abs{a+b}> 1$ that necessarily $C = 0$. Since $p_{A_1} p_{A_2} \leq p_{\bar B} p_{C}$ (by testing our assumption at $x_1 = x_2 = 0$), the assertion immediately follows. 
\end{proof}

\section{Gaussian saturation in the GCMPI} \label{sec:GCMPI}

In this section we apply our machinery to derive the generalized conjugate Milman--Pajor inequality from Theorem \ref{thm:intro-GCMPI}.
We obtain the following forward-reverse functional formulation.

\begin{thm} \label{thm:GCMPI-FRBL}
Let $f_1,f_2 \in \F_{LC}^{(o)}(\gamma^n)$, and let $h_1 : \R^{2n} \rightarrow \R_+$ and $h_2 : \R^{n} \rightarrow \R_+$ denote two log-concave functions. 
Let $a,b \in \R \setminus \{0\}$ with
\begin{equation} \label{eq:GCMPI-ab}
 a^2,b^2 \leq 1 ~,~ 3 \min(a^2,b^2) + \max(a^2,b^2) \geq 1 . 
 \end{equation}
 If
\begin{equation} \label{eq:GCMPI-FRBL-assumption}
f_1(x_1) f_2(x_2) \leq h_1\brac{\frac{1}{b} x_1, \frac{1}{a} x_2} h_2(a x_1 + b x_2) \;\;\; \forall x_1,x_2 \in \R^n ,
\end{equation}
then
\begin{equation} \label{eq:GCMPI-FRBL-conclusion}
\int_{\R^n} f_1 d\gamma \int_{\R^n} f_2 d\gamma \leq \int_{\R^n} h_1(y,y) d\gamma(y) \int_{\R^n} h_2 d\gamma . 
\end{equation}
\end{thm}

Applying Theorem \ref{thm:GCMPI-FRBL} to $f_1 = \1_{K}$, $f_2 = \1_L$, $h_1 = \1_{\frac{1}{b} K \times \frac{1}{a} L}$ and $h_2 = \1_{a K + b L}$, the inequality statement of Theorem \ref{thm:intro-GCMPI} immediately follows (when $a,b \neq 0$; the case when $ab=0$ is trivially true). The analysis of equality is deferred to Section \ref{sec:equality}. 

\medskip

To establish Theorem \ref{thm:GCMPI-FRBL}, we apply Theorem \ref{thm:intro-FRBL} for our Brascamp--Lieb datum (\ref{eq:our-BL-datum}) with $\alpha = \frac{1}{b}$ and $\beta = \frac{1}{a}$, namely with $L_1(x) = (\frac{1}{b} x_1,\frac{1}{a} x_2)$ for $x = (x_1,x_2) \in \R^{n} \times \R^n$. 
It follows that the optimal constant on the right-hand-side of (\ref{eq:GCMPI-FRBL-conclusion}) is saturated by centered Gaussians satisfying (\ref{eq:GCMPI-FRBL-assumption}). Therefore, to establish (\ref{eq:GCMPI-FRBL-conclusion}), the remaining task is to show that for this datum, $\FR^{(\G)}_{LC} \leq 1$ (and hence, by testing constant functions in (\ref{eq:GCMPI-FRBL-assumption}), in fact $\FR^{(\G)}_{LC} = 1$). 

In other words, for all $A_1,A_2,C \in \Sym_{\geq 0}(n)$ and $B = \begin{pmatrix} B_1 & B_2 \\ B^*_2 & B_4 \end{pmatrix} \in \Sym_{\geq 0}(2n)$ such that
\[
g_{A_1}(x_1) g_{A_2}(x_2) \leq g_{B}\brac{\frac{1}{b} x_1,\frac{1}{a} x_2} g_C(a x_1+b x_2) \;\;\; \forall x_1,x_2 \in \R^n ,
\]
we need to show (assuming $a,b$ satisfy (\ref{eq:GCMPI-ab})) that
\[
\int_{\R^n} g_{A_1} d\gamma \int_{\R^n} g_{A_2} d\gamma \leq \int_{\R^n} g_{\bar B} d\gamma \int_{\R^n} g_{C} d\gamma ,
\]
where as usual $\bar B := B_1+B_2+B_2^*+B_4$. 
Equivalently, whenever
\begin{equation} \label{eq:GCMPI-matrix-assumption}
\begin{pmatrix} A_1 & 0 \\ 0 & A_2 \end{pmatrix} \geq \begin{pmatrix} \frac{1}{b^2} B_1 & \frac{1}{ab} B_2 \\ \frac{1}{ab} B^*_2 & \frac{1}{a^2} B_4 \end{pmatrix} + \begin{pmatrix} a^2 C & ab C \\ ab C & b^2 C \end{pmatrix} ,
\end{equation}
we would like to show that:
\begin{equation} \label{eq:GCMPI-det-conclusion}
\det(\Id_n + A_1) \det(\Id_n + A_2) \geq \det(\Id_n + \bar B) \det(\Id_n + C) .
\end{equation}

\subsection{Two Lemmas} 

\begin{lem} \label{lem:GCMPI1} 
Let $A,B,C \in \Sym_{\geq 0}(n)$ with $A,B \geq C$. Then for any $t_2 \geq t_1 \geq 0$,
\[
\det(\Id_n + t_1 A) \det(\Id_n + t_2 B) \geq \det(\Id_n + t_2 C) \det(\Id_n + t_1(A+B-C)) . 
\]
\end{lem}
\begin{proof}
The claim is obvious if $t_1 = 0$, so we may assume that $t_1 > 0$. 
By redefining $A,B,C$, we may in fact assume that $t_1 = 1$. Note that for $t_2 = 1$, the asserted inequality
\[
\det(\Id_n + A) \det(\Id_n + B) \geq \det(\Id_n + C) \det(\Id_n + A + B - C) 
\]
follows from Corollary \ref{cor:GaussianGCI} with $\mathcal{A} = A-C \geq 0$, $\mathcal{B} = B-C \geq 0$ and $\mathcal{C} = \Id_n + C$ (this 
was used several times in the previous section). It is therefore enough to show that $t_2 \mapsto \det(\Id_n + t_2 B) / \det(\Id_n + t_2 C)$ is non-decreasing in $t_2$. 

To see this, we may assume by approximation that $B \geq C > 0$. Therefore 
\begin{align*}
& B^{-1} + t_2 \Id_n \leq C^{-1} + t_2 \Id_n \\
& \Leftrightarrow B^{-1/2} (\Id_n + t_2 B) B^{-1/2} \leq C^{-1/2} (\Id_n + t_2 C) C^{-1/2} \\
& \Leftrightarrow B^{1/2} (\Id_n + t_2 B)^{-1} B^{1/2} \geq C^{1/2} (\Id_n + t_2 C)^{-1} C^{1/2} \\
& \Rightarrow \tr((\Id_n + t_2 B)^{-1} B) \geq \tr((\Id_n + t_2 C)^{-1} C) \\
& \Leftrightarrow \frac{d}{dt_2} \brac{\log \det(\Id_n + t_2 B) - \log \det(\Id_n + t_2 C)} \geq 0 ,
\end{align*}
because $\frac{d}{dt} \log \det(A_t) = \tr((A_t)^{-1} \frac{d}{dt} A_t)$. 
This concludes the proof. 
\end{proof}

\begin{lem} \label{lem:GCMPI2} 
Let $C,Z \in \Sym_{\geq 0}(n)$, and let $\alpha,\beta \geq 0$ with $\alpha + \beta = 1$. Then:
\begin{equation} \label{eq:GCMPI2}
\det(\Id_n + \alpha C) \det(\Id_n + \beta C + Z) \geq \det(\Id_n + C) \det(\Id_n + \alpha Z) ,
\end{equation}
with equality when $\alpha=1$. When $\alpha = 0$, equality occurs if and only if $Z =0$, and when $\alpha \in (0,1)$, equality occurs if and only if $C = Z = 0$. 
\end{lem}
\begin{proof}
By choosing an appropriate orthonormal basis, we may assume that $C$ is diagonal. By monotonicity of the determinant on $\Sym_{\geq 0}(n)$, note that
\begin{align*}
& \det(\Id_n + \alpha C) \det(\Id_n + \beta C + Z) \\
&= \det((\Id_n + \alpha C) (\Id_n + \beta C) + \sqrt{\Id_n + \alpha C} Z \sqrt{\Id_n + \alpha C}) \\
& \geq \det(\Id_n + C + \sqrt{\Id_n + \alpha C} Z \sqrt{\Id_n + \alpha C}) \\
& = \det(\Id_n + C) \det(\Id_n + D  Z D ) , 
\end{align*}
where $D = \sqrt{(\Id_n + \alpha C)(\Id_n + C)^{-1}}$. Consequently, it is enough to show that
\[
\det(\Id_n + D Z D) \geq \det(\Id_n + \alpha Z) .
\]
If $\alpha = 0$ there is nothing to prove, so we may assume $\alpha \in (0,1]$. 
Renaming $Z$ to be $\alpha Z$ and $D$ to be  $D / \sqrt{\alpha} = \sqrt{(\Id_n + \alpha C)(\alpha \Id_n + \alpha C)^{-1}}$, note that $D \geq \Id_n$. 
The inequality $\det(\Id_n + D Z D) \geq \det(\Id_n +Z)$ then follows from Lemma \ref{lem:preparatory-det}. 

If equality occurs in (\ref{eq:GCMPI2}) then all of the inequalities we have used above must be equalities. In particular, the very first inequality we used implies that $\alpha \beta C^2 = 0$, and so when $\alpha \in (0,1)$ then necessarily $C = 0$; plugging this into (\ref{eq:GCMPI2}), it follows that $Z = 0$ as well. When $\alpha = 0$, (\ref{eq:GCMPI2}) directly implies that equality occurs if and only if $Z = 0$, whereas when $\alpha=1$, (\ref{eq:GCMPI2}) holds with equality identically. \end{proof}

\subsection{GCMPI for Gaussians}

\begin{proposition} \label{prop:GaussianGCMPI}
Let $A_1,A_2,C \in \Sym_{\geq 0}(n)$ and $B \in \Sym_{\geq 0}(2n)$, and let  $a,b \in \R \setminus \{0\}$ satisfy (\ref{eq:GCMPI-ab}). Then (\ref{eq:GCMPI-matrix-assumption}) implies (\ref{eq:GCMPI-det-conclusion}). Assuming (\ref{eq:GCMPI-matrix-assumption}), if equality holds in  (\ref{eq:GCMPI-det-conclusion}) then necessarily $C =0$, $\bar B = A_1 + A_2$, and in addition:
\begin{enumerate}
\item If $a^2 < 1$ then $A_2 = 0$. \item If $b^2 < 1$ then $A_1 = 0$. \item If $\max(a^2,b^2) < 1$ then equality holds  in (\ref{eq:GCMPI-det-conclusion}) if and only if $A_1 = A_2 = C = 0$ and $B = 0$.  
\end{enumerate}
\end{proposition}
\begin{remark}
When equality holds in (\ref{eq:GCMPI-det-conclusion}) then in addition we have $B_4 = B_2 = 0$ if $a^2 < 1$, $B_1 = B_2 = 0$ if $b^2 < 1$, and $A_1 A_2 = 0$ if $a^2=b^2 = 1$, 
but we will not require this here. 
\end{remark}
\begin{proof}
Assume without loss of generality that $a^2 \leq b^2 (\leq 1$). We may rewrite (\ref{eq:GCMPI-matrix-assumption})  as
\begin{equation} \label{eq:GCMPI-again}
\begin{pmatrix} \frac{1}{a^2} A_1 & 0 \\ 0 & \frac{1}{b^2} A_2 \end{pmatrix} \geq \frac{1}{a^2 b^2} \begin{pmatrix} B_1 &  B_2 \\  B^*_2 & B_4 \end{pmatrix} + \begin{pmatrix} C & C \\ C & C \end{pmatrix} .
\end{equation}
Since $B_1,B_4 \geq 0$, we see that $\frac{A_1}{a^2} , \frac{A_2}{b^2} \geq C$. 
Applying Lemma \ref{lem:GCMPI1} with these three matrices and $t_1 = a^2 \leq b^2 = t_2$, we deduce:
\[
\det(\Id_n + A_1) \det(\Id_n + A_2) \geq \det(\Id_n + b^2 C) \det\brac{\Id_n + a^2 \brac{\frac{A_1}{a^2} + \frac{A_2}{b^2} -C}} . 
\]
Denote $Z = \frac{A_1}{a^2} + \frac{A_2}{b^2} - 4 C$, and observe by inspecting (\ref{eq:GCMPI-again}) on diagonal entries $(y,y)$ that $Z \geq \frac{1}{a^2 b^2} \bar B \geq 0$. Therefore, to conclude (\ref{eq:GCMPI-det-conclusion}) it is enough to show that:
\[
\det(\Id_n + b^2 C) \det(\Id_n + a^2 (3C + Z)) \geq \det(\Id_n + C) \det(\Id_n + a^2 b^2 Z) . 
\]
Since $3 a^2 \geq 1-b^2$, we have 
\begin{equation} \label{eq:GCMPI-temp1}
\det(\Id_n + a^2 ( 3C + Z)) \geq \det(\Id_n + (1-b^2) C + a^2 Z),
\end{equation}
 and hence it is enough to show
\begin{equation} \label{eq:GCMPI-temp2}
\det(\Id_n + b^2 C) \det(\Id_n + (1-b^2) C + a^2 Z) \geq \det(\Id_n + C) \det(\Id_n + a^2 b^2 Z) .
\end{equation}
But this follows from Lemma \ref{lem:GCMPI2} with $\alpha = b^2,\beta = 1-b^2 \geq 0$ (and $a^2 Z$). 

If equality holds in (\ref{eq:GCMPI-det-conclusion}) then all of the inequalities we have used above must be equalities. In particular, by Lemma \ref{lem:GCMPI2}, equality holds in (\ref{eq:GCMPI-temp2}) when $b^2 \in (0,1)$ if and only if $C = Z = 0$. Since $Z = \frac{A_1}{a^2} + \frac{A_2}{b^2}$ with $A_1,A_2 \geq 0$, it follows that $A_1=A_2 = 0$. Plugging $A_1 = A_2 = C = 0$ into (\ref{eq:GCMPI-again}) we deduce that $B \leq 0$, but since $B \geq 0$ it follows that $B = 0$. 

When $b^2 = 1$ then $3 a^2 > 0 = 1-b^2$, and so equality in (\ref{eq:GCMPI-temp1}) implies that $C = 0$. Inspecting (\ref{eq:GCMPI-again}) on diagonal entries $(y,y)$ shows that $A_1 + A_2 \geq A_1 + a^2 A_2 \geq \bar B$, and together with the equality in (\ref{eq:GCMPI-det-conclusion}) and Lemma \ref{lem:GaussianGCI}, this gives:
\begin{align*}
& \det(\Id_n + A_1) \det(\Id_n + A_2) = \det(\Id_n + \bar B) \\
& \leq \det(\Id_n + A_1 +A_2) \leq \det(\Id_n + A_1) \det(\Id_n + A_2) .
\end{align*}
This means that we must have equality everywhere above, so in particular $\bar B = A_1 + A_2$. If we assume that $a^2 < 1$, then equality in $A_1 + A_2 \geq A_1 + a^2 A_2$ implies that $A_2 = 0$. Exchanging the roles of $a,b$, the proof is complete. 
\end{proof}

\subsection{A more general version}

To properly analyze the equality case in Theorem \ref{thm:GCMPI-FRBL} (and hence Theorem \ref{thm:intro-GCMPI} from the Introduction), we will need a more general version of Proposition \ref{prop:GaussianGCMPI}. 

\begin{proposition} \label{prop:GenGaussianGCMPI}
Let $A_1,A_2,\bar B , C \in \Sym_{\geq 0}(n)$, and assume that: 
\begin{equation} \label{eq:GenGaussianGCMPI-assumption}
g_{A_1}(x_1) g_{A_2}(x_2) \leq H\brac{\frac{1}{b} x_1,\frac{1}{a} x_2} g_C(a x_1+ bx_2) \;\;\; \forall x_1,x_2 \in \R^n , 
\end{equation}
for some $a,b \in \R \setminus \{0\}$ satisfying (\ref{eq:GCMPI-ab}) and a log-concave function $H : \R^{2n} \rightarrow \R_+$ such that 
\[
H(y,y) = p_{\bar B} g_{\bar B}(y) \;\;\; \forall y \in \R^n ,  
\]
for some constant $p_{\bar B} > 0$. Then
\begin{equation} \label{eq:GenGaussianGCMPI-conclusion}
\int g_{A_1} d\gamma \int g_{A_2} d\gamma  \leq \int  g_{\bar B} d\gamma \int g_C d\gamma .
\end{equation}
If equality holds in (\ref{eq:GenGaussianGCMPI-conclusion}) then the same conclusion holds as in Proposition \ref{prop:GaussianGCMPI} (with the statement that $B=0$ replaced by $\bar B=0$). In particular, if $\max(a^2,b^2) < 1$ then equality holds in (\ref{eq:GenGaussianGCMPI-conclusion}) if and only if $A_1 = A_2 = \bar B = C = 0$. 
\end{proposition}
\begin{proof}
Since $\R^n \times \R^n \ni (x_1,x_2) \mapsto H(\frac{1}{b} x_1,\frac{1}{a} x_2)$ remains log-concave, Lemma \ref{lem:log-concave} and (\ref{eq:GenGaussianGCMPI-assumption}) imply that
\[
\scalar{\left [ \begin{pmatrix} A_1 & 0 \\ 0 & A_2 \end{pmatrix} - \begin{pmatrix} a^2 C & ab C \\ ab C & b^2 C \end{pmatrix} \right ] x , x} \geq \scalar{v,x} + w \;\;\; \forall x \in \R^{2n} ,
\]
for some vector $v \in \R^{2n}$ and scalar $w \in \R$.  Letting $x$ tend to infinity, we deduce that
\[
\begin{pmatrix} A_1 & 0 \\ 0 & A_2 \end{pmatrix} \geq \begin{pmatrix} a^2 C & ab C \\ ab C & b^2 C \end{pmatrix} .
\]

On the other hand, inspecting (\ref{eq:GenGaussianGCMPI-assumption}) for $x_1 = b y$ and $x_2 = a y$, we have
\[
\scalar{ (b^2 A_1 + a^2 A_2 - \bar B - 4 a^2 b^2 C) y, y} \geq -2 \log p_{\bar B} \;\;\; \forall y \in \R^n . 
\]
Letting $y$ tend to infinity and denoting $Z = \frac{A_1}{a^2} + \frac{A_2}{b^2} - 4 C$, we deduce
\begin{equation} \label{eq:GenGaussianGCMPI-temp}
a^2 b^2 Z = b^2 A_1 + a^2 A_2  - 4 a^2 b^2 C \geq \bar B . 
\end{equation}
We can now conclude (\ref{eq:GenGaussianGCMPI-conclusion}) and its equality case exactly as in Proposition \ref{prop:GaussianGCMPI}. The only difference is that once it is established that $A_1 = A_2 = C =0$, then (\ref{eq:GenGaussianGCMPI-temp}) implies that $\bar B = 0$. 
\end{proof}

As with Corollary \ref{cor:GenGaussian}, since the value of $p_{\bar B}$ above is totally immaterial, we can rewrite this as follows. 
\begin{cor} \label{cor:GenGaussianGCMPI}
With the same assumptions as in the previous proposition, assume instead of (\ref{eq:GenGaussianGCMPI-assumption}) that
\[
p_{A_1} g_{A_1}(x_1) \cdot p_{A_2} g_{A_2}(x_2) \leq H\brac{\frac{1}{b} x_1,\frac{1}{a} x_2}  \cdot p_C g_C(a x_1+ b x_2) \;\;\; \forall x_1,x_2 \in \R^n , 
\]
for some constants $p_{A_1}, p_{A_2}, p_C > 0$. Then
\[
\int p_{A_1} g_{A_1} d\gamma \int p_{A_2} g_{A_2} d\gamma  \leq  \int  p_{\bar B} g_{\bar B} d\gamma \int p_C g_C d\gamma ,
\]
with equality implying that $p_{A_1} p_{A_2} = p_{\bar B} p_C$ and that the same conclusions as in Proposition \ref{prop:GenGaussianGCMPI} hold. In particular, 
equality holds when $\max(a^2,b^2) < 1$ if and only if $A_1 = A_2 = \bar B = C = 0$ and $p_{A_1} p_{A_2} = p_{\bar B} p_C$. 
\end{cor}

\section{Analysis of equality} \label{sec:equality}

We now turn to analyze the equality in Theorems \ref{thm:GCRSI-FRBL} and \ref{thm:GCMPI-FRBL}. Since our proof of the Gaussian FRBL Theorem \ref{thm:intro-FRBL} involved several approximation arguments, tracking the cases of equality in that generality seems genuinely intractable. Nevertheless, for our particular applications, we are able to show the following. 

\begin{thm} \label{thm:GCRSI-equality}
Let $f_1,f_2 \in \F_{LC}^{(o)}(\gamma^n)$, and let $h_1 : \R^{2n} \rightarrow \R_+$ and $h_2 : \R^{n} \rightarrow \R_+$ denote two log-concave functions. 
Let $a,b \in \R$ with $\abs{a},\abs{b} \geq 1$ and $\abs{a+b} > 1$. Assume that
\begin{equation} \label{eq:GCRSI-FRBL-equality-assumption}
f_1(x_1) f_2(x_2) \leq h_1(x_1, x_2) h_2(a x_1 + b x_2) \;\;\; \forall x_1,x_2 \in \R^n ,
\end{equation}
and that we have equality
\begin{equation} \label{eq:GCRSI-FRBL-equality}
\int_{\R^n} f_1 d\gamma \int_{\R^n} f_2 d\gamma  = \int_{\R^n} h_1(y,y) d\gamma(y) \int_{\R^n} h_2 d\gamma  > 0 . 
\end{equation}
Then $h_2 \equiv c > 0$ is a constant function, $c \cdot h_1(y,y) = f_1(y) f_2(y)$ for almost-every $y \in \R^n$,  
and $f_1(x) = \bar f_1(P_E x)$ and $f_2(x) = \bar f_2(P_{E^{\perp}} x)$ for some linear subspace $E \subset \R^n$ and almost every $x \in \R^n$. 
\end{thm}

Applying Theorem \ref{thm:GCRSI-equality} with $f_1 = \1_{K}$, $f_2 = \1_L$, $h_1 = \1_{K \times L}$ and $h_2 = \1_{a K + b L}$, 
the equality case of Theorem \ref{thm:intro-ab} from the Introduction immediately follows.

\begin{thm} \label{thm:GCMPI-equality}
Let $f_1,f_2 \in \F_{LC}^{(o)}(\gamma^n)$, and let $h_1 : \R^{2n} \rightarrow \R_+$ and $h_2 : \R^{n} \rightarrow \R_+$ denote two log-concave functions. 
Let $a,b \in \R \setminus \{0\}$ such that $3 \min(a^2,b^2) + \max(a^2,b^2) \geq 1$ and $\max(a^2,b^2) \leq 1$. Assume that
\begin{equation} \label{eq:GCMPI-FRBL-equality-assumption}
f_1(x_1) f_2(x_2) \leq h_1\brac{\frac{1}{b} x_1, \frac{1}{a} x_2} h_2(a x_1 + b x_2) \;\;\; \forall x_1,x_2 \in \R^n ,
\end{equation}
and that we have equality
\begin{equation} \label{eq:GCMPI-FRBL-equality}
\int_{\R^n} f_1 d\gamma \int_{\R^n} f_2 d\gamma  = \int_{\R^n} h_1(y,y) d\gamma(y) \int_{\R^n} h_2 d\gamma  > 0 . 
\end{equation}
Then $h_2$ is constant, and in addition:
\begin{enumerate}
\item If $a^2 < b^2=1$ then $f_2$ is constant and $h_1(y,y) = c \cdot f_1(y)$ for a.e.~$y \in \R^n$. 
\item If $b^2 < a^2=1$ then $f_1$ is constant and $h_1(y,y) = c \cdot f_2(y)$ for a.e.~$y \in \R^n$. 
\item If $\max(a^2,b^2) < 1$ then $f_1$, $f_2$, $h_1$ and $h_2$ are all constant.
\end{enumerate}
\end{thm}

Applying Theorem \ref{thm:GCMPI-equality} with $f_1 = \1_{K}$, $f_2 = \1_L$, $h_1 = \1_{\frac{1}{b}K \times \frac{1}{a} L}$ and $h_2 = \1_{a K + b L}$, 
the equality case of Theorems \ref{thm:intro-CMPI} and \ref{thm:intro-GCMPI} from the Introduction immediately follows. 

\subsection{Partial Gaussian saturation}

Our proof of Theorems \ref{thm:GCRSI-equality} and Theorems \ref{thm:GCMPI-equality} is based on the following. 

\begin{proposition}[Partial Gaussian saturation of equality] \label{prop:partial-Gaussian-saturation}
Let $f_1,f_2 \in \F_{LC}^{(o)}(\gamma^n)$, and let $h_1 : \R^{2n} \rightarrow \R_+$ and $h_2 : \R^{n} \rightarrow \R_+$ denote two log-concave functions. 
Let $\alpha,\beta, a,b \in \R \setminus \{0\}$. 
Assume that
\begin{equation} \label{eq:FRBL-equality-assumption}
f_1(x_1) f_2(x_2) \leq \frac{1}{\FR_{LC}} h_1(\alpha x_1, \beta x_2) h_2(a x_1 + b x_2) \;\;\; \forall x_1,x_2 \in \R^n ,
\end{equation}
where $\FR_{LC} \in (0,\infty)$ is the forward-reverse constant corresponding to our Brascamp--Lieb datum (\ref{eq:our-BL-datum}), 
and that we have equality
\begin{equation} \label{eq:FRBL-equality}
\int_{\R^n} f_1 d\gamma \int_{\R^n} f_2 d\gamma  = \int_{\R^n} h_1(y,y) d\gamma(y) \int_{\R^n} h_2 d\gamma  > 0 . 
\end{equation}
Denote
\[
\mathcal{A}_i := \Cov(f_i \gamma),\quad \mathcal{\bar B} :=   {\rm Cov} ( y \mapsto h_1(y,y) \gamma(y)) , \quad
\mathcal{C}:= \Cov( h_2 \gamma ) ,
\]
and
\[
A_i:= \mathcal{A}_i^{-1} - \Id_n,\quad \bar B := \mathcal{\bar B}^{-1} - \Id_n , \quad C:= \mathcal{C}^{-1} - \Id_n  . 
\]
Then $A_1,A_2 , \bar B , C \in \Sym_{\geq 0}(n)$, and there exist constants $p_{A_1},p_{A_2}, p_{\bar B}, p_C > 0$ and a log-concave function $H : \R^{2n} \to \R_+$ such that
\begin{align*}
 p_{A_1} g_{A_1}(x_1) \cdot p_{A_2} g_{A_2}(x_2) &\leq \frac{1}{\FR_{LC}} H(\alpha x_1, \beta x_2) \cdot p_C g_C(a x_1 + b x_2) \;\;\; \forall (x_1,x_2) \in \R^n \times \R^n , \\
 H(y,y) & = p_{\bar B} g_{\bar B}(y) \;\;\; \forall y \in \R^n ,
\end{align*}
and
\begin{equation} \label{eq:partial-equality}
\int_{\R^n} p_{A_1} g_{A_1} d\gamma \int_{\R^n} p_{A_2} g_{A_2} d\gamma = \int_{\R^n} p_{\bar B} g_{\bar B} d\gamma \int_{\R^n} p_C g_C d\gamma . 
\end{equation}
\end{proposition}

\begin{proof}
Note that $\mathcal{A}_i, \mathcal{\bar B}, \mathcal{C} \leq \Id_n$ by Proposition \ref{prop:cov}, and hence $A_i, \bar B, C \in \Sym_{\geq 0}(n)$. 
Since we have equality in (\ref{eq:FRBL-equality}) and all four integrals are positive, we may normalize $f_1,f_2,h_1,h_2$ and assume that all four integrals are equal to $1$ without altering (\ref{eq:FRBL-equality-assumption}). 

As in the proofs of Steps \ref{it:step0} and \ref{it:step1} in Section \ref{sec:FRBL-proof}, let $f^{(0)} = f$ for $f \in \{f_1,f_2,h_1,h_2\}$, and 
inductively define $f^{(N+1)} : \R^n \rightarrow \R_+$ for $f \in \{f_1,f_2,h_2\}$ and $h_1^{(N+1)} : \R^{2n} \rightarrow \R_+$ by
\begin{align*}
    f^{(N+1)}(x) &:= \int_{\R^{n}} f^{(N)}\brac{ \frac{x+y}{\sqrt2} } f^{(N)} \brac{ \frac{x-y}{\sqrt2} }\, d\gamma_{\Id_n}(y) ~,~ x \in \R^{n} ,
    \\
    h_1^{(N+1)}(x) & := 
    \int_{\R^{2n}} h_1^{(N)} \brac{\frac{x+y}{\sqrt{2}}} h_1^{(N)}\brac{\frac{x-y}{\sqrt{2}}}\, d\gamma_{\Gamma_1}(y) ~,~ x = (x_1,x_2) \in \R^{n} \times \R^n  \\
    & = \int_{\R^n} h_1^{(N)} \brac{ \frac{(x_1,x_2)}{\sqrt2} + \frac{(y,y)}{\sqrt2} } h_1^{(N)}\brac{\frac{(x_1,x_2)}{\sqrt2} - \frac{(y,y)}{\sqrt2} }\, d\gamma_{\Id_n}(y)  . 
\end{align*}
Denoting for $f \in  \{f_1,f_2,h_2\}$ and $\mathfrak{f} \in \{\mathfrak{f}_1,\mathfrak{f}_2,\mathfrak{h}_2\}$ (respectively),
\[
\mathfrak{f}^{(N)} := f^{(N)} \gamma_{\Id_n} = f^{(N)} \frac{g_{\Id_n}}{(2 \pi)^{n/2}} ~,~ \mathfrak{h}_1^{(N)} := h_1^{(N)} \frac{g_{\frac12 \Id_{2n}}}{(2\pi)^{n/2}}  ,
\]
one readily checks using (\ref{eq:parallel}) and $\scalar{\frac12 \Id_{2n} (y,y), (y,y)} = \scalar{\Id_n y, y}$ that these satisfy:
\begin{align*}
& \mathfrak{f}^{(N+1)}(x) = \int_{\R^n} \mathfrak{f}^{(N)}\brac{\frac{x+y}{\sqrt2}}\mathfrak{f}^{(N)}\brac{\frac{x-y}{\sqrt2} } \, dy = 2^{n/2}  (\mathfrak{f}^{(N)} \ast \mathfrak{f}^{(N)}) (\sqrt{2} x) ,\quad x\in \R^n , \\
& \mathfrak{h}_1^{(N+1)}(x_1,x_2) = \int_{\R^n} \mathfrak{h}_1^{(N)}\brac{ \frac{(x_1,x_2)}{\sqrt2} + \frac{(y,y)}{\sqrt2} } \mathfrak{h}_1^{(N)}\brac{ \frac{(x_1,x_2)}{\sqrt2} - \frac{(y,y)}{\sqrt2} }\, dy,\quad (x_1,x_2)\in\R^{n} \times \R^n. 
\end{align*}
We will further denote the restrictions of $h^{(N)}_1$ and $\mathfrak{h}^{(N)}_1$ on the diagonal by
\[
\bar h^{(N)}_1(y) := h^{(N)}_1(y,y) ~,~  \mathfrak{\bar h}^{(N)}_1(y) := \mathfrak{h}^{(N)}_1(y,y) = \bar h^{(N)}_1(y) \frac{g_{\Id_n}(y)}{(2 \pi)^{n/2}}  ~,~ y \in \R^n ,
\]
and note that
\[
\mathfrak{\bar h}_1^{(N+1)}(y) = 2^{n/2}  (\mathfrak{\bar h}_1^{(N)} \ast \mathfrak{\bar h}_1^{(N)}) (\sqrt{2} y) ,\quad y\in \R^n  .
\]

Recall that $\FR_{LC} = \FR_{LC}^{(e)}= \FR^{(\G)}_{LC}$ by Theorem \ref{thm:FRBL}. It follows by Proposition \ref{p:BLTrick} (\ref{it:BLTrick1}) and (\ref{it:BLTrick4}) applied with $D = \frac{1}{\FR_{LC}}$ that $f^{(N)} : \R^n \rightarrow \R_+$ and $h_1^{(N)} : \R^{2n} \rightarrow \R_+$ remain log-concave for all $N \geq 1$, and that
\begin{equation}  \label{eq:equality-proof-inq}
f^{(N)}_1(x_1) f^{(N)}_2(x_2) \leq \frac{1}{\FR_{LC}} h^{(N)}_1(\alpha x_1,\beta x_2) h^{(N)}_2(a x_1 + b x_2) \;\;\; \forall x_1,x_2 \in \R^n . 
\end{equation}
Furthermore, Proposition \ref{p:BLTrick} verifies for $\mathfrak{f} \in \{\mathfrak{f}_1,\mathfrak{f}_2,\mathfrak{h}_2 , \mathfrak{\bar h}_1 \}$ that for all $N \geq 1$,
\[
\int_{\R^n} \mathfrak{f}^{(N)}(x) dx = 1  ~,~ \bary(\mathfrak{f}^{(N)}_1) = \bary(\mathfrak{f}^{(N)}_2) = 0 ,
\]
and
\begin{equation} \label{eq:center-increasing}
 \bary(\mathfrak{h}^{(N)}_2) = 2^{N/2}  \bary(\mathfrak{h}^{(0)}_2) ~,~
 \bary(\mathfrak{\bar h}^{(N)}_1) = 2^{N/2} \bary(\mathfrak{\bar h}^{(0)}_1) .
\end{equation}
We will see in Lemma \ref{lem:equality-centered} below that the above is enough to imply that necessarily the barycenters in (\ref{eq:center-increasing}) are all equal to $0$; let us proceed under this assumption. 

Since $\mathfrak{f}^{(0)}_1,\mathfrak{f}^{(0)}_2, \mathfrak{\bar h}^{(0)}_1, \mathfrak{h}^{(0)}_2$ are log-concave (in fact, more log-concave than $g_{\Id_n}$) and integrable, they are all bounded (e.g.~by Lemma \ref{lem:log-concave-cov}). 
Applying a local version of the Central Limit Theorem as in Step \ref{it:step0} in Section \ref{sec:FRBL-proof}, we have the following pointwise convergence on $\R^n$:
\[
\lim_{N\to\infty}\mathfrak{f}_i^{(N)} = \gamma_{\Cov(\mathfrak{f}^{(0)}_i)} = \gamma_{\mathcal{A}_i} ,\;\; 
\lim_{N\to\infty} \mathfrak{\bar h}_1^{(N)} = \gamma_{\Cov(\mathfrak{\bar h}^{(0)}_1)} = \gamma_{\mathcal{\bar B}} , \;\;
\lim_{N\to\infty} \mathfrak{h}_2^{(N)} = \gamma_{\Cov(\mathfrak{h}^{(0)}_2)} = \gamma_{\mathcal{C}} .
\]	
Consequently, we have the pointwise convergence:
\[
\lim_{N\to \infty} f_i^{(N)} = \frac{\gamma_{\mathcal{A}_i}}{\gamma_{\Id_n}} = p_{A_i} g_{A_i} , \;\; \lim_{N\to\infty} \bar h_1^{(N)} = \frac{\gamma_{\mathcal{\bar B}}}{\gamma_{\Id_n}} = p_{\bar B} g_{\bar B} , \;\; \lim_{N\to \infty} h_2^{(N)} = \frac{\gamma_{\mathcal{C}}}{\gamma_{\Id_n}} = p_{C} g_{C} ,
\]
for appropriate constants $p_{A_i}, p_{\bar B}, p_{C} > 0$. Therefore, 
defining
\[
H(x_1,x_2) := \liminf_{N \rightarrow \infty} h^{(N)}_1(x_1,x_2) ~,~  x_1,x_2 \in \R^n , 
\]
and taking the pointwise limit inferior as $N \rightarrow \infty$ in (\ref{eq:equality-proof-inq}), we deduce: 
\[
p_{A_1} g_{A_1}(x_1) \cdot p_{A_2} g_{A_2}(x_2) \leq \frac{1}{\FR_{LC}} H(\alpha x_1,\beta x_2) \cdot p_{C} g_{C}(a x_1 + b x_2) \;\; \forall (x_1,x_2) \in \R^n \times \R^n . 
\]
It remains to note that $H(x_1,x_2)$ is log-concave as the limit inferior of log-concave functions, that $H(y,y) = p_{\bar B} g_{\bar B}(y)$, and that all 4 integrals in (\ref{eq:partial-equality}) are equal to $1$ and hence (\ref{eq:partial-equality}) holds, because e.g.
\[
\int_{\R^n} p_{C} g_{C} d\gamma = \int_{\R^n} \frac{\gamma_{\mathcal{C}}}{\gamma_{\Id_n}} d\gamma = \int_{\R^n} d\gamma_{\mathcal{C}} = 1.
\]
This concludes the proof, modulo Lemma \ref{lem:equality-centered} below. 
\end{proof}

\begin{lem} \label{lem:equality-centered}
With the same assumptions and notation as in the proof of Proposition \ref{prop:partial-Gaussian-saturation}, we have
\[
\bary(\mathfrak{\bar h}^{(0)}_1) = 0 ~,~ \bary(\mathfrak{h}^{(0)}_2) = 0 .
\]
\end{lem}

\begin{proof}
For ease of notation, we abbreviate $\mathfrak{h}^{(N)}_1 = \mathfrak{\bar h}^{(N)}_1$ (there will not be any confusion with the notation $\mathfrak{h}^{(N)}_1$ used in the proof of Proposition \ref{prop:partial-Gaussian-saturation}).

Denote $\xi_j := \bary(\mathfrak{h}^{(0)}_j)$, and recall that  $\xi^{(N)}_j:= \bary(\mathfrak{h}^{(N)}_j) = 2^{N/2} \xi_j$, $j=1,2$. 
Evaluating (\ref{eq:equality-proof-inq}) at $(x_1,x_2) = (\frac{1}{\alpha} y,\frac{1}{\beta} y)$, we have for all $N \geq 1$, 
\[
f^{(N)}_1(y/\alpha) f^{(N)}_2(y/\beta) \leq \frac{(2\pi)^n}{\FR_{LC}} e^{\frac{(a/\alpha+b/\beta)^2+1}{2} |y|^2} \mathfrak{h}^{(N)}_1(y) \mathfrak{h}^{(N)}_2((a/\alpha+b/\beta)y) \;\;\; \forall y \in \R^n .
\]
Also recall that $\int \mathfrak{h}^{(N)}_j(y) dy = 1$, and that $\Cov(\mathfrak{h}^{(N)}_j) = \Cov(\mathfrak{h}^{(0)}_j)$ by Proposition \ref{p:BLTrick}. 

Now, assume in the contrapositive that $\xi_j \neq 0$. 
Then, denoting the half-space $H_\theta := \{ y \in \R^n : \scalar{y,\theta} \leq 0 \}$, we have by the Markov-Chebyshev inequality applied to the projection of $\mathfrak{h}^{(N)}_j$ onto the linear span of $u_j = \xi_j / |\xi_j| = \xi^{(N)}_j / |\xi^{(N)}_j|$:
\[
\int_{H_{\xi_j}} \mathfrak{h}^{(N)}_j(y) dy \leq \frac{\sscalar{\Cov(\mathfrak{h}^{(N)}_j) u_j, u_j}}{|\xi^{(N)}_j|^2} = \frac{\sscalar{\Cov(\mathfrak{h}^{(0)}_j) \xi_j, \xi_j}}{2^N |\xi_j|^4} .
\]

Recall that $\mathfrak{h}^{(N)}_j$ remains more log-concave that $g_{\Id_n}$, and in particular log-concave, for all $N \geq 1$. Hence, by Lemma \ref{lem:log-concave-cov}, 
\[
\norm{\mathfrak{h}^{(N)}_j}_{L^\infty} \leq \frac{C_n}{\det^{\frac{1}{2}} \Cov(\mathfrak{h}^{(N)}_j)} = \frac{C_n}{\det^{\frac{1}{2}} \Cov(\mathfrak{h}^{(0)}_j)} := M_j < \infty \;\;\; \forall N \geq 1 . 
\]

Now, if $\xi_1 \neq 0$ then:
\[
 \frac{\FR_{LC}}{(2\pi)^n} e^{-\frac{(a/\alpha+b/\beta)^2+1}{2} |y|^2} f^{(N)}_1(y/\alpha) f^{(N)}_2(y/\beta) \leq M_2  \mathfrak{h}^{(N)}_1(y) , 
\]
and integrating over $H_{\xi_1}$ we obtain:
\[
 \frac{\FR_{LC}}{(2\pi)^n} \int_{H_{\xi_1}} e^{-\frac{(a/\alpha+b/\beta)^2+1}{2} |y|^2} f^{(N)}_1(y/\alpha) f^{(N)}_2(y/\beta)  dy \leq  
\frac{M_2 \sscalar{\Cov(\mathfrak{h}^{(0)}_1) \xi_1, \xi_1}}{2^N |\xi_1|^4} .
\]
Recall that $\lim_{N\to \infty} f_i^{(N)} =  p_{A_i} g_{A_i}$ (pointwise), and so Fatou's lemma implies that the limit inferior as $N \rightarrow \infty$ of the left-hand-side is bounded below by an 
integral of centered Gaussians over a half-plane, and hence is strictly positive; on the other hand, the right-hand-side tends to $0$, a contradiction. 

Similarly, if $\xi_2 \neq 0$, then:
\[
 \frac{\FR_{LC}}{(2\pi)^n} e^{-\frac{(a/\alpha+b/\beta)^2+1}{2} |y|^2} f^{(N)}_1(y/\alpha) f^{(N)}_2(y/\beta)  \leq M_1  \mathfrak{h}^{(N)}_2((a/\alpha+b/\beta) y) , 
\]
and so integrating over $H_{\xi_2}$ we obtain
\[
 \frac{\FR_{LC}}{(2\pi)^n} \int_{H_{\xi_2}} e^{-\frac{(a/\alpha+b/\beta)^2+1}{2} |y|^2} f^{(N)}_1(y/\alpha) f^{(N)}_2(y/\beta)  dy \leq  
\frac{M_1 \sscalar{\Cov(\mathfrak{h}^{(0)}_2) \xi_2, \xi_2}}{2^N (a/\alpha+b/\beta)^n|\xi_2|^4} ,
\]
and again we get a contradiction as $N \rightarrow \infty$. It follows that $\xi_1=\xi_2 = 0$, concluding the proof. 
\end{proof}

\begin{remark}
It is clear from the proof that Proposition \ref{prop:partial-Gaussian-saturation} applies in much greater generality than for our particular Brascamp--Lieb datum (\ref{eq:our-BL-datum}), but we do not insist on this here. 
\end{remark}

\subsection{Conclusions}

\begin{proof}[Proof of Theorem \ref{thm:GCRSI-equality}]
Recall that in this case $\FR_{LC} = \FR_{LC}^{(\G)} = 1$ by Theorem \ref{thm:GCRSI-FRBL}. 
Combining Proposition \ref{prop:partial-Gaussian-saturation} and Corollary \ref{cor:GenGaussian}, we deduce that $C = 0$, and hence $\Cov(h_2 \gamma) =\mathcal{C} = \Id_n$. By Proposition \ref{prop:cov} this means that $h_2$ must be identically equal to some constant $c > 0$. 
Therefore
\[
f_1(x_1) f_2(x_2) \leq c \cdot h(x_1,x_2)  \;\;\; \forall (x_1,x_2) \in \R^n \times \R^n ,
\]
and integrating this on the diagonal $(x_1,x_2) = (y,y)$ with respect to $\gamma$, we obtain
\[
\int_{\R^n} f_1 f_2 d\gamma \leq c \int_{\R^n} h(y,y) d\gamma = \int_{\R^n} f_1 d\gamma \int_{\R^n} f_2 d\gamma ,
\]
where we used (\ref{eq:GCRSI-FRBL-equality}) again in the last transition.  On the other hand, by the GCI for $f_1,f_2 \in \F_{LC}^{(o)}(\gamma)$ established in \cite{NakamuraTsuji-GCIForCentered}, we have
\[
\int_{\R^n} f_1 f_2 d\gamma \geq \int_{\R^n} f_1 d\gamma \int_{\R^n} f_2 d\gamma ,
\]
and so we must have equality. By the characterization of the equality case in \cite[Theorem 5.5]{NakamuraTsuji-GCIForCentered}, we must have
$f_1(x) = \bar f_1(P_E x)$ and $f_2(x) = \bar f_2(P_{E^{\perp}} x)$ for some linear subspace $E \subset \R^n$ and almost every $x \in \R^n$. 
Finally, since $f_1(y) f_2(y) \leq c \cdot h(y,y)$ and their integrals with respect to $\gamma$ coincide, it follows that these expressions must coincide for almost-every $y \in \R^n$. 
This concludes the proof. 
\end{proof}

\begin{proof}[Proof of Theorem \ref{thm:GCMPI-equality}] 
Recall that in this case $\FR_{LC} = \FR_{LC}^{(\G)} = 1$ by Theorem \ref{thm:GCMPI-FRBL}. 
Combining Proposition \ref{prop:partial-Gaussian-saturation} and the equality cases of Corollary \ref{cor:GenGaussianGCMPI}, we first handle the case when $\max(a^2,b^2) < 1$. In that case, Corollary \ref{cor:GenGaussianGCMPI} implies that $A_1 = A_2 = \bar B = C = 0$, and hence 
$\Cov(f_i \gamma) = \mathcal{A}_i = \Id_n$, $\Cov(y \mapsto h_1(y,y) \gamma(y)) = \Id_n$ and $\Cov(h_2 \gamma) =\mathcal{C} = \Id_n$. By Proposition \ref{prop:cov} this means that $f_1$, $f_2$, $y \mapsto h_1(y,y)$ and $h_2$ must be identically equal to some positive constants $p_1$, $p_2$, $q_1$ and $q_2$, respectively. By (\ref{eq:GCMPI-FRBL-equality}) we have $p_1 p_2 = q_1 q_2$, and (\ref{eq:GCMPI-FRBL-equality-assumption}) implies that $h_1(x_1,x_2) \geq \frac{p_1 p_2}{q_2} = q_1 > 0$ for all $(x_1,x_2) \in \R^n \times \R^n$. Since $h_1$ is log-concave, it follows that if must be constant (equal to $q_1$). 

When $a^2 < 1 = b^2$, Corollary \ref{cor:GenGaussianGCMPI} implies that $A_2 = C = 0$. Therefore $\Cov(f_2 \gamma) = \mathcal{A}_2 = \Id_n$ and $\Cov(h_2 \gamma) =\mathcal{C} = \Id_n$, and so Proposition \ref{prop:cov} implies that $f_2$ and $h_2$ must be identically equal to some positive constants $p_2$ and $q_2$, respectively. Inspecting (\ref{eq:GCMPI-FRBL-equality-assumption}) on $(x_1,x_2) = (y, a y)$, we know that $f_1(y) \leq \frac{q_2}{p_2} h_1(y,y)$ for all $y \in \R^n$. But since we have equality in (\ref{eq:GCMPI-FRBL-equality}), their integrals with respect to $\gamma$ are equal, and so they must coincide for almost every $y \in \R^n$. Exchanging the roles of $a,b$, the proof is complete. 
\end{proof}

\section{The general problem} \label{sec:unify}

The results of the previous sections raise the following natural question: for which $\alpha,\beta, a, b \in \R \setminus \{0\}$ does it hold that for all convex sets $K,L \subset \R^n$ with non-empty interior and Gaussian barycenters at the origin, one has
\[
\gamma(K) \gamma(L) \leq \gamma(\alpha K \cap \beta L) \gamma(a K + b L) \; ? 
\]
Since $\gamma$ is invariant under reflection, by switching between $K,-K$ and $L,-L$ if necessary, this is equivalent to requiring
\begin{equation} \label{eq:unify-preface}
\gamma(K) \gamma(L) \leq \gamma(\abs{\alpha} K \cap \abs{\beta} L) \gamma(\abs{a} K + \sigma \abs{b} L) ~,~ \sigma = \sgn(\alpha \beta a b) . 
\end{equation}
This reduces the question to the case that $\alpha,\beta,a,b > 0$ and $\sigma \in \{-1,+1\}$, so we proceed under these assumptions. 

By using $K$ or $L$ equal to $\R^n$ (or an approximation thereof), it is clear that necessarily $\alpha,\beta \geq 1$. 
By applying (\ref{eq:unify-preface}) to $\lambda K,\lambda L$ and taking the limit as $\lambda \rightarrow 0$, we must have
\[
\abs{K} \abs{L} \leq \abs{ \alpha K \cap \beta L } \abs{ a K + \sigma b L } ,
\]
for all bounded convex $K,L \subset \R^n$ with non-empty interior and Lebesgue barycenters at the origin. Equivalently, replacing $K$ and $L$ by $\beta K$ and $\alpha L$, we must have
\[
\abs{K} \abs{L} \leq \abs{K \cap L} \abs{ a \beta K + \sigma b \alpha L} ,
\]
and so by taking $L$ to be a small ball centered at the origin, we see that necessarily $a \beta, b \alpha \geq 1$. Consequently, the best we could hope for is to have $\alpha = \max(1/b,1)$ and $\beta = \max(1/a,1)$. This leads to the following:

\begin{problem} \label{prob:problem}
For each $\sigma \in \{-1,+1\}$, characterize those $a,b > 0$ for which
\begin{equation} \label{eq:problem}
\gamma^n(K) \gamma^n(L) \leq \gamma^n(\max(\frac{1}{b},1) K \cap \max(\frac{1}{a},1) L) \gamma^n(a K + \sigma b L) ,
\end{equation}
for all $n \geq 1$ and all convex $K,L \subset \R^n$ with non-empty interior and Gaussian barycenters at the origin. 
\end{problem}

\begin{lemma} \label{lem:a+b}
If (\ref{eq:problem}) holds for all $K,L$ as above then necessarily $a+b \geq 1$. 
\end{lemma}
\begin{proof}
By considering cylinders we may reduce to the case $n=1$. Taking $K = L = [-R,R]$, (\ref{eq:problem}) implies
\[
\gamma^1([-R,R])^2 \leq \gamma^1((a+b)[-R,R]) . 
\]
Denoting $\Phi_0(R) = \gamma^1([R,\infty))$, this implies
\[
\Phi_0((a+b) R) \leq 2 \Phi_0(R) . 
\]
Letting $R \rightarrow \infty$ and using the standard Gaussian tail estimates
\[
\frac{R}{R^2+1} \gamma^1(R) \leq \Phi_0(R) \leq \frac{1}{R} \gamma^1(R) ,
\]
the assertion easily follows. 
\end{proof}

Note that the necessary conditions $\alpha,\beta, a+b \geq 1$ already imply that (\ref{eq:unify-preface}) cannot hold
when $a=b = 1/\alpha = 1/\beta = \lambda$ for $\lambda < 1/2$ or $\lambda > 1$. On the other hand, Theorem \ref{thm:intro-GCMPI} verifies that (\ref{eq:unify-preface}) holds for $\sigma = +1$ and $\lambda \in [1/2,1]$. Together, this verifies Corollary \ref{cor:intro-lambda}. 

\subsection{Limitations of the forward-reverse Gaussian saturation} \label{subsec:counterexamples}

In view of Lemma \ref{lem:a+b}, it may be reasonable to expect that (\ref{eq:problem}) should hold for all $a+b \geq 1$. 
What is clear is that the Gaussian saturation fails for parts of this domain. In this subsection, we provide some counterexamples to that effect. 

\begin{itemize} 
\item $\sigma=+1$, $a,b \in (0,1)$, $a+b=1$. In this case, the Gaussian saturation cannot yield (\ref{eq:problem}) unless $a=b=1/2$ (the latter case was verified in Proposition \ref{prop:GaussianGCMPI}). Indeed, we have
    \[
    \gamma(K)\gamma(L) \le \FR_{LC}^{(\G)}(a,b) \cdot \gamma \brac{\frac{1}{b} K\cap \frac{1}{a} L} \gamma(a K + b L) ,
    \]
    where the Gaussian saturation constant $\FR_{LC}^{(\G)}(a,b) = \FR_{LC}^{(\G)}$ is given by
    \begin{equation} \label{eq:unify-temp}
   (\FR^{(\G)}_{LC})^2 := \sup     \frac{{\rm det}\, ( {\rm Id}_n + \bar B ) {\rm det}\, ( {\rm Id}_n + C ) }{ {\rm det}\, ({\rm Id}_n +A_1){\rm det}\, ({\rm Id}_n +A_2) } ,    \end{equation}
    and the supremum is over $A_1,A_2, C \in \Sym_{\geq 0}(n)$ and $B \in \Sym_{\geq 0}(2n)$ such that
    \[ 
    \begin{pmatrix}
        A_1 & 0 \\ 0 & A_2
    \end{pmatrix}
    \ge 
    \begin{pmatrix} \frac{1}{b^2} B_1 & \frac{1}{ab} B_2 \\ \frac{1}{ab} B_2^* & \frac{1}{a^2} B_4 \end{pmatrix} + \begin{pmatrix}
        a^2 C & abC \\ ab C & b^2 C 
    \end{pmatrix} .
    \]
    By testing cylindrical sets $K,L$, it is enough to check what happens in dimension $n=1$. For any $r,s \geq 0$, the following selection satisfies the requirement:
\[
A_1= \frac{1}{b^2} s r ,\;  A_2 = \frac{1}{a^2} s , \; C = \frac{1}{a^2b^2} s \frac{r}{r+1} ,\; 
B = s \begin{pmatrix} \frac{r^2}{r+1} & -\frac{r}{r+1}\\ -\frac{r}{r+1} & \frac{1}{r+1} \end{pmatrix} \geq 0 .
\]
Plugging this into (\ref{eq:unify-temp}), a long computation verifies that
\[
     (\FR^{(\G)}_{LC})^2 - 1 \geq
    \frac{ 
    s
    \bigg((r+1)\phi(r)
    -4s r^2\bigg) }{(b^2 +sr)  (a^2  +s) (r+1)^2} ,
\]
where 
$$
\phi(r):= a^2(b^2-1)r^2 + (1-2a^2b^2 -(a^2+b^2))r + (a^2-1)b^2. $$
If $r > 0$ is such that $\phi(r)\ge0$ then we can choose 
$$
s(r) = \frac{\phi(r)(r+1)}{8r^2} \geq 0 ,
$$
to have
\begin{equation} \label{eq:complicated}
    (\FR^{(\G)}_{LC})^2 - 1 \geq
    \frac{ 
    4 r^2 s(r)^2 }{(b^2 +s(r) r)  (a^2  +s(r)) (r+1)^2}.
\end{equation}
Selecting
$$
r_a := \frac{(1-a)(1-a+a^2)}{a^2(2-a)} > 0 , 
$$
another long computation shows that
\[
\phi(r_a) = \frac{(1-a)^2(2a-1)^2}{a(2-a)} \geq 0. 
\]
Thus, for $a \in (0,1)$, $\phi(r_a)=0$ if and only if $a=1/2$. 
Consequently, if $a\neq 1/2$ then $\phi(r_a)>0$ and hence $s(r_a) > 0$, and so (\ref{eq:complicated}) verifies that $(\FR^{(\G)}_{LC})^2 > 1$. 
\item $\sigma=-1$, $a,b \in (0,1]$. In this case, the Gaussian saturation cannot yield (\ref{eq:problem}) unless $a^2 + b^2 = 1$ as in the Milman--Pajor inequality (\ref{eq:intro-MilmanPajor}). Indeed, we have
    \[
    \gamma(K)\gamma(L) \le \FR_{LC}^{(\G)}(a,b) \cdot \gamma \brac{\frac{1}{b} K\cap \frac{1}{a} L} \gamma(a K- b L) ,
    \]
    where the Gaussian saturation constant $\FR_{LC}^{(\G)}(a,b) = \FR_{LC}^{(\G)}$ is given by (\ref{eq:unify-temp}) and the supremum is over $A_1,A_2, C \in \Sym_{\geq 0}(n)$ and $B \in \Sym_{\geq 0}(2n)$ such that
    \[ 
    \begin{pmatrix}
        A_1 & 0 \\ 0 & A_2
    \end{pmatrix}
    \ge 
    \begin{pmatrix} \frac{1}{b^2} B_1 & \frac{1}{ab} B_2 \\ \frac{1}{ab} B_2^* & \frac{1}{a^2} B_4 \end{pmatrix} + \begin{pmatrix}
        a^2 C & - abC \\ -ab C & b^2 C 
    \end{pmatrix} .
    \]
    By testing cylindrical sets $K,L$, it is enough to check what happens in dimension $n=1$. For any $z > 0$, the following selection satisfies the requirement:
\[
A_1=A_2 = z , \; C = \frac{z}{a^2 + b^2} ,\; 
B = \frac{z}{a^2+b^2} \begin{pmatrix} b^4 & a^2 b^2 \\ a^2 b^2 & a^4  \end{pmatrix} \geq 0 .
\]
Since $\bar B = (a^2+b^2) z$, we see that
\[
(\FR^{(\G)}_{LC})^2(a,b)  \geq  \sup_{z>0} \frac{(1 + (a^2+b^2) z)(1 + \frac{1}{a^2+b^2} z)}{(1+z)^2}=  \sup_{z>0} \frac{ 1+z^2 + (a^2+b^2 + \frac{1}{a^2 + b^2}) z }{ 1+z^2 + 2z }.
\]
In particular, by the arithmetic-geometric means inequality,
\[
1 \geq \FR^{(\G)}_{LC}(a,b) \quad \Rightarrow \quad a^2+b^2 + \frac{1}{a^2 + b^2} \leq 2 \quad \Rightarrow \quad  a^2+b^2 = 1 . 
\]
It is not hard to check that the reverse implication is also true, but as this only recover the original Milman--Pajor inequality (\ref{eq:intro-MilmanPajor}), we refrain from doing this here. 
\end{itemize}

\subsection{A unified formulation}

Since the ``classical" case when $\sigma=-1$  fails to satisfy the Gaussian saturation sufficient condition required by our forward-reverse Brascamp--Lieb reduction, besides in the known (Milman--Pajor) case $a^2+b^2=1$, we only consider the ``conjugate" case $\sigma = +1$ in this work. 
As we've already seen in the previous sections, the reduction to the Gaussian saturation question yields a sharp constant for a wide range of $(a,b)$ in the conjugate case. Below, we present a unified formulation which partially resolves Problem \ref{prob:problem} for $\sigma=+1$.

\begin{thm} \label{thm:unify}
Let $a,b \geq 0$. 
\begin{enumerate}
\item
If 
\[
3 \min(a^2,b^2) + \max(a^2,b^2) \geq 1 ,
\]
then for any convex $K,L \subset \R^n$ with non-empty interior and Gaussian barycenters at the origin, 
it holds that
\begin{equation} \label{eq:unified}
\gamma(K) \gamma(L) \leq \gamma(\max(\frac{1}{b},1) K \cap \max(\frac{1}{a},1) L) \gamma(a K + b L) ,
\end{equation}
with equality as follows:
\begin{enumerate}
\item When in addition $0 < a < 1 \leq b$ if and only if $L = \R^n$.
\item When in addition $0 < b < 1 \leq a$ if and only if $K = \R^n$. 
\item When in addition $\max(a,b) < 1$ if and only if $K = L = \R^n$. 
\item When in addition $\min(a,b) \geq 1$ if and only if (\ref{eq:intro-equality}) holds. 
\end{enumerate}
\item
When
\[
 a+b < 1,
 \]
 then (\ref{eq:unified}) is false for some $K,L$ as above. 
\end{enumerate}
\end{thm}
\begin{proof}
The case when $\min(a,b) = 0$ is trivial, the case when $a,b \geq 1$ (and in particular $a+b > 1$)  has already been treated in Theorem \ref{thm:intro-ab}, and the case when $0 < a,b \leq 1$ has already been treated in Theorem \ref{thm:intro-GCMPI}. As for the remaining case $0 < a \leq 1 \leq b$ (and similarly when the roles of $a$ and $b$ are reversed), the claim is that
\[
\gamma(K) \gamma(L) \leq \gamma( K \cap \frac{1}{a} L) \gamma(a K + b L) .
\]
But since $L$ contains the origin, $b L \supset L$ and so this inequality and its equality conditions follow from the case when $b=1$ which has already been treated. This concludes the proof.
\end{proof}

This leaves a small range of values of $(a,b)$, namely
\[
 \{a,b > 0, ~ a+b \geq 1 , ~ 3 \min(a^2,b^2) + \max(a^2,b^2) < 1 \},
\]
 for which the validity of (\ref{eq:unified}) remains undecided; see Figure \ref{fig:diagram}. 

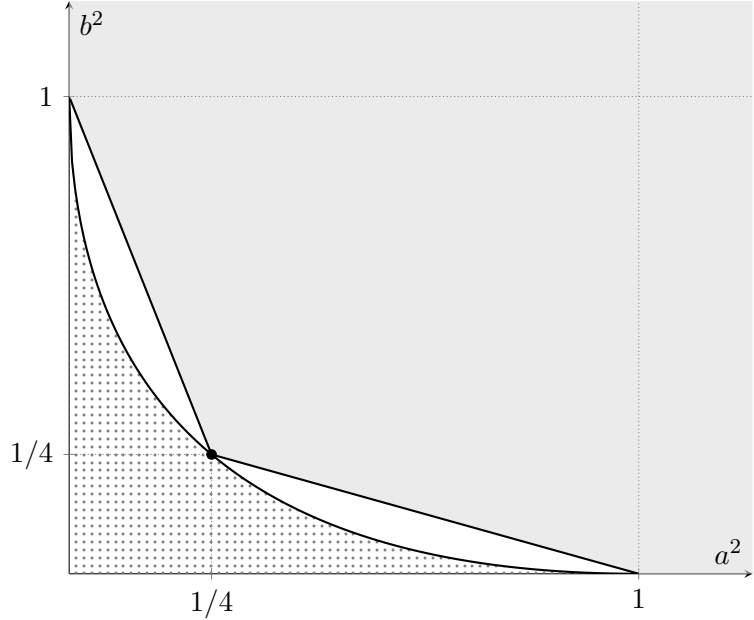
\begin{figure}[H] 
  \centering
  \begin{tikzpicture}
    \begin{axis}[
      width=0.7\linewidth,
      axis lines=middle,
      xmin=0, xmax=1.2,
      ymin=0, ymax=1.2,
      xlabel={$a^2$}, ylabel={$b^2$},
      xtick={0,0.25,1}, ytick={0,0.25,1},
      xticklabels={$0$,$1/4$,$1$},
      yticklabels={$0$,$1/4$,$1$},
      samples=200,
      domain=0:1,
      clip=true,
    ]

        \addplot[
      draw=none,
      fill=gray!25,
      fill opacity=0.6,
    ] coordinates {
      (0,1) (0,1.2) (1.2,1.2) (1.2,0) (1,0) (0.25,0.25) (0,1)
    } \closedcycle;

        \path[name path=xaxis] (axis cs:0,0) -- (axis cs:1,0);
    \addplot[draw=none,name path=rootsum,domain=0:1] {(1 - sqrt(x))^2};
    \addplot[
      draw=none,
      pattern=dots,
      pattern color=gray,
    ] fill between[of=rootsum and xaxis];

        \addplot[thick,black,domain=0:0.333333] {((1 - 3*x) > (1 - sqrt(x))^2) ? (1 - 3*x) : nan};

        \addplot[thick,black,domain=0:1] {(((1 - x)/3) > (1 - sqrt(x))^2) ? ((1 - x)/3) : nan};

        \addplot[thick,black,domain=0:1] {(1 - sqrt(x))^2};

        \addplot[only marks,mark=*,mark size=1.8pt] coordinates {(0.25,0.25)};

        \addplot[gray,densely dotted] coordinates {(0.25,0) (0.25,0.25)};
    \addplot[gray,densely dotted] coordinates {(0,0.25) (0.25,0.25)};

        \addplot[gray,densely dotted] coordinates {(1,0) (1,1.2)};
    \addplot[gray,densely dotted] coordinates {(0,1) (1.2,1)};

    \end{axis}
  \end{tikzpicture}
  \caption{\label{fig:diagram}The various regions of $(a^2,b^2)$ covered by Theorem \ref{thm:unify}: (\ref{eq:unified}) holds inside the grey region, is violated inside the dotted region, and remains open in the white region. }
\end{figure}

\section{Functional formulation of geometric inequalities} \label{sec:functional}

In this section, we derive an equivalent functional formulation of the Gaussian (conjugate) Rogers--Shephard inequality for convex sets, establishing Theorem
\ref{thm:intro-functional}, as well as several other equivalent functional formulations of the geometric inequalities we obtain in this work.

Recall that we define
\[
\maxo(p,q) := \begin{cases} \max(p,q) & p,q > 0 \\ 0 & pq = 0 \end{cases} ,
\]
and that for two functions $u,v : \R^n \rightarrow \R_+$, we denote
\[
u \square v (z) := \sup_{z = x + y} \min(u(x),v(y)) .
\]
We also denote for $c \in \R \setminus \{0\}$:
\[
u_c := u(\cdot / c) . 
\]
\begin{rem}
Whenever $u,v$ are Borel measurable, $u \square v$ is analytic and hence Lebesgue measurable. 
\end{rem}

\begin{proposition} \label{prop:functional}
Let $\C$ denote a family of Borel subsets of $\R^n$, and let $\mu_i$, $i=1,\ldots,4$, denote four measures on $\R^n$ equipped with the Lebesgue $\sigma$-algebra. Then the following statements are equivalent for any fixed $\alpha,\beta,a,b \in \R \setminus \{0\}$: 
\begin{enumerate}
\item \label{it:functional1}
For all $K,L \in \C$:
\[
\mu_1(K) \mu_2(L) \leq \mu_3(\alpha K \cap \beta L) \mu_4(a K+ b L) . 
\]
\item \label{it:functional2}
For all functions $f,h : \R^n \rightarrow \R_+$ whose super level sets are in $\mathcal{C}$, we have:
\[
\int f d\mu_1 \int h  d\mu_2 \leq \int \maxo(f_\alpha ,h_\beta) d\mu_3 \int f_a \square h_b  d\mu_4 .
\]
\end{enumerate}
\end{proposition}

For the proof, we will require the following elementary version of the four functions theorem, originally proved by Ahlswede and Daykin  on a discrete lattice \cite{AhlswedeDaykin-FourFunctionsThm}, and extended to the continuous setting in \cite{BattyBollmann-FourFunctionsThm} (see also \cite{HKS-FourFunctionsThm} for remarkable extensions).  

\begin{lemma} \label{lem:four-funcs}
Let $a,b,c,d : (0,\infty) \rightarrow \R_+$ denote four measurable functions satisfying:
\[
a(t) b(s) \leq c(t \vee s) d(t \wedge s) \;\;\; \forall t, s  > 0 . 
\]
Then:
\[
\int_{0}^\infty a(t) dt \int_{0}^\infty b(s) ds \leq \int_{0}^\infty c(p) dp \int_{0}^\infty d(q) dq .
\]
\end{lemma}
\begin{proof}
 Define $\tilde h(r_0) := h(e^{r_0}) e^{r_0}$ for all $h \in \{ a,b,c,d \}$ and $r_0 \in \R$. Since $t_0 + s_0 = t_0 \wedge s_0 + t_0 \vee s_0$ we have:
\[
\tilde a(t_0) \tilde b(s_0) \leq \tilde c(t_0 \wedge s_0) \tilde d(t_0 \vee s_0) \;\;\; \forall t_0,s_0 \in \R ,
\]
and hence by the four functions theorem on $\R$ \cite{BattyBollmann-FourFunctionsThm,HKS-FourFunctionsThm} we obtain:
\[
\int_{-\infty}^{\infty} \tilde a(t_0) dt_0 \int_{-\infty}^{\infty} \tilde b(s_0) ds_0 \leq \int_{-\infty}^\infty \tilde c(p_0) dp_0 \int_{-\infty}^\infty \tilde d(q_0) dq_0 .
\]
This is precisely the asserted inequality after applying the change of variables $r = e^{r_0}$ in all four integrals.  
\end{proof}

\begin{proof}[Proof of Proposition \ref{prop:functional}]
Clearly (\ref{it:functional2}) implies (\ref{it:functional1}) by applying it to $f = \1_K$ and $h = \1_L$, since $\maxo((\1_K)_\alpha , (\1_L)_\beta) = \1_{\alpha K \cap \beta L}$ and $(\1_K)_a \square (\1_L)_b = \1_{aK + b L}$. To see the other direction, define the following four functions on $(0,\infty)$:
\[
a(t) := \mu_1 \{ x : f(x) \geq t \} ~,~ b(s) := \mu_2 \{ y : h(y) \geq s \} ,
\]
and
\[
c(p) := \mu_3 \{ x : \maxo(f_\alpha(x),h_\beta(x)) \geq p \} ~,~ d(q) := \mu_4 \{ z : f_a \square  h_b(z) \geq q \} .
\]
If $K := \{ x : f(x) \geq t \}$ and $L := \{ y : h(y) \geq s \} $ for $t,s > 0$, note that:
\begin{align*}
\alpha K \cap \beta L & \subset \{ x : \maxo(f_\alpha(x),h_\beta(x)) \geq t \vee s \} ~,~ \\
 aK + bL & \subset \{ x + y : \min(f_a(x), h_b(y)) \geq t \wedge s \} \subset \{ z : f_a\square h_b (z) \geq t \wedge s \} . 
\end{align*}
Consequently, our assumption exactly implies that:
\[
a(t) b(s) \leq c(t \vee s) d(t \wedge s) \;\;\; \forall t, s  > 0 . 
\]
By the four functions Lemma \ref{lem:four-funcs}, it follows that:
\[
\int_{0}^\infty a(t) dt \int_{0}^\infty b(s) ds \leq \int_{0}^\infty c(p) dp \int_{0}^\infty d(q) dq ,
\]
concluding the proof by the tail formula. 
\end{proof}

Applying Proposition \ref{prop:functional} to the GCRSI (\ref{eq:intro-GCRSI}), Theorem \ref{thm:intro-functional} immediately follows. More generally, applying it to Theorems \ref{thm:intro-ab} and \ref{thm:unify}, we have:
\begin{cor} \label{cor:functional-Gaussian} 
Let $\alpha,\beta,a,b \in \R \setminus \{0\}$. 
For all quasi-concave Borel functions $f,h : \R^n \rightarrow \R_+$ whose super level sets have Gaussian barycenters at the origin, we have
\[
\int f d\gamma  \int h d\gamma\leq \int \maxo(f_\alpha,h_\beta) d\gamma \int f_a \square h_b d\gamma ,
\]
in any of the cases below:
\begin{enumerate}
\item $\alpha = \beta = 1$ and $\abs{a},\abs{b},\abs{a+b} \geq 1$.
\item $\alpha = \max(1/b,1)$, $\beta = \max(1/a,1)$, $a,b > 0$ with $3 \min(a^2,b^2) + \max(a^2,b^2) \geq 1$.  
\end{enumerate}
\end{cor} 

In view of Remark \ref{rem:intro-scaling}, when passing from $\gamma$ to the Lebesgue measure in the scaling limit, the only two interesting cases are
obtained by applying Proposition \ref{prop:functional} to the RSSI (\ref{eq:intro-RSSI}) and the CRSSI (\ref{eq:intro-CRSSI}) inequalities. 

\begin{cor} \label{cor:functional-Lebesgue} 
For all quasi-concave Borel functions $f,h : \R^n \rightarrow \R_+$ whose super level sets have Lebesgue barycenters at the origin, we have for both choices of $\pm \in \{+,-\}$:
\[
\int f dx  \int h dx \leq \int \maxo(f,h) dx \int f \square h_\pm dx ,
\]
where $h_{\pm} : \R^n \rightarrow \R_+$ is defined as $h_{\pm}(x) = h(\pm x)$. 
\end{cor}

\bibliographystyle{plain}
\bibliography{../../../ConvexBib}

\def\cprime{$'$} \def\textasciitilde{$\sim$}
\begin{thebibliography}{10}

\bibitem{AhlswedeDaykin-FourFunctionsThm}
R.~Ahlswede and D.~E. Daykin.
\newblock An inequality for the weights of two families of sets, their unions
  and intersections.
\newblock {\em Z. Wahrsch. Verw. Gebiete}, 43(3):183--185, 1978.

\bibitem{ACS-RefinedKhatriSidak}
R.~Assouline, A.~Chor, and S.~Sadovsky.
\newblock A refinement of the \v{S}id\'ak-{K}hatri inequality and a strong
  {G}aussian correlation conjecture.
\newblock arXiv:2407.15684, 2024.

\bibitem{BGL-Book}
D.~Bakry, I.~Gentil, and M.~Ledoux.
\newblock {\em Analysis and geometry of {M}arkov diffusion operators}, volume
  348 of {\em Grundlehren der Mathematischen Wissenschaften [Fundamental
  Principles of Mathematical Sciences]}.
\newblock Springer, Cham, 2014.

\bibitem{Barthe-ReverseBL-CRAS}
F.~Barthe.
\newblock In\'egalit\'es de {B}rascamp-{L}ieb et convexit\'e.
\newblock {\em C. R. Acad. Sci. Paris S\'er. I Math.}, 324(8):885--888, 1997.

\bibitem{Barthe-ReverseBL}
F.~Barthe.
\newblock On a reverse form of the {B}rascamp-{L}ieb inequality.
\newblock {\em Invent. Math.}, 134(2):335--361, 1998.

\bibitem{BartheCordero-InverseBLviaSemiGroup}
F.~Barthe and D.~Cordero-Erausquin.
\newblock Inverse {B}rascamp-{L}ieb inequalities along the heat equation.
\newblock In {\em Geometric aspects of functional analysis}, volume 1850 of
  {\em Lecture Notes in Math.}, pages 65--71. Springer, Berlin, 2004.

\bibitem{BartheHuet}
F.~Barthe and N.~Huet.
\newblock On {G}aussian {B}runn-{M}inkowski inequalities.
\newblock {\em Studia Math.}, 191(3):283--304, 2009.

\bibitem{BartheWolff-InverseBrascampLieb}
F.~Barthe and P.~Wolff.
\newblock Positive {G}aussian kernels also have {G}aussian minimizers.
\newblock {\em Mem. Amer. Math. Soc.}, 276(1359):v+90, 2022.

\bibitem{BattyBollmann-FourFunctionsThm}
C.~J.~K. Batty and H.~W. Bollmann.
\newblock Generalised {H}olley-{P}reston inequalities on measure spaces and
  their products.
\newblock {\em Z. Wahrsch. Verw. Gebiete}, 53(2):157--173, 1980.

\bibitem{BCCT-BrascampLieb}
J.~Bennett, A.~Carbery, M.~Christ, and T.~Tao.
\newblock The {B}rascamp-{L}ieb inequalities: finiteness, structure and
  extremals.
\newblock {\em Geom. Funct. Anal.}, 17(5):1343--1415, 2008.

\bibitem{BhattacharyaRao-Book2010}
R.~N. Bhattacharya and R.~R. Rao.
\newblock {\em Normal approximation and asymptotic expansions}, volume~64 of
  {\em Classics in Applied Mathematics}.
\newblock Society for Industrial and Applied Mathematics (SIAM), Philadelphia,
  PA, corrected edition, 2010.

\bibitem{BillingsleyConvergenceBook-2ndEd}
P.~Billingsley.
\newblock {\em Convergence of probability measures}.
\newblock Wiley Series in Probability and Statistics: Probability and
  Statistics. John Wiley \& Sons, Inc., New York, second edition, 1999.
\newblock A Wiley-Interscience Publication.

\bibitem{Bobkov-LocalLimitTheoremsInOrliczSpace}
S.~G. Bobkov.
\newblock Local limit theorems for densities in {O}rlicz spaces.
\newblock {\em J. Math. Sci. (N.Y.)}, 242(1):52--68, 2019.

\bibitem{BrascampLieb-YoungInq}
H.~J. Brascamp and E.~H. Lieb.
\newblock Best constants in {Y}oung's inequality, its converse, and its
  generalization to more than three functions.
\newblock {\em Advances in Math.}, 20(2):151--173, 1976.

\bibitem{BrascampLiebPLandLambda1}
H.~J. Brascamp and E.~H. Lieb.
\newblock On extensions of the {B}runn-{M}inkowski and {P}r\'ekopa-{L}eindler
  theorems, including inequalities for log concave functions, and with an
  application to the diffusion equation.
\newblock {\em J. Func. Anal.}, 22(4):366--389, 1976.

\bibitem{GreekBook}
S.~Brazitikos, A.~Giannopoulos, P.~Valettas, and B.-H. Vritsiou.
\newblock {\em Geometry of Isotropic Convex Bodies}, volume 196 of {\em
  Mathematical Surveys and Monographs}.
\newblock Amer. Math. Soc., 2014.

\bibitem{CarlenLiebLoss-EntropyOnSn}
E.~A. Carlen, E.~H. Lieb, and M.~Loss.
\newblock A sharp analog of {Y}oung's inequality on {$S^N$} and related entropy
  inequalities.
\newblock {\em J. Geom. Anal.}, 14(3):487--520, 2004.

\bibitem{ChenLou-CharacterizationOfGaussianViaPoincare}
L.~H.~Y. Chen and J.~H. Lou.
\newblock Characterization of probability distributions by {P}oincar\'e-type
  inequalities.
\newblock {\em Ann. Inst. H. Poincar\'e{} Probab. Statist.}, 23(1):91--110,
  1987.

\bibitem{CourtadeLiu-BrascampLieb}
T.~A. Courtade and J.~Liu.
\newblock Euclidean forward-reverse {B}rascamp-{L}ieb inequalities: finiteness,
  structure, and extremals.
\newblock {\em J. Geom. Anal.}, 31(4):3300--3350, 2021.

\bibitem{CourtadeWang-BlaschkeSantalo}
T.~A. Courtade and E.~Wang.
\newblock Generalized {B}laschke--{S}antal\'o-type inequalities, without
  symmetry restrictions.
\newblock Manuscript, arXiv:2509.08998, 2025.

\bibitem{HKS-FourFunctionsThm}
D.~Halikias, B.~Klartag, and B.~A. Slomka.
\newblock Discrete variants of {B}runn-{M}inkowski type inequalities.
\newblock {\em Ann. Fac. Sci. Toulouse Math. (6)}, 30(2):267--279, 2021.

\bibitem{HJS-ExtremalPropertyOfGaussians}
E.~Hillion, O.~Johnson, and A.~Saumard.
\newblock An extremal property of the normal distribution, with a discrete
  analog.
\newblock {\em Statist. Probab. Lett.}, 145:181--186, 2019.

\bibitem{HornJohnson-MatrixAnalysis}
R.~A. Horn and C.~R. Johnson.
\newblock {\em Matrix analysis}.
\newblock Cambridge University Press, Cambridge, second edition, 2013.

\bibitem{Khatri}
C.~G. Khatri.
\newblock On certain inequalities for normal distributions and their
  applications to simultaneous confidence bounds.
\newblock {\em Ann. Math. Statist.}, 38:1853--1867, 1967.

\bibitem{KlartagLehec-SlicingSolved}
B.~Klartag and J.~Lehec.
\newblock Affirmative resolution of {B}ourgain's slicing problem using {G}uan's
  bound.
\newblock {\em Geom. Funct. Anal.}, 35(4):1147--1168, 2025.

\bibitem{LatalaMatlak-GaussianCorrelation}
R.~Lata{\l}a and D.~Matlak.
\newblock Royen's proof of the {G}aussian correlation inequality.
\newblock In {\em Geometric aspects of functional analysis}, volume 2169 of
  {\em Lecture Notes in Math.}, pages 265--275. Springer, Cham, 2017.

\bibitem{Lieb-MultiDimBL}
E.~H. Lieb.
\newblock Gaussian kernels have only {G}aussian maximizers.
\newblock {\em Invent. Math.}, 102(1):179--208, 1990.

\bibitem{LCCV-BrascampLieb}
J.~Liu, T.~A. Courtade, P.~W. Cuff, and S.~Verd\'u.
\newblock A forward-reverse {B}rascamp-{L}ieb inequality: entropic duality and
  {G}aussian optimality.
\newblock {\em Entropy}, 20(6):Paper No. 418, 32, 2018.

\bibitem{EMilman-GCI}
E.~Milman.
\newblock Gaussian correlation via inverse {B}rascamp--{L}ieb.
\newblock Probab. Theory Relat. Fields, 2025.
  https://doi.org/10.1007/s00440-025-01445-x.

\bibitem{MilmanPajor-NonSymmetric}
V.~D. Milman and A.~Pajor.
\newblock Entropy and asymptotic geometry of non-symmetric convex bodies.
\newblock {\em Adv. Math.}, 152(2):314--335, 2000.

\bibitem{NakamuraTsuji-GCIForCentered}
S.~Nakamura and H.~Tsuji.
\newblock The {G}aussian correlation inequality for centered convex sets and
  the case of equality.
\newblock arXiv:2504.04337, 2025.

\bibitem{NakamuraTsuji-InverseBrascampLieb}
S.~Nakamura and H.~Tsuji.
\newblock A generalized {L}egendre duality relation and {G}aussian saturation.
\newblock {\em Invent. Math.}, 243:607--655, 2026.

\bibitem{Pitt-GaussianCorrelationInPlane}
L.~D. Pitt.
\newblock A {G}aussian correlation inequality for symmetric convex sets.
\newblock {\em Ann. Probability}, 5(3):470--474, 1977.

\bibitem{RogersShephard-ConvexBodiesAssociated}
C.~A. Rogers and G.~C. Shephard.
\newblock Convex bodies associated with a given convex body.
\newblock {\em J. London Math. Soc.}, 33:270--281, 1958.

\bibitem{Royen-GaussianCorrelation}
T.~Royen.
\newblock A simple proof of the {G}aussian correlation conjecture extended to
  some multivariate gamma distributions.
\newblock {\em Far East J. Theor. Stat.}, 48(2):139--145, 2014.

\bibitem{SSZ-GaussianCorrelationConjecture}
G.~Schechtman, Th. Schlumprecht, and J.~Zinn.
\newblock On the {G}aussian measure of the intersection.
\newblock {\em Ann. Probab.}, 26(1):346--357, 1998.

\bibitem{Spingarn-RogersShephard}
J.~E. Spingarn.
\newblock An inequality for sections and projections of a convex set.
\newblock {\em Proc. Amer. Math. Soc.}, 118(4):1219--1224, 1993.

\bibitem{Tehranchi-RefinedGaussianCorrelation}
M.~R. Tehranchi.
\newblock Inequalities for the {G}aussian measure of convex sets.
\newblock {\em Electron. Commun. Probab.}, 22:Paper No. 51, 7, 2017.

\bibitem{Valdimarsson-GenCaffarelli}
S.~I. Valdimarsson.
\newblock On the {H}essian of the optimal transport potential.
\newblock {\em Ann. Sc. Norm. Super. Pisa Cl. Sci. (5)}, 6(3):441--456, 2007.

\bibitem{Sidak}
Z.~\v{S}id\'ak.
\newblock Rectangular confidence regions for the means of multivariate normal
  distributions.
\newblock {\em J. Amer. Statist. Assoc.}, 62:626--633, 1967.

\end{thebibliography}

\end{document}